\PassOptionsToPackage{dvipsnames}{xcolor}
\pdfoutput=1
\documentclass{article}

\usepackage[utf8]{inputenc}
\usepackage{amssymb}
\usepackage{amsfonts}
\usepackage[tbtags]{amsmath}
\usepackage{amsthm}
\usepackage{romanbar}
\usepackage{mathrsfs}
\usepackage{tikz}
\usepackage{float}
\usepackage[left=3.5cm,right=3.5cm,top=3cm,bottom=3cm]{geometry}
\usetikzlibrary{cd}
\usetikzlibrary{decorations.pathmorphing}
\usetikzlibrary{positioning}
\usetikzlibrary{patterns}
\usetikzlibrary{patterns.meta}
\usepackage[backend=biber,style=alphabetic,maxnames=6]{biblatex}
\usepackage{titlesec}
\usepackage[dvipsnames]{xcolor}
\usepackage{stmaryrd}
\usepackage{enumitem}
\usepackage{titletoc}
\usepackage{hyperref}
\usepackage{amsmath}
\allowdisplaybreaks

\newcommand\Hom[3][]{\mathrm{Hom}_{#1}(#2,#3)}
\newcommand\m[1]{\mathrm{mod}(#1)}

\newcommand\Db[1]{\mathrm{D}^{\mathrm{b}}(#1)}

\newcommand\id[1]{\mathrm{id}_{#1}}

\newcommand\lotimes[2]{\overset{\mbox{\tiny \bf{L}}}{\otimes}_{#1}#2}
\newcommand\per[1]{\mathrm{per}(#1)}
\newcommand\T[1]{\mathrm{Triv}(#1)}

\let\ens\mathbb
\newcommand\Z{\ens{Z}}

\tikzstyle{etiquette}=[minimum height=0.8cm, minimum width=1cm]

\newtheoremstyle{definition}
{\topsep}
{\topsep}
{\normalfont}
{0pt}
{\bfseries\fontfamily{bch}\selectfont}
{.}
{ }
{\thmname{#1}\thmnumber{ #2}\textnormal{\thmnote{ (#3)}}}

\newtheoremstyle{notation}{\topsep}{\topsep}{\normalfont}{0pt}{\bfseries\fontfamily{bch}\selectfont}{.}{ }{\thmname{#1}\thmnumber{ #2}\textnormal{\thmnote{ (#3)}}}

\newtheoremstyle{lemma}{\topsep}{\topsep}{\itshape}{0pt}{\bfseries\fontfamily{bch}\selectfont}{.}{ }{\thmname{#1}\thmnumber{ #2}\textnormal{\thmnote{ (#3)}}}

\newtheoremstyle{proposition}{\topsep}{\topsep}{\itshape}{0pt}{\bfseries\fontfamily{bch}\selectfont}{.}{ }{\thmname{#1}\thmnumber{ #2}\textnormal{\thmnote{ (#3)}}}

\newtheoremstyle{corollary}{\topsep}{\topsep}{\itshape}{0pt}{\bfseries\fontfamily{bch}\selectfont}{.}{ }{\thmname{#1}\thmnumber{ #2}\textnormal{\thmnote{ (#3)}}}

\newtheoremstyle{theorem}{\topsep}{\topsep}{\itshape}{0pt}{\bfseries\fontfamily{bch}\selectfont}{.}{ }{\thmname{#1}\thmnumber{ #2}\textnormal{\thmnote{ (#3)}}}

\newtheoremstyle{example}{\topsep}{\topsep}{\normalfont}{0pt}{\bfseries\fontfamily{bch}\selectfont}{.}{ }{\thmname{#1}\thmnumber{ #2}\textnormal{\thmnote{ (#3)}}}

\newtheoremstyle{remark}{\topsep}{\topsep}{\normalfont}{0pt}{\bfseries\fontfamily{bch}\selectfont}{.}{ }{\thmname{#1}\thmnumber{ #2}\textnormal{\thmnote{ (#3)}}}

\theoremstyle{definition}
\newtheorem{defn}{Definition}[section]

\theoremstyle{notation}

\theoremstyle{lemma}
\newtheorem{lem}[defn]{Lemma}

\theoremstyle{proposition}
\newtheorem{prop}[defn]{Proposition}

\theoremstyle{corollary}

\theoremstyle{theorem}
\newtheorem{thm}[defn]{Theorem}
\newtheorem*{thm*}{Theorem}

\theoremstyle{example}
\newtheorem{ex}[defn]{Example}

\theoremstyle{remark}
\newtheorem{rem}[defn]{Remark}

\titleformat{\section}[block]
{\centering\normalfont\Large\fontfamily{qhv}\selectfont}
{\thesection.}{0.5em}{}

\titleformat{\subsection}[block]
{\flushleft\normalfont\large\itshape\fontfamily{cmss}\selectfont}
{\thesubsection.}{0.5em}{}

\titlecontents{section}[0pt]{\bfseries\addvspace{0.1cm}}
{\thecontentslabel. \hspace{0.1mm}}
{}
{\hfill\contentspage}

\titlecontents{subsection}[15pt]{}
{\thecontentslabel. \hspace{0.1mm}}
{}
{\titlerule*[0.5pc]{.}\contentspage}

\title{\fontfamily{qhv}\selectfont Tilting mutations as generalized Kauer moves for (skew) Brauer graph algebras with multiplicity}
\author{\scshape Valentine Soto}
\date{}
\begin{document}

\maketitle

\medskip

\begin{center}
\begin{minipage}{0.75\textwidth}
\textbf{\fontfamily{bch}\selectfont Abstract.} Generalized Kauer moves are local moves of multiple edges in a Brauer graph that yield derived equivalences between Brauer graph algebras of multiplicity identically 1 \cite{Soto}. Moreover, these derived equivalences are given by a tilting mutation. The goal of this paper is to generalize this result first for Brauer graph algebras with arbitrary multiplicity and second for a generalization of Brauer graph algebras called skew Brauer graph algebras. In these contexts, we prove that the generalized Kauer moves induce derived equivalences via tilting mutations. We also show that skew Brauer graph algebras of multiplicity identically 1 can be seen as the trivial extension of skew gentle algebras.
\end{minipage}
\end{center}

\medskip

\tableofcontents

\medskip

\phantomsection
\section*{Introduction}
\addcontentsline{toc}{section}{Introduction}

\medskip

Tilting theory plays an important role in the representation theory of algebras as it yields derived equivalences \cite{Happel, Rickard}. Given a tilting object, one can construct others under some assumptions via the notion of mutation \cite{BGP,APR,RS,Okuyama}. This operation consists in replacing a direct summand of the given object using an approximation. Unfortunately, tilting mutation does not always exist. This problem has led to the introduction of silting mutations in triangulated categories which generalizes these tilting mutations \cite{AI}. The goal of this paper is to describe the silting mutation for a certain class of algebras called skew Brauer graph algebras, which generalizes the class of Brauer graph algebras.

Brauer graph algebras were first introduced in modular representation theory of finite groups by Donovan and Freislich \cite{DF}. They are finite dimensional algebras that are constructed thanks to the combinatorial data of a Brauer graph. A Brauer graph is a finite graph equipped with a multiplicity function on each vertex and with a cyclic ordering of the edges around each vertex. 
Moreover, Brauer graph algebras are symmetric special biserial algebras \cite{Schroll2}. Hence, the representation theory of Brauer graph algebras is well-known \cite{DF,WW,ES,Green,Roggenkamp}. We refer to the survey \cite{Schroll} for a detailed list of results on Brauer graph algebras. In the recent years, the study of derived equivalences of Brauer graph algebras has drawn some attention. Antipov and Zvonareva proved that this class of algebras is closed under derived equivalences \cite{AZ}. Moreover, Opper and Zvonareva described a complete numerical derived invariant for Brauer graph algebras \cite{OZ} (see also \cite{Zvonareva} for a survey on derived equivalences of Brauer graph algebras).

Silting mutation in Brauer graph algebras can be understood in terms of generalized Kauer moves which are particular moves of edges in the corresponding Brauer graph. These moves were first described by Kauer in the case of one edge \cite{Kauer}. Aihara proved that these moves of one edge give rise to an irreducible tilting mutation in the case of Brauer tree algebras (i.e. Brauer graph algebras whose underlying graph is a tree). These moves were then generalized in \cite{Soto} to describe all silting mutations in Brauer graph algebras of multiplicity identically one. Since Brauer graph algebras are symmetric \cite{Schroll2}, the silting mutation in these algebras are in fact tilting. Hence, these generalized Kauer moves yield derived equivalences of Brauer graph algebras. The first main result of this paper is the generalization of these results for Brauer graph algebras with arbitrary multiplicity.

\medskip

\begin{thm*}[Theorem \ref{thm:generalized Kauer moves with multiplicity}]
    Let $\Gamma=(\Gamma_{0},\Gamma_{1})$ be a Brauer graph and $B$ be its corresponding Brauer graph algebra. We assume that the least common multiple $\overline{m}$ of the multiplicities in $\Gamma$ is invertible in the field $k$. Then, for any subset of edges $E$ in $\Gamma$, the following holds.

    \begin{enumerate}[label=(\arabic*)]
    \item The Brauer graph algebra $B'$ obtained by a generalized Kauer move of $E$ in $\Gamma$ is derived equivalent to $B$.
    \item The silting mutation of $B$ over the projective $B$-module corresponding to the edges in $\Gamma_{1}\backslash E$ is tilting and its endomorphism algebra is isomorphic to $B'$.
    \end{enumerate}
\end{thm*}

\medskip

The idea of the proof is to construct a $\Z/\overline{m}\, \Z$-covering $\widetilde{\Gamma}$ of $\Gamma$ which is a Brauer graph of multiplicity identically one using \cite{Asashiba}. In this case, one can recover $B$ from the Brauer graph algebra $\widetilde{B}$ associated to $\widetilde{\Gamma}$ thanks to the construction of skew group algebras as defined in \cite{RR} (cf Proposition \ref{prop:Galois covering and skew group algebra}). Moreover, there is a commutativity between constructing this covering and applying the generalized Kauer moves. This can be summarized with the following commutative diagram

\smallskip

\begin{center}
    \begin{tikzcd}[column sep=5cm, row sep=2cm]
        \widetilde{\Gamma} \arrow{r}[above]{\mbox{generalized Kauer}}[below]{\mbox{ move of $\widetilde{H'}$}} \arrow{d}[description]{\mbox{$\Z/\overline{m}\,\Z$-covering}} &\mu^{+}_{\widetilde{H'}}(\widetilde{\Gamma}) \arrow{d}[description]{\mbox{$\Z/\overline{m}\,\Z$-covering}} \\
        \Gamma \arrow{r}[above]{\mbox{generalized Kauer}}[below]{\mbox{move of $H'$}}&\mu^{+}_{H'}(\Gamma)
    \end{tikzcd}
\end{center}

\smallskip

\noindent where $\widetilde{H'}$ is the set of half-edges in $\widetilde{\Gamma}$ consisting of all the lifts of the elements in $H'$.

In a second part of this paper, we generalize the previous result for skew Brauer graph algebras, which generalize the class of Brauer graph algebras. Their construction relies on the combinatorial data of a skew Brauer graph which is a Brauer graph with possible ``degenerate" edges, meaning that the corresponding half-edges are fixed by the pairing. The second main result of this paper describes the combinatorics of silting mutation for skew Brauer graph algebras. 

\medskip

\begin{thm*}[Theorem \ref{thm:generalized Kauer moves for skew Brauer graphs}]
    Let $\Gamma=(\Gamma_{0},\Gamma_{1})$ be a skew Brauer graph and $B$ be its corresponding skew Brauer graph algebra. We assume that the least common multiple $\overline{m}$ of the multiplicities in $\Gamma$ and 2 are invertible in the field $k$. Then, for any subset of edges $E$ in $\Gamma$, the following holds.

    \begin{enumerate}[label=(\arabic*)]
    \item The Brauer graph algebra $B'$ obtained by a generalized Kauer move of $E$ in $\Gamma$ is derived equivalent to $B$.
    \item The silting mutation of $B$ over the projective $B$-module corresponding to the edges in $\Gamma_{1}\backslash E$ is tilting and its endomorphism algebra is isomorphic to $B'$.
    \end{enumerate}
\end{thm*}

\medskip

The proof for skew Brauer graph algebras is analogous to the one for Brauer graph algebras with arbitrary multiplicity. We first construct a $\Z/2\Z$-covering $\widetilde{\Gamma}$ of $\Gamma$ which is a standard Brauer graph (with multiplicity). In this case, one can recover $B$ from the Brauer graph algebra $\widetilde{B}$ associated to $\widetilde{\Gamma}$ thanks to the construction of skew group algebras as defined in \cite{RR} (cf Proposition \ref{prop:covering and skew group algebra}). Again, there is a commutativity between constructing this covering and applying the generalized Kauer moves. This can be summarized with the following commutative diagram

\smallskip

\begin{center}
    \begin{tikzcd}[column sep=5cm, row sep=2cm]
        \widetilde{\Gamma} \arrow{r}[above]{\mbox{generalized Kauer}}[below]{\mbox{ move of $\widetilde{H'}$}} \arrow{d}[description]{\mbox{$\Z/2\Z$-covering}} &\mu^{+}_{\widetilde{H'}}(\widetilde{\Gamma}) \arrow{d}[description]{\mbox{$\Z/2\Z$-covering}} \\
        \Gamma \arrow{r}[above]{\mbox{generalized Kauer}}[below]{\mbox{move of $H'$}}&\mu^{+}_{H'}(\Gamma)
    \end{tikzcd}
\end{center}

\smallskip

\noindent where $\widetilde{H'}$ is the set of half-edges in $\widetilde{\Gamma}$ consisting of all the lifts of the elements in $H'$. Using this covering, we also show that skew Brauer graph algebras of multiplicity identically 1 can be understood as trivial extension of skew gentle algebras (cf Theorem \ref{thm:skew Brauer graph algebras of multiplicity as trivial extension}). This generalizes Theorem 3.13 in \cite{Schroll} which interprets Brauer graph algebra of multiplicity identically 1 as trivial extension of gentle algebras.

Note that a notion of skew Brauer graph algebra is also studied in an ongoing project by Elsener, Guazzelli and Valdivieso. Some of their results were presented in a conference in Oslo.

\medskip

\noindent \textbf{\fontfamily{bch}\selectfont Organisation of the paper.} In the first section, we recall the definition and some properties of skew group algebras. In particular, we study the trivial extension, the silting mutation and the morphism spaces of $G$-invariant objects in skew group algebras. These are tools needed in the next two sections.

In the second section, we prove our first main result for Brauer graph algebras with arbitrary multiplicity. For this, we first recall the construction of the covering by a Brauer graph of multiplicity identically 1 in \cite{Asashiba}. Then, we explain how one can recover $B$ from the Brauer graph algebra associated to its covering thanks to the construction of skew group algebras. Finally, we show a commutativity between the covering and the generalized Kauer moves. 

The third section follows the same outline than the second for skew Brauer graph algebras. We also explain how skew Brauer graph algebras of multiplicity identically 1 can be seen as trivial extensions of skew gentle algebras.

\medskip

\noindent \textbf{\fontfamily{bch}\selectfont Acknowledgments.} This paper is part of a PhD thesis supervised by Claire Amiot and supported by the CNRS. The author deeply thanks her advisor for taking the time to read this paper and for the helpful advice and comments on it. Part of this work was done while the author was in Sherbrooke University. The mobility was co funded by the Mitacs Globalink Research Award and the IDEX. She also thanks Thomas Brüstle for interesting discussions and Sherbrooke University for warm welcome.

\medskip

\phantomsection
\section*{Notations and conventions}
\addcontentsline{toc}{section}{Notations and conventions}

\medskip

In this paper, all algebras are supposed to be finite dimensional over an algebraically closed field $k$. For any algebra $\Lambda$, we denote by $\m{\Lambda}$ the category of finitely generated right $\Lambda$-modules, $\Db{\Lambda}$ the bounded derived category associated to $\m{\Lambda}$ and $\per{\Lambda}$ the full subcategory of $\Db{\Lambda}$ consisting of perfect complexes. Moreover, arrows in a quiver will be composed from right to left : for arrows $\alpha$, $\beta$, we write $\beta \alpha$ for the path from the start of $\alpha$ to the target of $\beta$. Similarly, morphisms will be composed from right to left.

\medskip

\section{Skew group algebras}

\medskip

In this section, we collect some properties of skew group algebras. First, we recall the definition of a skew group algebra from \cite{RR}. Throughout this section, $G$ denotes a finite abelian group whose order $|G|$ is invertible in the field $k$. By a $G$-action on an algebra, we mean a left action by automorphisms.

\medskip

\begin{defn} \label{def:skew group algebra}
    Let $\Lambda$ be a $k$-algebra with a $G$-action. The \textit{skew group algebra} $\Lambda G$ associated to $\Lambda$ is the algebra defined by $\Lambda G=\Lambda\otimes_{k}kG$ as a $k$-vector space and whose multiplication is induced by the $G$-action on $\Lambda$ i.e.

    \[(\lambda\otimes g) \ . \ (\mu\otimes h)=\lambda (g\, . \, \mu)\otimes gh\]

    \noindent for all $\lambda,\mu\in\Lambda$, $g,h\in G$ and we extend this relation by linearity and by distributivity.
\end{defn}

\medskip

Denoting by $\widehat{G}$ the dual of $G$, note that $\Lambda G$ has a natural $\widehat{G}$-action which is given by

\[\chi\, .\, (\lambda\otimes g)=\chi(g)\lambda\otimes g\]

\noindent for all $\chi\in\widehat{G}$, $\lambda\in\Lambda$, $g\in G$ and where the relation is extended by linearity.

\medskip

Moreover, given two algebras $\Lambda$ and $\Lambda'$ with a $G$-action and $M$ a $\Lambda$-$\Lambda'$-bimodule, one can define $MG:=M\otimes_{k}kG$. It has a natural structure of $\Lambda G$-$\Lambda'\,G$-bimodule defined by

\[(\lambda\otimes g)\ .\ (m\otimes h)=\lambda (g \,. \, m)\otimes gh \qquad \mbox{and} \qquad (m\otimes h)\ .\ (\lambda' \otimes g')=m (h\, . \, \lambda')\otimes hg'\]

\noindent for all $\lambda\in \Lambda$, $\lambda'\in \Lambda'$, $g,h,g'\in G$ and we extend these relations by linearity and by distributivity. 

\subsection{Trivial extension}

\medskip

We first study the trivial extension of a skew group algebra. Let us recall the definition of the trivial extension of an algebra. We denote by $D=\Hom[k]{-}{k}$ the standard duality on $\m{\Lambda}$.

\medskip

\begin{defn} \label{def:trivial extension}
Let $\Lambda$ be a $k$-algebra. The \textit{trivial extension} $\T{\Lambda}$ of $\Lambda$ by the $\Lambda$-$\Lambda$-bimodule $D\Lambda$ is the algebra defined by $\T{\Lambda}=\Lambda\oplus D\Lambda$ as a $k$-vector space and whose multiplication is induced by the $\Lambda$-$\Lambda$-bimodule structure of $D\Lambda$ i.e.

\[(a,\phi) \ . \ (b,\psi)=(ab, a\psi+\phi b)\]

\noindent for all $a,b\in \Lambda$ and $\phi,\psi\in D\Lambda$.
\end{defn}

\medskip

Let $\Lambda$ be a $k$-algebra with a $G$-action. Then, $G$ acts on $D\Lambda$ as follows

\[g\ . \ \phi:(b\mapsto \phi(g^{-1} \, . \, b))\]

\noindent for all $g\in G$ and $\phi\in D\Lambda$. In particular, $(D\Lambda)G=D\Lambda\otimes_{k}\Lambda G$ has a natural structure of $\Lambda G$-$\Lambda G$-bimodule. 

\medskip

\begin{lem}
    For any $k$-algebra $\Lambda$ with a $G$-action, there is an isomorphism of $\Lambda G$-$\Lambda G$-bimodules 

    \[(D\Lambda)G\simeq D(\Lambda G)\]
\end{lem}

\medskip

\begin{proof}
    The isomorphism of $\Lambda G$-$\Lambda G$-bimodules is given by

    \smallskip

    \begin{equation*}
        \begin{aligned}
            D(\Lambda G) && \longrightarrow &&&(D\Lambda) G \\
            \phi &&\longmapsto &&&\sum_{h\in G} \phi_{h}\otimes h \\
            \psi_{g} && \longmapsfrom &&& \psi \otimes g
        \end{aligned}
    \end{equation*}

    \smallskip

    \noindent where $\phi_{h} : \lambda \mapsto \phi((h^{-1}\, .\, \lambda )\otimes h^{-1})\in D\Lambda$ and $\psi_{g}: (\lambda \otimes h) \mapsto \psi(\delta_{h^{-1}}^{g} (g\, .\, \lambda)) \in D(\Lambda G)$ with $\delta$ being the Kronecker symbol.
\end{proof}

\medskip

Hence, the $G$-action on $\Lambda$ induces a $G$-action on $\T{\Lambda}$ which is given by

\[g\ . \ (a,\phi)=(g \, . \, a, g\, . \, \phi)\]

\noindent for all $g\in G$ and $(a,\phi)\in \T{\Lambda}$. In particular, one can construct the skew group algebra associated to $\T{\Lambda}$.

\medskip

\begin{prop} \label{prop:trivial extension of a skew group algebra}
    For any $k$-algebra $\Lambda$ with a $G$-action, there is an isomorphism of algebras

    \[\T{\Lambda G}\simeq \T{\Lambda}G\]
\end{prop}

\medskip

\begin{proof}
    Let us define the following linear map
    
    \smallskip
    
    \begin{equation*}
        \begin{aligned}
            \Phi: &&\T{\Lambda G} &&&\longrightarrow &&\T{\Lambda}G \\
            &&(a\otimes g,\phi) &&&\longmapsto &&(a,0)\otimes g +\sum_{h\in G}(0,\phi_{h})\otimes h
        \end{aligned}   
    \end{equation*}

    \smallskip

    \noindent where $\phi_{h}$ is the morphism in $D\Lambda$ defined by 

    \[\phi_{h}:b\longmapsto \phi(h^{-1}\, .\, b\otimes h^{-1})\]
    
    \noindent We denote by $A:= \Phi(a\otimes g,\phi)\Phi(a'\otimes g',\phi')$. We have,

    \smallskip

    \begin{equation*}
        \begin{aligned}
            A&=\left[(a,0)\otimes g + \sum_{h\in G}(0,\phi_{h})\otimes h\right]\left[(a',0)\otimes g' + \sum_{h\in G}(0,\phi'_{h})\otimes h\right] \\
            &=(a(g\, .\, a'),0)\otimes gg'+ \sum_{h\in G}(0, \phi_{h}\, . \, (h\, . \, a'))\otimes hg'+ \sum_{h\in G}(0, a\, .\, (g\, .\, \phi'_{h}))\otimes gh \\
            &=(a(g\, .\, a'),0)\otimes gg'+\sum_{h\in G}(0,(\phi\, .\, (a'\otimes g'))_{h})\otimes h +\sum_{h\in G}(0,((a\otimes g)\, .\, \phi')_{h})\otimes h \\
            &=\Phi((a\otimes g,\phi)(a'\otimes g',\phi'))
        \end{aligned}
    \end{equation*}

    \smallskip

    \noindent for all $(a\otimes g,\phi),(a'\otimes g',\phi')\in\T{\Lambda G}$. Hence, $\Phi$ is a morphism of algebras. Moreover, one can easily check that $\Phi$ is also invertible. Its inverse is given by

    \smallskip

    \begin{equation*}
    \begin{aligned}
        \T{\Lambda}G &&&\longrightarrow &&\T{\Lambda G} \\
        (b,\psi)\otimes g &&&\longmapsto &&(b\otimes g, \psi_{g})
    \end{aligned}
    \end{equation*}

    \smallskip

    \noindent where $\psi_{g}$ is the morphism in $D(\Lambda G)$ given by

    \[\psi_{g}:(a\otimes h)\longmapsto \psi(\delta_{h^{-1}}^{g}(g\, . \, a))\]

    \noindent where $\delta$ denotes the Kronecker symbol. This concludes the proof.
\end{proof}

\medskip

\subsection{Morphism space of G-invariant objects}

\medskip

In this subsection, we are interested in the morphism space between $G$-invariant objects in the bounded derived category. The case of the endomorphism space of a $G$-invariant object has already been examined in \cite[Theorem 2.10]{AB}.

\medskip

A $G$-action on a finite dimensional $k$-algebra $\Lambda$ induces a right $G$-action on its bounded derived category $\Db{\Lambda}$ in the sense of \cite[Definition 3.1]{Elagin}. This action is given by the autoequivalences $-\lotimes{\Lambda}{\Lambda_{g}}$ for all $g\in G$, where $\Lambda_{g}$ is the $\Lambda$-module whose underlying group is $\Lambda$ and whose action is twisted by $g$. In the following, $X^{g}$ denotes $X\lotimes{\Lambda}\Lambda_{g}$ for all $X\in\Db{\Lambda}$ and $f^{g}:X^{g}\rightarrow Y^{g}$ denotes $f\lotimes{\Lambda}{\Lambda_{g}}$ for all morphism $f:X\rightarrow Y$ in $\Db{\Lambda}$. Moreover, 

\smallskip

\begin{center}
    \begin{tikzcd}[column sep=2cm]
        \Db{\Lambda} \arrow[cramped,yshift=1mm]{r}[above]{-\lotimes{\Lambda}{\Lambda G}} &\Db{\Lambda G} \arrow[cramped, yshift=-1mm]{l}[below]{\mathrm{Res}}
    \end{tikzcd}
\end{center}

\smallskip

\noindent form an adjoint pair of triangulated functors in both directions and the unit of the adjunction splits \cite[Lemma 2.3.1]{LeMeur}. In particular, we obtain a functorial isomorphism for all $X\in\Db{\Lambda}$

\begin{equation} \label{eq:restriction}
    \mathrm{Res}(X\lotimes{\Lambda}{\Lambda G})\simeq \bigoplus_{g\in G}X^{g}
\end{equation}

\medskip

\begin{defn} \label{def:G-invariant object}
    An object $X\in\Db{\Lambda}$ is said to be \textit{$G$-invariant} if there exist isomorphisms $\iota_{g}:X^{g^{-1}}\rightarrow X$
    for all $g\in G$ satisfying that $\iota_{gh}=\iota_{g}(\iota_{h})^{g^{-1}}$ for all $g,h\in G$.
\end{defn}

\medskip

\begin{ex} \label{ex:G-invariant object}

    \begin{enumerate}[label=(\arabic*)]
    \item \label{item:G-invariant object 1} It is easy to check that $\Lambda$ is a $G$-invariant object in $\Db{\Lambda}$ where the isomorphisms $\iota_{g}:\Lambda^{g^{-1}}\simeq\Lambda_{g^{-1}}\rightarrow\Lambda$ are given by the action of $g\in G$. Moreover, if $e$ is an idempotent of $\Lambda$ stable under the action of $G$, the same isomorphisms give us that $e\Lambda$ is also a $G$-invariant object.

    \item \label{item:G-invariant object 2} Let $f:X\rightarrow Y$ be a morphism in $\Db{\Lambda}$ between two $G$-invariant objects. We denote by $\iota_{g}^{X}:X^{g^{-1}}\rightarrow X$ and $\iota_{g}^{Y}:Y^{g^{-1}}\rightarrow Y$ the isomorphisms corresponding to $X$ and $Y$ respectively. Let us assume that for all $g\in G$ the following diagram commutes

    \smallskip

    \begin{center}
        \begin{tikzcd}[row sep=1.5cm, column sep=1.5cm]
            X^{g^{-1}} \arrow{d}[left]{\iota_{g}^{X}} \arrow{r}[above]{f^{g^{-1}}} &Y^{g^{-1}} \arrow{d}[right]{\iota_{g}^{Y}} \\
            X \arrow{r}[below]{f} &Y
        \end{tikzcd}
    \end{center}

    \smallskip
        
    \noindent Then, the cone of $f$ denoted $\mathrm{Cone}(f)$ is also $G$-invariant. Indeed, one can check that the isomorphisms $\iota_{g}:\mathrm{Cone}(f)^{g^{-1}}\rightarrow\mathrm{Cone}(f)$ defined by

    \begin{equation*}
        \iota_{g}=\begin{pmatrix}
            \iota_{g}^{X}[1] &0 \\[0.7em]
            0 &\iota_{g}^{Y}
        \end{pmatrix}
    \end{equation*}

    \noindent for all $g\in G$ satisfy the condition of Definition \ref{def:G-invariant object}.
    \end{enumerate}
\end{ex}

\medskip

We now adapt the results about the endomorphism algebra of a $G$-invariant object from \cite{AB} to the morphism space of $G$-invariant objects. The proof of the following statements follows the same outline than the proofs of Lemma 2.5 and Proposition 2.8 in \cite{AB}.

\medskip

\begin{lem} \label{lem:G acts on morphism space}
    Let $X$ and $Y$ be two $G$-invariant objects in $\Db{\Lambda}$. Then, $G$ acts on the morphism space $\Hom[\Db{\Lambda}]{X}{Y}$ as follows

    \[g\ .\ f=\iota_{g}^{Y}\circ f^{g^{-1}}\circ (\iota_{g}^{X})^{-1}\]

    \noindent for all $g\in G$ and $f\in\Hom[\Db{\Lambda}]{X}{Y}$, where $\iota_{g}^{X}:X^{g^{-1}}\rightarrow X$ and $\iota_{g}^{Y}:Y^{g^{-1}}\rightarrow Y$ are the isomorphisms corresponding to $X$ and $Y$ respectively.
\end{lem}

\medskip

 Given two $G$-invariant objects $X$ and $Y$ in $\Db{\Lambda}$, $\Hom[\Db{\Lambda}]{X}{Y}G=\Hom[\Db{\Lambda}]{X}{Y}\otimes_{k}kG$ has a natural structure of $(\mathrm{End}_{\Db{\Lambda}}(Y)G)$-$(\mathrm{End}_{\Db{\Lambda}}(X)G)$-bimodule (cf after Definition \ref{def:skew group algebra}). Moreover, thanks to Theorem 2.10 in \cite{AB}, there is an isomorphism of algebras

\[\mathrm{End}_{\Db{\Lambda}}(Z)G\simeq\mathrm{End}_{\Db{\Lambda G}}(Z\lotimes{\Lambda}{\Lambda G})\]

\noindent for all $G$-invariant object $Z$ in $\Db{\Lambda}$. Hence, $\Hom[\Db{\Lambda G}]{X\lotimes{\Lambda}{\Lambda G}}{Y\lotimes{\Lambda}{\Lambda G}}$ has also a natural structure of $(\mathrm{End}_{\Db{\Lambda}}(Y)G)$-$(\mathrm{End}_{\Db{\Lambda}}(X)G)$-bimodule.

\medskip

\begin{prop} \label{prop:isomorphism of morphism space}
    Let $X$ and $Y$ be two $G$-invariant objects in $\Db{\Lambda}$ and $\iota_{g}^{X}:X^{g^{-1}}\rightarrow X$ and $\iota_{g}^{Y}:Y^{g^{-1}}\rightarrow Y$ be the isomorphisms corresponding to $X$ and $Y$ respectively.  Then, there is an isomorphism of $(\mathrm{End}_{\Db{\Lambda}}(Y)G)$-$(\mathrm{End}_{\Db{\Lambda}}(X)G)$-bimodules given by 

    \begin{equation*}
        \begin{aligned}
            \Hom[\Db{\Lambda}]{X}{Y}G &&\longrightarrow &&&\Hom[\Db{\Lambda G}]{X\lotimes{\Lambda}{\Lambda G}}{Y\lotimes{\Lambda}{\Lambda G}} \\
            f\otimes g &&\longmapsto &&& ((f\circ\iota_{g}^{X})\otimes 1_{\Lambda G})\circ L_{g}^{X}
        \end{aligned}
    \end{equation*}

    \noindent where $L_{g}^{X}:X\lotimes{\Lambda}{\Lambda G}\rightarrow X^{g^{-1}}\lotimes{\Lambda}{\Lambda G}$ is the isomorphism induced by the multiplication on the left by $1\otimes g$ in $\Lambda G$.
\end{prop}

\medskip

\subsection{Silting mutation}

\medskip

Let us now study silting mutations for skew group algebras. We first recall the definition and some properties of silting mutations given in \cite{AI}. Throughout this section, $\Lambda$ denotes a finite dimensional $k$-algebra. For any object $M$ in $\per{\Lambda}$, $\mathrm{add}(M)$ denotes the full subcategory of $\per{\Lambda}$ consisting of direct sums of direct summands of $M$ and $\mathrm{thick}(M)$ denotes the smallest triangulated subcategory of $\per{\Lambda}$ containing $M$ that is closed under direct summands.

\medskip

\begin{defn}\label{def:silting and tilting object}

Let $M$ be an object in $\per{\Lambda}$.
\smallskip
    
     \begin{itemize}[label=\textbullet, font=\tiny]
        \item $M$ is said to be \textit{silting} if $\Hom[\per{\Lambda}]{M}{M[>0]}=0$ and $\mathrm{thick}(M)=\per{\Lambda}$.
        \item $M$ is said to be \textit{tilting} if $M$ is silting and $\Hom[\per{\Lambda}]{M}{M[<0]}=0$.
            
     \end{itemize}
\end{defn}

\medskip

\begin{defn}\label{def:left mutation}
   Let $M$ be an object in $\per{\Lambda}$.

   \smallskip

    \begin{itemize}[label=\textbullet, font=\tiny]
        \item We say that a morphism $f:X\rightarrow Y$ in $\per{\Lambda}$ is \textit{left minimal} if the only morphisms $h\in\mathrm{End}_{\per{\Lambda}}(Y)$ satisfying $hf=f$ are isomorphisms.
        \item We say that a morphism $f:X\rightarrow Y$ in $\per{\Lambda}$ is a \textit{left $\mathrm{add}(M)$-approximation} of $X$ if $Y\in\mathrm{add}(M)$ and $\Hom[\per{\Lambda}]{f}{M}$ is surjective.
        \item We say that a morphism $f:X\rightarrow Y$ in $\per{\Lambda}$ is a \textit{left minimal $\mathrm{add}(M)$-approximation} of $X$ if $f$ is left minimal and a left $\mathrm{add}(M)$-approximation of $X$. It is unique up to isomorphisms.
        
        \item Let $M_{0}$ be a direct summand of $M$ and $f:M/M_{0}\rightarrow M_{0}'$ be the left minimal $\mathrm{add}(M_{0})$-approximation of $M/M_{0}$. The \textit{left mutation} of $M$ over $M_{0}$ is the object 
        
        \[\mu^{+}(M;M_{0}):=M_{0}\,\oplus\, \mathrm{Cone}(f)\in\per{\Lambda}\]
        
        \noindent where $\mathrm{Cone}(f)$ denotes the cone of $f$.

        \end{itemize}
\end{defn}

\medskip

\begin{thm}[Aihara-Iyama {\cite[Theorem 2.32]{AI}}] \label{thm:AI}

\phantom{} 
\begin{itemize}[label={\textbullet}, font=\tiny]

\item Any left mutation of a silting object of $\per{\Lambda}$ is a silting object of $\per{\Lambda}$.

\item For any $M_{0}$ direct summand of a tilting object $M$ in $\per{\Lambda}$, the left mutation $\mu^{+}(M;M_{0})$ is tilting if and only if $\Hom[\per{\Lambda}]{M_{0}}{f}$ is injective where $f$ is the left minimal $\mathrm{add}(M_{0})$-approximation of $M/M_{0}$.

\end{itemize}

\end{thm}

\medskip

\begin{rem} \label{rem:right mutation}
    We can define dually a right approximation and a right mutation. Moreover, we have a dual version of Theorem \ref{thm:AI} for right mutations.
\end{rem}

\medskip

We want now to prove the following result on silting mutations for skew group algebras.

\medskip

\begin{prop} \label{prop:silting mutation in skew group algebras}
    Let $\Lambda$ be a finite dimensional $k$-algebra with a $G$-action. Let $M$ be a silting object in $\per{\Lambda}$ and $M_{0}$ be a direct summand of $M$. We denote by $f:M/M_{0}\rightarrow M_{0}'$ the left minimal $\mathrm{add}(M_{0})$-approximation of $M/M_{0}$. If $M_{0}$, $M/M_{0}$ and $M_{0}'$ are $G$-invariant, then there is an isomorphism in $\per{\Lambda G}$

    \[\mu^{+}(M\lotimes{\Lambda}{\Lambda G};M_{0}\lotimes{\Lambda}{\Lambda G})\simeq\mu^{+}(M;M_{0})\lotimes{\Lambda}{\Lambda G}\]
\end{prop}

\medskip

\begin{rem} \label{rem:AB}
    Let us assume that the hypotheses of the previous proposition hold. Thanks to Theorem 2.10 of \cite{AB}, we conclude that $\mu^{+}(M\lotimes{\Lambda}{\Lambda G};M_{0}\lotimes{\Lambda}{\Lambda G})$ is tilting in $\per{\Lambda G}$ if $\mu^{+}(M;M_{0})$ is tilting in $\per{\Lambda}$ and $G$-invariant.
\end{rem}

\medskip

In order to prove the previous proposition, we need to study how the functor $-\lotimes{\Lambda}{\Lambda G}$ behaves with left minimal morphisms.

\medskip

\begin{lem} \label{lem:left minimal morphism in skew group algebra}
    Let $\Lambda$ be a finite dimensional $k$-algebra with a $G$-action and $f:X\rightarrow Y$ be a left minimal morphism in $\per{\Lambda}$. If $X$ and $Y$ are $G$-invariant, then $f\lotimes{\Lambda}{\Lambda G}$ is a left minimal morphism in $\per{\Lambda G}$.
\end{lem}

\medskip

\begin{proof}
    Let $h:Y\lotimes{\Lambda}{\Lambda G}\rightarrow Y\lotimes{\Lambda}{\Lambda G}$ be an endomorphism in $\per{\Lambda G}$ such that $h\circ (f\lotimes{\Lambda}{\Lambda G})=f\lotimes{\Lambda}{\Lambda G}$. We want to prove that $h$ is an isomorphism. In the following, we denote by $g_{0},\ldots,g_{n-1}$ the elements of the group $G$ where $g_{0}=1_{G}$. Thanks to Proposition 2.8 in \cite{AB}, $h$ can be written as the sum of $h_{g_{i}}\otimes g_{i}$ for $i=0,\ldots,n-1$ in $\mathrm{End}_{\per{\Lambda}}(Y)G$, for some endomorphisms $h_{g_{i}}:Y\rightarrow Y$ in $\per{\Lambda}$. Similarly, using Proposition \ref{prop:isomorphism of morphism space}, $f\lotimes{\Lambda}{\Lambda G}$ can be written as $f\otimes g_{0}$ in $\Hom[\per{\Lambda}]{X}{Y}G$ and the equality $h\circ (f\lotimes{\Lambda}{\Lambda G})=f\lotimes{\Lambda}{\Lambda G}$ gives us

    \[f\otimes g_{0}=\left(\sum_{i=0}^{n-1}h_{g_{i}}\otimes g_{i}\right)(f\otimes g_{0})=\sum_{i=0}^{n-1}h_{g_{i}}\circ (g_{i}\, . \,f)\otimes g_{i}\]

    \noindent By identification, we deduce that $h_{g_{0}}f=f$ and $h_{g_{i}}\circ (g_{i}\, .\, f)=0$ for $i=1,\ldots,n-1$. Let us prove that $h$ admits a right inverse i.e. that there exist $t_{g_{0}},\ldots,t_{g_{n-1}}\in\mathrm{End}_{\per{\Lambda}}(Y)$ such that

    \begin{equation} \label{eq:right inverse}
        \id{\Lambda}\otimes g_{0}=\left(\sum_{i=0}^{n-1}h_{g_{i}}\otimes g_{i}\right)\left(\sum_{j=0}^{n-1}t_{g_{j}}\otimes g_{j}\right)=\sum_{i=0}^{n-1}\sum_{j=0}^{n-1}h_{g_{i}g_{j}^{-1}}\circ (g_{i}g_{j}^{-1}\ .\ t_{g_{j}})\otimes g_{i}
    \end{equation}

    \noindent Thanks to the relations obtained by identification in \eqref{eq:right inverse}, one can check by a descending induction that for all $i=1,\ldots,n-1$, $t_{g_{i}}$ can be written as the sum of $h_{g_{l}}'\circ (g_{i}g_{l}^{-1} \ . \ t_{g_{l}})$ for $l=0,\ldots,i-1$ for some $h_{g_{l}}':Y\rightarrow Y$ in $\per{\Lambda}$ satisfying that $h_{g_{l}}'\circ (g_{i}g_{l}^{-1} \ . \ f)=0$. In particular, by identification in \eqref{eq:right inverse}, we obtain

    \[\id{\Lambda}=h_{g_{0}}t_{g_{0}}+\sum_{j=1}^{n-1}h_{g_{j}^{-1}}\circ (g_{j}^{-1} \ . \ t_{g_{j}})=h_{g_{0}}t_{g_{0}}+h_{g_{0}}''t_{g_{0}}\]

    \noindent for some endomorphism $h_{g_{0}}'':Y\rightarrow Y$ in $\per{\Lambda}$ satisfying that $h_{g_{0}}''f=0$, thanks to the previous induction. Since $(h_{g_{0}}+h_{g_{0}}'')\circ f=f$ and $f$ is left minimal, $h_{g_{0}}+h_{g_{0}}''$ is invertible and we can set $t_{g_{0}}$ to be its inverse. Hence, one can use the previous induction to define the $t_{g_{i}}$ for $i=1,\ldots,n-1$ so that \eqref{eq:right inverse} holds. Similarly, one can check that $h$ admits a left inverse which ends the proof. 
\end{proof}

\medskip

\begin{proof}[Proof of Proposition \ref{prop:silting mutation in skew group algebras}]
We assume that $M_{0}$, $M/M_{0}$ and $M_{0}'$ are $G$-invariant. Since the functor $-\lotimes{\Lambda}{\Lambda G}:\per{\Lambda}\rightarrow\per{\Lambda G}$ is triangulated, note that we have

\[(M/M_{0})\lotimes{\Lambda}{\Lambda G}=(M\lotimes{\Lambda}{\Lambda G})/(M_{0}\lotimes{\Lambda}{\Lambda G})\]

\noindent Thanks to Lemma \ref{lem:left minimal morphism in skew group algebra}, it remains to prove that $f\lotimes{\Lambda}{\Lambda G}$ is a left $\mathrm{add}(M_{0}\lotimes{\Lambda}{\Lambda G})$-approximation of $(M/M_{0})\lotimes{\Lambda}{\Lambda G}$. Applying the functor $-\lotimes{\Lambda}{\Lambda G}$ on the triangle induced by $f$ in $\per{\Lambda}$, we obtain the following triangle in $\per{\Lambda G}$

\smallskip

\begin{center}
    \begin{tikzcd}[column sep=1.3cm]
        (M/M_{0})\lotimes{\Lambda}{\Lambda G} \arrow{r}[above]{f\lotimes{\Lambda}{\Lambda G}} &M_{0}'\lotimes{\Lambda}{\Lambda G} \arrow{r} &\mathrm{Cone}(f)\lotimes{\Lambda}{\Lambda G} \arrow{r} &(M/M_{0})\lotimes{\Lambda}{\Lambda G}\, [1]
    \end{tikzcd}
\end{center}

\smallskip

\noindent where $\mathrm{Cone}(f)$ denotes the cone of $f$. Hence, applying $\Hom[\per{\Lambda G}]{-}{M_{0}\lotimes{\Lambda}{\Lambda G}}$ to the previous triangle, it suffices to prove that 

\[H:=\Hom[\per{\Lambda G}]{\mathrm{Cone}(f)\lotimes{\Lambda}{\Lambda G}}{M_{0}\lotimes{\Lambda}{\Lambda G}[1]}=0\]

\noindent For this, note that we have the following isomorphisms of vector spaces

\begin{equation} \notag
    \begin{aligned}
        H&\simeq \Hom[\per{\Lambda}]{\mathrm{Cone}(f)}{\mathrm{Res}(M_{0}\lotimes{\Lambda}{\Lambda G}\, [1])} \\
        &\simeq \Hom[\per{\Lambda}]{\mathrm{Cone}(f)}{\oplus_{g\in G}\, (M_{0})^{g}\, [1]} \\
        &\simeq \Hom[\per{\Lambda}]{\mathrm{Cone}(f)}{\oplus_{g\in G}\, M_{0}[1]}
    \end{aligned}
\end{equation}

\noindent where the second isomorphism comes from \eqref{eq:restriction} and the third arises from the fact that $M_{0}$ is $G$-invariant. Since the left mutation of a silting object is silting by Theorem \ref{thm:AI}, we obtain that $H=0$ which concludes the proof. 
\end{proof}

\medskip

\section{Generalized Kauer moves for Brauer graphs with multiplicity}

\medskip

In this section, we extend the notion of generalized Kauer moves defined for Brauer graphs of multiplicity one in  \cite{Soto} to the case of Brauer graphs with arbitrary multiplicity. We will show that these generalized Kauer moves can be understood in terms of silting mutations. This generalizes Theorem 3.10 in \cite{Soto} for arbitrary multiplicity.

\medskip

\subsection{Brauer graph algebras with multiplicity}

\medskip

Let us recall the notion of Brauer graph algebras with multiplicity. These are finite dimensional algebras defined thanks to a combinatorial data called a Brauer graph. As in \cite{Soto}, we will use the definition of a Brauer graph using combinatorial maps (see for instance \cite{Lazarus}).

\medskip

\begin{defn} \label{def:Brauer graph H}
A \textit{Brauer graph} is the data $\Gamma=(H,\iota,\sigma,m)$ where 

\smallskip

\begin{itemize}[label=\textbullet, font=\tiny]
\item $H$ is the set of half-edges;
\item $\iota$ is a permutation of $H$ without fixed points satisfying $\iota^{2}=\id{H}$ : it is called the \textit{pairing};
\item $\sigma$ is a permutation of $H$ called the \textit{orientation};
\item $m:H\rightarrow \Z_{>0}$ is a map that is constant on a $\sigma$-orbit : it is called the \textit{multiplicity}.
\end{itemize}

\end{defn}

\medskip

To the data $\Gamma=(H,\iota,\sigma,m)$, one can naturally define a graph whose vertex set is $\Gamma_{0}=H/\sigma$, edge set is $\Gamma_{1}=H/\iota$ and source map is the natural projection $s:H\rightarrow H/\sigma$. Moreover, each cycle in the decomposition of $\sigma$ gives rise to a cyclic ordering of the edges around a vertex. Since $m:H\rightarrow\Z_{>0}$ is constant on the $\sigma$-orbits, it induces a map $\widetilde{m}:H/\sigma\rightarrow\Z_{>0}$. Hence, we recover the usual definition of a Brauer graph with multiplicity as defined in \cite{Schroll} for instance. 

\medskip

\begin{rem} \label{rem:excluded Brauer graphs}
    In order to simplify the definition of a Brauer graph algebra, we do not take into consideration the following Brauer graph

    \smallskip
    
    \begin{center}
        \begin{tikzpicture}
            \tikzstyle{vertex}=[circle,draw]
        \node[vertex] (a) at (-1.5,0) {};
        \node[vertex] (b) at (1.5,0) {};
        \draw (a)--(b) node[near start, above, scale=0.9]{$1^{+}$} node[near end, above, scale=0.9]{$1^{-}$};
        \end{tikzpicture}
    \end{center}

    \smallskip

    \noindent which corresponds to $H=\{1^{+}, 1^{-}\}$, $\iota=(1^{+} \ 1^{-})$, $\sigma=\id{H}$ and $m$ being identically 1. More precisely, we assume that such graph does not appear in any connected component of a given Brauer graph. Moreover, the previous graph is not really interesting for our purpose since the generalized Kauer move of the edge 1 is trivial.
    
\end{rem}

\medskip

From now on, we identify $\Gamma=(H,\iota,\sigma,m)$ with its corresponding graph that we have constructed above. By convention, the orientation of a Brauer graph will correspond to the local embedding of each vertex into the counterclockwise oriented plane. 

\medskip

\begin{ex} \label{ex:Brauer graph with multiplicity}
    Let $\Gamma=(H,\iota,\sigma,m)$ be the Brauer graph where the set of half-edges is $H=\{1^{+},1^{-},2^{+},2^{-},3^{+},3^{-},4^{-},4^{+}\}$, the involution $\iota$ is $(1^{+}\, 1^{-})(2^{+}\ 2^{-})(3^{+}\ 3^{-})(4^{+}\ 4^{-})$, the permutation $\sigma$ is $(1^{-} \ 4^{-} \ 3^{-} \ 2^{-})(2^{+} \ 3^{+})$ and the multiplicity $m:H\rightarrow\Z_{>0}$ is given by $m(1^{-})=m(2^{-})=m(3^{-})=m(4^{-})=m(4^{+})=1$ and $m(2^{+})=m(3^{+})=m(1^{+})=2$. With our previous convention, this Brauer graph can be represented as follows

    \smallskip

    
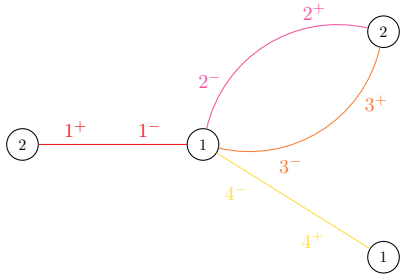
\begin{figure}[H]
        \centering
        \begin{tikzpicture}[scale=1.2]
        \tikzstyle{vertex}=[circle,draw, scale=0.6, minimum size=0.7cm]
        \node[vertex] (a) at (-1,0) {2};
        \node[vertex] (b) at (1,0) {1};
        \node[vertex] (c) at (3,1.25) {2};
        \node[vertex] (d) at (3,-1.25) {1};
        \draw[Red] (a)--(b) node[near end, above, scale=0.75]{$1^{-}$} node[near start, above, scale=0.75]{$1^{+}$};
        \draw[VioletRed] (b) to[bend left=45] node[near end, above, scale=0.75]{$2^{+}$} node[near start,left,scale=0.75]{$2^{-}$} (c) ; 
        \draw[Orange] (b) to[bend right=45] node[near start, right, yshift=-2mm,scale=0.75]{$3^{-}$} node[near end, right, yshift=-1mm, scale=0.75]{$3^{+}$} (c) ;
        \draw[Goldenrod] (b)--(d) node[near start, below left, scale=0.75]{$4^{-}$} node[near end, below left, scale=0.75]{$4^{+}$};
        \end{tikzpicture}
    \end{figure}

\smallskip

\noindent where the multiplicity of an half-edge is given by the value in the vertex it is incident to. From now on, we will only write on the graph the multiplicities that are strictly greater than one.
\end{ex}

\medskip

Given a Brauer graph $\Gamma=(H,\iota,\sigma,m)$, one can construct a quiver $Q_{\Gamma}=(Q_{0},Q_{1})$ as in \cite{Schroll}, which is reinterpreted with our definition as follows

\smallskip

\begin{itemize}[label=\textbullet, font=\tiny]
\item The vertex set $Q_{0}$ is the edge set $H/\iota$ of $\Gamma$. In the following, we denote by $[h]$ the edge in $H/\iota$ associated to $h\in H$.
\item The arrow set $Q_{1}$ is induced by the orientation $\sigma$ and the multiplicity $m$ of $\Gamma$. More precisely, for any half-edge $h\in H$ that is not fixed by $\sigma$ , there is an arrow $\alpha_{h}$ from $[h]$ to $[\sigma h]$. Moreover, for any half-edge $h\in H$ that is fixed by $\sigma$, there is an arrow $\alpha_{h}$ from $[h]$ to itself if and only if $m(h)>1$.

\end{itemize}

\medskip

Note that any half-edge $h\in H$ inducing an arrow $\alpha_{h}$ from $[h]$ to $[\sigma h]$ in $Q_{\Gamma}$ gives rise to an oriented cycle $C_{h}$ in $Q_{\Gamma}$ that begins and ends at $[h]$. We call such an oriented cycle a \textit{special $h$-cycle} and it is of the form $C_{h}=\alpha_{\sigma^{n-1}h}\ldots\alpha_{h}$, where $n$ is the size of the $\sigma$-orbit of $h$.

\medskip

\begin{ex} \label{ex:quiver of Brauer graph with multiplicity}
    Let us consider $\Gamma=(H,\iota,\sigma,m)$ the Brauer graph defined in Example \ref{ex:Brauer graph with multiplicity}. Then, its associated quiver $Q_{\Gamma}$ is given by 

    \smallskip

    
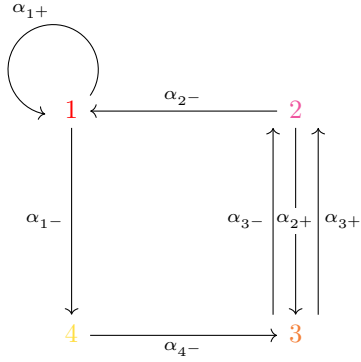
\begin{center}
        \begin{tikzcd}[column sep=2.5cm, row sep=2.5cm]
            \textcolor{red}{1} \arrow[d, phantom, ""{coordinate, name=1}]\arrow[cramped, "\alpha_{1^{+}}", to path={(\tikzcdmatrixname-1-1.40) arc(-35:260:0.6) [midway, xshift=-4mm, yshift=4mm]\tikztonodes}] \arrow[cramped]{d}[left]{\alpha_{1^{-}}} & \textcolor{VioletRed}{2} \arrow[cramped]{l}[above]{\alpha_{2^{-}}} \arrow[cramped]{d}[description]{\alpha_{2^{+}}} \\
            \textcolor{Goldenrod}{4}\arrow[cramped]{r}[below]{\alpha_{4^{-}}}&\textcolor{Orange}{3} \arrow[cramped,xshift=-3mm]{u}[left]{\alpha_{3^{-}}} \arrow[cramped, xshift=3mm]{u}[right]{\alpha_{3^{+}}}
        \end{tikzcd}
    \end{center}

    \smallskip

    \noindent In this case, the special $h$-cycles are given by $C_{1^{+}}=\alpha_{1^{+}}$,  $C_{1^{-}}=\alpha_{2^{-}}\alpha_{3^{-}}\alpha_{4^{-}}\alpha_{1^{-}}$, $C_{2^{+}}=\alpha_{3^{+}}\alpha_{2^{+}}$,  $C_{2^{-}}=\alpha_{3^{-}}\alpha_{4^{-}}\alpha_{1^{-}}\alpha_{2^{-}}$,  $C_{3^{+}}=\alpha_{2^{+}}\alpha_{3^{+}}$, $C_{3^{-}}=\alpha_{4^{-}}\alpha_{1^{-}}\alpha_{2^{-}}\alpha_{3^{-}}$ and $C_{4^{-}}=\alpha_{1^{-}}\alpha_{2^{-}}\alpha_{3^{-}}\alpha_{4^{-}}$. Note that there is no special $4^{+}$-cycle since $4^{+}$ is fixed by $\sigma$ and is of multiplicity one so it does not induce an arrow in $Q_{\Gamma}$.
\end{ex}

\medskip

\begin{defn} \label{def:Brauer graph algebra}
    Let $\Gamma=(H,\iota,\sigma,m)$ be a Brauer graph and $Q_{\Gamma}$ be its associated quiver. The \textit{Brauer graph algebra} $B_{\Gamma}$ of $\Gamma$ is the path algebra $kQ_{\Gamma}/I_{\Gamma}$ where the ideal of relations $I_{\Gamma}$ is generated by

    \smallskip

    \begin{enumerate}[label=(\Roman*)]
        \item \[(C_{h})^{m(h)}-(C_{\iota h})^{m(\iota h)}\]

        \noindent for any half-edge $h\in H$ such that $h$ and $\iota h$ both induce an arrow in $Q_{\Gamma}$.

        \item \[\alpha_{h}(C_{h})^{m(h)}\]

        \noindent for any half-edge $h\in H$ such that $h$ induces an arrow in $Q_{\Gamma}$.

        \item \[\alpha_{\iota\sigma h}\alpha_{h}\]

        \noindent for any half-edge $h\in H$ such that $h$ and $\iota\sigma h$ both induce an arrow in $Q_{\Gamma}$.
    \end{enumerate}
\end{defn}

\medskip

In general, the relations that generates the ideal of relations $I_{\Gamma}$ are not minimal as we can see in the following example.

\medskip

\begin{ex} \label{ex:Brauer graph algebra with multiplicity}
    Let us consider $\Gamma=(H,\iota,\sigma,m)$ the Brauer graph defined in Example \ref{ex:Brauer graph with multiplicity}. Then, its associated Brauer graph algebra $B_{\Gamma}$ is defined as follows : its quiver is described in Example \ref{ex:quiver of Brauer graph with multiplicity} and its ideal of relations is generated by

    \smallskip

    \begin{enumerate}[label=(\Roman*)]
    \item \label{item:Brauer graph algebra with multiplicity 1}$\alpha_{1^{+}}^{2}-\alpha_{2^{-}}\alpha_{3^{-}}\alpha_{4^{-}}\alpha_{1^{-}}$, $(\alpha_{3^{+}}\alpha_{2^{+}})^{2}-\alpha_{3^{-}}\alpha_{4^{-}}\alpha_{1^{-}}\alpha_{2^{-}}$, $(\alpha_{2^{+}}\alpha_{3^{+}})^{2}-\alpha_{4^{-}}\alpha_{1^{-}}\alpha_{2^{-}}\alpha_{3^{-}}$;
    \item \label{item:Brauer graph algebra with multiplicity 2}$\alpha_{1^{+}}^{3}$, $\alpha_{1^{-}}\alpha_{2^{-}}\alpha_{3^{-}}\alpha_{4^{-}}\alpha_{1^{-}}$, $\alpha_{2^{-}}\alpha_{3^{-}}\alpha_{4^{-}}\alpha_{1^{-}}\alpha_{2^{-}}$,  $\alpha_{3^{-}}\alpha_{4^{-}}\alpha_{1^{-}}\alpha_{2^{-}}\alpha_{3^{-}}$, $\alpha_{4^{-}}\alpha_{1^{-}}\alpha_{2^{-}}\alpha_{3^{-}}\alpha_{4^{-}}$, \newline $\alpha_{2^{+}}(\alpha_{3^{+}}\alpha_{2^{+}})^{2}$, $\alpha_{3^{+}}(\alpha_{2^{+}}\alpha_{3^{+}})^{2}$;
    \item \label{item:Brauer graph algebra with multiplicity 3}$\alpha_{1^{+}}\alpha_{2^{-}}$, $\alpha_{1^{-}}\alpha_{1^{+}}$, $\alpha_{2^{+}}\alpha_{3^{-}}$, $\alpha_{2^{-}}\alpha_{3^{+}}$, $\alpha_{3^{-}}\alpha_{2^{+}}$, $\alpha_{3^{+}}\alpha_{4^{-}}$.
    \end{enumerate}

    \smallskip

    \noindent In particular, the relations \ref{item:Brauer graph algebra with multiplicity 2} can be obtained from the relations \ref{item:Brauer graph algebra with multiplicity 1} and \ref{item:Brauer graph algebra with multiplicity 3}.
\end{ex}

\medskip

\subsection{Covering of Brauer graph algebras with multiplicity}

\medskip

In this subsection, we recall how Brauer graph algebras with multiplicity can be covered by a Brauer graph algebra with multiplicity identically one \cite{Asashiba}. We begin with defining the notion of morphisms of Brauer graphs.

\medskip

\begin{defn} \label{def:morphism of Brauer graph}
    Let $\Gamma=(H,\iota,\sigma,m)$ and $\Gamma'=(H',\iota',\sigma',m')$ be Brauer graphs. A \textit{morphism of Brauer graphs} $p:\Gamma\rightarrow \Gamma'$ is a map $p:H\rightarrow H'$ that commutes with the pairings and the orientations i.e. $p\circ \iota=\iota'\circ p$ and $p\circ\sigma=\sigma'\circ p$.
\end{defn}

\medskip

To define this covering, we will need to introduce a particular grading on the Brauer graph with multiplicity called an \textit{admissible} grading.

\medskip

\begin{defn} \label{def:graded Brauer graphs}
    Let $\Gamma=(H,\iota,\sigma,m)$ be a Brauer graph and $\overline{m}$ be the least common multiple of the $m(h)$ for $h\in H$. 

    \smallskip

    \begin{itemize}[label=\textbullet, font=\tiny]
        \item We say that a $\Z/\overline{m}\, \Z$-grading $d:H\rightarrow \Z/\overline{m}\, \Z$ is \textit{admissible} if 

        \[\sum_{h\in H, s(h)=v}d(h)=\frac{\overline{m}}{\widetilde{m}(v)}\]

        \noindent for all $v\in H/\sigma$, where $s:H\rightarrow H/\sigma$ is the source map of $\Gamma$ and $\widetilde{m}:H/\sigma\rightarrow\Z_{>0}$ is the map induced by $m$ on $H/\sigma$ (cf after Definition \ref{def:Brauer graph H}).

        \item We say that $(\Gamma,d)$ is a \textit{$\Z/\overline{m}\,\Z$-graded Brauer graph} if $d:H\rightarrow \Z/\overline{m}\,\Z$ is an admissible $\Z/\overline{m}\,\Z$-grading.
    \end{itemize}
\end{defn}

\medskip

Let $\Gamma=(H,\iota,\sigma,m)$ be a Brauer graph. Denoting by $\overline{m}$ the least common multiple of the $m(h)$ for $h\in H$, we equip $\Gamma$ with an admissible $\Z/\overline{m}\, \Z$-grading $d:H\rightarrow\Z/\overline{m}\, \Z$. One can construct a new Brauer graph $\Gamma_{d}=(H_{d},\iota_{d},\sigma_{d})$ of multiplicity identically one as follows \cite{Asashiba}

\smallskip

\begin{itemize}[label=\textbullet,font=\tiny]
\item $H_{d}=H\times(\Z/\overline{m}\, \Z)$ : an element of $H_{d}$ will be denoted $h_{i}$ for $h\in H$ and $i\in\Z/\overline{m}\, \Z$;
\item For all $h_{i}\in H_{d}$, $\iota_{d}(h_{i})=(\iota h)_{i}$;
\item For all $h_{i}\in H_{d}$, $\sigma_{d}(h_{i})=(\sigma h)_{i+d(h)}$.
\end{itemize}

\smallskip

\noindent It is clear that the projection $p_{\Gamma}:H_{d}\rightarrow H$ defines a morphism of Brauer graphs from $\Gamma_{d}$ to $\Gamma$. From now on, $\Gamma_{d}$ will be called the \textit{Galois covering} of $\Gamma$. This terminology is motivated by the following result.

\medskip

\begin{prop}[Asashiba {\cite[Proposition 1.17 and Theorem 2.11]{Asashiba}}] \label{prop:Asashiba}
Let $(Q_{\Gamma},I_{\Gamma})$ and $(Q_{d},I_{d})$ the bound quivers corresponding to $\Gamma$ and $\Gamma_{d}$ respectively. Then the morphism of bound quivers $F:(Q_{d},I_{d})\rightarrow(Q_{\Gamma},I_{\Gamma})$ arising from the morphism of Brauer graphs $p:\Gamma_{d}\rightarrow \Gamma$ defined previously is a Galois covering with group $Z/\overline{m}\, \Z$.
    
\end{prop}

\medskip

\begin{ex}\label{ex:Galois covering of Brauer graph with multiplicity}
    Let $\Gamma=(H,\iota,\sigma,m)$ be the Brauer graph defined in Example \ref{ex:Brauer graph with multiplicity}. We equip $\Gamma$ with an admissible $\Z/2\Z$-grading $d:H\rightarrow\Z/2\Z$ defined by $d(1^{+})=d(3^{+})=1$ and $d(1^{-})=d(2^{-})=d(2^{+})=d(3^{-})=d(4^{-})=d(4^{+})=0$ which can be represented on the graph as follows

    \smallskip

   
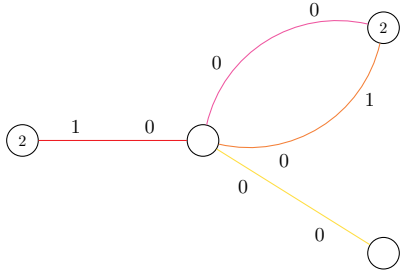
\begin{figure}[H]
        \centering
        \begin{tikzpicture}[scale=1.2]
        \tikzstyle{vertex}=[circle,draw, scale=0.6, minimum size=0.7cm]
        \node[vertex] (a) at (-1,0) {2};
        \node[vertex] (b) at (1,0) {};
        \node[vertex] (c) at (3,1.25) {2};
        \node[vertex] (d) at (3,-1.25) {};
        \draw[Red] (a)--(b) node[near end, above, scale=0.75, black]{$0$} node[near start, above, scale=0.75,black]{$1$};
        \draw[VioletRed] (b) to[bend left=45] node[near end, above, scale=0.75,black]{$0$} node[near start, above left,scale=0.75,black]{$0$} (c) ; 
        \draw[Orange] (b) to[bend right=45] node[near start, right, yshift=-2mm,scale=0.75,black]{$0$} node[near end, right, yshift=-1mm, scale=0.75,black]{$1$} (c) ;
        \draw[Goldenrod] (b)--(d) node[near start, below left, scale=0.75,black]{$0$} node[near end, below left, scale=0.75,black]{$0$};
        \end{tikzpicture}
    \end{figure}

\smallskip

\noindent Then the Galois covering $\Gamma_{d}=(H_{d},\iota_{d},\sigma_{d})$ of the Brauer graph $\Gamma$ equipped with the admissible $\Z/2\Z$-grading $d:H\rightarrow \Z/2\Z$ is given by

\smallskip

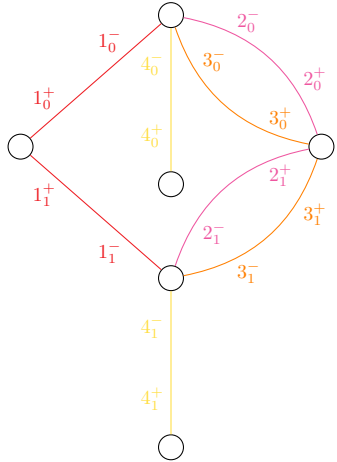
\begin{center}
    \begin{tikzpicture}
        \tikzstyle{vertex}=[circle,draw]
        \node[vertex] (a) at (-2,0) {};
        \node[vertex] (b) at (2,0) {};
        \node[vertex] (c) at (0,1.75) {};
        \node[vertex] (c') at (0,-1.75) {};
        \node[vertex] (d) at (0,-0.5) {};
        \node[vertex] (d') at (0,-4) {};
        \draw[Red] (a)--(c) node[near start, scale=0.75, above, left, yshift=2mm]{$1^{+}_{0}$} node[near end, scale=0.75, above, left, yshift=2mm]{$1^{-}_{0}$};
        \draw[Red] (a)-- (c') node[near start, scale=0.75, below, left, yshift=-2mm]{$1^{+}_{1}$} node[near end, scale=0.75, below, left, yshift=-2mm]{$1^{-}_{1}$};
        \draw[VioletRed] (c) to[bend left=30] node[near start, scale=0.75, above, right, yshift=2mm]{$2^{-}_{0}$} node[near end, scale=0.75, above, right, yshift=2mm]{$2^{+}_{0}$} (b) ;
        \draw[orange] (c) to[bend right=30] node[near start, scale=0.75, above, right, yshift=2mm]{$3^{-}_{0}$} node[near end, scale=0.75, above, right, yshift=2mm]{$3^{+}_{0}$} (b);
        \draw[VioletRed] (c') to[bend left=30]  node[near start, scale=0.75, below, right, yshift=-2mm]{$2^{-}_{1}$} node[near end, scale=0.75, below, right, yshift=-2mm]{$2^{+}_{1}$} (b);
        \draw[orange] (c') to[bend right=30]  node[near start, scale=0.75, below, right, yshift=-2mm]{$3^{-}_{1}$} node[near end, scale=0.75, below, right, yshift=-2mm]{$3^{+}_{1}$} (b);
        \draw[Goldenrod] (c)--(d) node[near start, scale=0.75, left]{$4^{-}_{0}$} node[near end,scale=0.75, left]{$4^{+}_{0}$};
        \draw[Goldenrod] (c')--(d') node[near start, scale=0.75, left]{$4^{-}_{1}$} node[near end,scale=0.75, left]{$4^{+}_{1}$};
        \end{tikzpicture}
\end{center}
\end{ex}

\medskip

In what follows, we will prove that the Brauer graph algebra $B$ associated to $\Gamma$ can be recovered from its Galois covering with the cyclic group $G=(\mathbb{C}_{\overline{m}},.)$ of order $\overline{m}$ thanks to the construction of the quiver and relations of a skew group algebra for a cyclic group \cite[Section 2.3]{RR}. From now on, we assume that $\overline{m}$ is invertible in $k$ and we fix $g$ a generator of the cyclic group $G$. Denoting by $e_{[h_{i}]}$ the idempotent in $B_{d}$ corresponding to the edge $[h_{i}]\in H_{d}/\iota_{d}$, there is a natural action of $G$ on $B_{d}=kQ_{d}/I_{d}$ given by 

\smallskip

\begin{itemize}[label=\textbullet, font=\tiny]
    \item $g\ .\ e_{[h_{i}]}=e_{[h_{i+1}]}$ for all $[h_{i}]\in H_{d}/\iota_{d}$;
    \item $g\ .\ \alpha_{h_{i}}=\alpha_{h_{i+1}}$ where $\alpha_{h_{i}}$ is the arrow in $Q_{d}$ induced by the half-edge $h_{i}\in H_{d}$. 
\end{itemize}

\smallskip

\noindent Similarly, the dual $\widehat{G}$ of $G$ is also a cyclic group. Let $\chi$ be a generator of $\widehat{G}$. Denoting by $e_{[h]}$ the idempotent in $B$ corresponding to the edge $[h]\in H/\iota$, there is a natural action of $\widehat{G}$ on $B=kQ/I$ given by

\smallskip

\begin{itemize}[label=\textbullet, font=\tiny]
    \item $\chi\ .\ e_{[h]}=e_{[h]}$ for all $[h]\in H/\iota$;
    \item $\chi\ .\ \alpha_{h}=\chi(g)^{-d(h)}\alpha_{h}$ where $\alpha_{h}$ is the arrow in $Q$ induced by the half-edge $h\in H$;
\end{itemize}

\smallskip

\noindent Note that the $G$-action on $B_{d}$ arises from a $G$-action on $Q_{d}$ which preserves the ideal of relations $I_{d}$ whereas the $\widehat{G}$-action on $B$ does not come from an action on its quiver.

\medskip

\begin{prop} \label{prop:Galois covering and skew group algebra}
    Let $f$ be the idempotent in $B_{d} \, G$ given by the sum of the $e_{[h_{0}]}\otimes 1_{G}$ for all $[h]\in H/\iota$. Then, the algebra $fB_{d}\, Gf$ has a natural $\widehat{G}$-action and the map

    \smallskip

    \begin{equation*}
        \begin{aligned}
            \phi: &&B &&&\longrightarrow &&fB_{d}\, Gf \\
            &&e_{[h]} &&&\longmapsto &&e_{[h_{0}]}\otimes 1_{G} \\
            &&\alpha_{h} &&&\longmapsto &&\beta_{h}:=\alpha_{h_{-d(h)}}\otimes g^{-d(h)}
        \end{aligned}
    \end{equation*}

    \smallskip

    \noindent is an isomorphism of algebras commuting with the $\widehat{G}$-actions.

\end{prop}

\medskip

\begin{proof}
    Since there is a $G$-action on $B_{d}$, one can naturally construct a $\widehat{G}$-action on $B_{d}\, G$ as explained after Definition \ref{def:skew group algebra}. Moreover, if $\chi$ denotes a generator of the cyclic group $\widehat{G}$, note that $\chi \ . \ f=f$. Hence, the $\widehat{G}$-action on $B_{d}\, G$ extends onto a $\widehat{G}$-action on $fB_{d}\, Gf$. Furthermore, using the construction in \cite{RR}, the quiver and relations of $fB_{d}\, G f$ are given as follows

    \smallskip

    \begin{itemize}[label=\textbullet,font=\tiny]
        \item The vertices of the quiver of $fB_{d}\, G f$ are in bijection with the idempotents $e_{[h_{0}]}\otimes 1_{G}$ for all $[h]\in H/\iota$. Moreover, the arrows are given by the $\beta_{h}:=\alpha_{h_{-d(h)}}\otimes g^{-d(h)}$ for all $h\in H$.

        \item The relations of $fB_{d}\, Gf$ are determined thanks to a representative in the $G$-orbit of the relations of $B_{d}$. Hence, its ideal of relations is generated by

        \begin{enumerate}[label=(\Roman*')]
            \item \label{item:Galois covering and skew group algebra 1}\[(\beta_{\sigma^{n-1}h}\ldots\beta_{h})^{m(h)}-(\beta_{\sigma^{n'-1}\iota h}\ldots\beta_{\iota h})^{m(\iota h)}\]

            \noindent for all $h\in H$ such that $h_{-d(h)}$ and $(\iota h)_{-d(\iota h)}$ both induce an arrow in $Q_{d}$, where $n$ and $n'$ are the size of the $\sigma$-orbit of $h$ and $\iota h$ respectively.

            \item \label{item:Galois covering and skew group algebra 2}\[\beta_{h}(\beta_{\sigma^{n-1}h}\ldots\beta_{h})^{m(h)}\]

            \noindent for all $h\in H$ such that $h_{-d(h)}$ induces an arrow in $Q_{d}$, where $n$ is the size of the $\sigma$-orbit of $h$.

            \item \label{item:Galois covering and skew group algebra 3}\[\beta_{\iota\sigma h}\beta_{h}\]

            \noindent for all $h\in H$ such that $h_{-d(h)}$ and $(\iota\sigma h)_{-d(\iota\sigma h)}$ both induce an arrow in $Q_{d}$.
            
        \end{enumerate}
    \end{itemize}

    \smallskip

    \noindent Let us detail how the relations \ref{item:Galois covering and skew group algebra 1} are obtained. For any generator $R$ of $I_{d}$, the ideal of relations of $B_{d}$, there exists a unique relation $R'$ in the $G$-orbit of $R$ which is of the form $R':g^{t}(e_{[h_{0}]})\rightarrow e_{[h'_{0}]}$ for some $0\leq t<\overline{m}$. In this case, the ideal of relations of $fB_{d}\, G f$ is generated by elements of the form $f(R'\otimes g^{t})f$. We recall that the type \ref{item:Brauer graph algebra with multiplicity 1} relation generating $I_{d}$ is of the form

    \[R=\alpha_{\sigma_{d}^{n_{d}-1}h_{i}}\ldots\alpha_{h_{i}}-\alpha_{\sigma_{d}^{n'_{d}-1}\iota_{d}h_{i}}\ldots\alpha_{\iota_{d}h_{i}}\]

    \noindent for all $h_{i}\in H_{d}$ such that $h_{i}$ and $\iota_{d}h_{i}$ both induce an arrow in $Q_{d}$, where $n_{d}$ and $n'_{d}$ denotes the size of the $\sigma_{d}$-orbit of $h_{i}$ and $\iota_{d} h_{i}$ respectively. By construction of $\Gamma_{d}$, $n_{d}=nm(h)$ and $n'_{d}=n'm(\iota h)$ where $n$ and $n'$ denotes the size of the $\sigma$-orbit of $h$ and $\iota h$ respectively. Moreover, since $d:H\rightarrow \Z/\overline{m}\, \Z$ is admissible, the unique representative $R'$ in the $G$-orbit of $R$ of the form $R':g^{t}(e_{[h_{0}]})\rightarrow e_{[h_{0}]}$ for some $0\leq t<\overline{m}$ is 

    \[R'=\alpha_{\sigma_{d}^{nm(h)-1}h_{0}}\ldots\alpha_{h_{0}}-\alpha_{\sigma_{d}^{nm(\iota h)-1}\iota_{d}h_{0}}\ldots\alpha_{\iota_{d}h_{0}}\]

    \noindent where $t=0$ in this case. Using the definition of $\sigma_{d}$, note that

    \[\beta_{\sigma^{n-1}h}\ldots\beta_{h}=\alpha_{\sigma_{d}^{n-1}h_{-d(\sigma^{n-1}h)}}\ldots\alpha_{h_{-\sum_{i=0}^{n-1}d(\sigma^{i}h)}}\otimes g^{-\sum_{i=0}^{n-1}d(\sigma^{i}h)}\]

    \noindent Hence, one can check that $f(R'\otimes 1_{G})f$ is indeed given by the relations \ref{item:Galois covering and skew group algebra 1}. The relations \ref{item:Galois covering and skew group algebra 2} and \ref{item:Galois covering and skew group algebra 3} are obtained similarly from the type \ref{item:Brauer graph algebra with multiplicity 2} and type \ref{item:Brauer graph algebra with multiplicity 3} relation in $B_{d}$ respectively.

    \medskip
    
    It is clear that $\phi$ defines an isomorphism between the quivers of $B$ and $fB_{d}\, G f$ which clearly extends on an isomorphism of algebras between the path algebras of these quivers. Moreover, since $h_{i}$ induces an arrow in $Q_{d}$ for $i\in\Z/\overline{m}\,\Z$ if and only if $h$ induces an arrow in $Q_{\Gamma}$, it is clear that the ideals of relations of $B$ and $fB_{d}\, G f$ coincide via the previous isomorphism of path algebras. Thus, $\phi$ defines an isomorphism of algebras between $B$ and $fB_{d}\, Gf$. Moreover, one can easily check that this isomorphism commutes with the $\widehat{G}$-action on these two algebras. 
\end{proof}

\medskip

\subsection{Generalized Kauer moves}

\medskip

Our goal is to define generalized Kauer moves for Brauer graph with arbitrary multiplicity so that they can be understood in terms of silting mutations as for the multiplicity one case \cite[Theorem 3.10]{Soto}. The idea of the proof is to use the Galois covering defined in the previous subsection which can be constructed whenever the Brauer graph is equipped with an admissible grading. Hence, we need to define a graded version of the generalized Kauer moves for Brauer graph with multiplicity. As in \cite{Soto}, we begin with defining these for successive half-edges called \textit{sectors}.

\medskip

\begin{defn} \label{def:sector}
    Let $\Gamma=(H,\iota,\sigma,m)$ be a Brauer graph and $H'\subset H$ stable under $\iota$.

    \smallskip

    \begin{itemize}[label=\textbullet, font=\tiny]
    \item We say that $(h,r)\in H\times\Z_{\ge0}$ is a \textit{sector} in $\Gamma$ of elements in $H'$ if $r+1$ is the smallest integer $r'\ge 0$ such that $\sigma^{r'}h\notin H'$.
    \item We say that $(h,r)\in H\times\Z_{\ge 0}$ is a \textit{maximal sector} in $\Gamma$ of elements in $H'$ if it is a sector satisfying that $\sigma^{-1}h\notin H'$.
        
    \end{itemize}

    \smallskip

    \noindent We denote respectively by $\mathrm{sect}(H',\sigma)$ and $\mathrm{Sect}(H',\sigma)$ the set of sectors and maximal sectors in $\Gamma$ of elements in $H'$.
\end{defn}

\medskip

From now on, if $m:H\rightarrow \Z_{>0}$ is the multiplicity of a Brauer graph $\Gamma=(H,\iota,\sigma,m)$, we denote by $\overline{m}$ the least common multiple of the $m(h)$ for $h\in H$.

\medskip

\begin{defn} \label{def:graded generalized Kauer moves with multiplicity of a sector}

Let $(\Gamma,d)=(H,\iota,\sigma,m,d)$ be a $\Z/\overline{m}\, \Z$-graded Brauer graph and $H'$ be a subset of $H$ stable under $\iota$. The \textit{$\Z/\overline{m}\, \Z$-graded generalized Kauer move} of a sector $(h,r)\in\mathrm{sect}(H',\sigma)$ in $\Gamma$ gives rise to a $\Z/\overline{m}\, \Z$-Brauer graph $\mu^{+}_{(h,r)}(\Gamma,d)=(H,\iota,\sigma_{(h,r)},m_{(h,r)},d_{(h,r)})$ where

\begin{equation*}
    \begin{aligned}
        \sigma_{(h,r)}=(h \ \sigma^{r+1}h)\sigma(\sigma^{r}h \ \iota\sigma^{r+1}h) \qquad \mbox{and} \qquad 
        m_{(h,r)}:  &&H &&&\rightarrow &&\Z_{>0} \\
        &&\sigma^{i}h &&&\mapsto &&m(\iota\sigma^{r+1}h) \quad \mbox{for $i=0,\ldots,r$} \\
        &&h' &&&\mapsto &&m(h') \quad \mbox{for $h'\neq\sigma^{i}h$}
    \end{aligned}
\end{equation*}

\noindent and where the grading $d_{(h,r)}:H\rightarrow \Z/\overline{m}\, \Z$ is defined by

\begin{equation*}
    \begin{aligned}
        d_{(h,r)}: \ &\iota\sigma^{r+1}h &&\mapsto &&-\sum_{i=0}^{r}d(\sigma^{i}h) \\
        &\sigma^{r}h &&\mapsto &&\left\{\begin{aligned}
            &d(\iota\sigma^{r+1}h)+d(\sigma^{r}h) &&\mbox{if $\iota\sigma^{r+1}h\neq\sigma^{-1}h$} \\
            &\sum_{i=-1}^{r}d(\sigma^{i}h)+d(\sigma^{r}h) &&\mbox{else}
        \end{aligned}\right. \\
        &\sigma^{-1}h &&\mapsto &&\left\{\begin{aligned}
            &\sum_{i=-1}^{r}d(\sigma^{i}h) &&\mbox{if $\iota\sigma^{r+1}h\neq\sigma^{-1}h$} \\
            &-\sum_{i=0}^{r}d(\sigma^{i}h) &&\mbox{else}   
        \end{aligned}\right. \\
        &h' &&\mapsto &&d(h') \qquad \mbox{for $h'\neq\iota\sigma^{r+1}h,\sigma^{r}h,\sigma^{-1}h$}
    \end{aligned}
\end{equation*}

\smallskip

\noindent The underlying Brauer graph $\mu^{+}_{(h,r)}(\Gamma)=(H,\iota,\sigma_{(h,r)},m_{(h,r)})$ can be obtained from $\Gamma$ as follows.

\smallskip

\begin{figure}[H]
    \centering
              
    \begin{tikzpicture}[scale=0.9]
        \tikzstyle{vertex}=[draw, circle, scale=0.6, minimum size=0.8cm]
        \begin{scope}[xshift=-3cm]
        \node[vertex] (1) at (-2,0) {m};
        \node[vertex] (2) at (2,0) {m'};
        \draw (1)--(2) node[near start, above]{$\sigma^{r+1}h$} node[near end, above]{$\iota\sigma^{r+1}h$};
        \draw (1)--(-2.5,-1.75) node[below]{$\sigma^{-1}h$};
        \draw[Orange] (1)--(-1.5,-1.75) node[below, right]{$h$};
        \draw[Orange] (1)--(-1,-1.25) node[below, right]{$\sigma^{j}h$};
        \draw[Orange] (1)--(-0.5,-0.75) node[below, right]{$\sigma^{r}h$};
        \draw (2)--(2.5,-1.75) node[below]{$\sigma\iota\sigma^{r+1}h$};
        \draw[|-|, Orange] (-1.5,-2.2)--(0,-2.2) node[below, midway]{$(h,r)$};
        \draw (0,-3.5) node{$\Gamma=(H,\iota,\sigma,m)$};
        \end{scope} 
        
        \draw[->] (0.5,0)--(3,0);

        \begin{scope}[xshift=6.5cm]
        \node[vertex] (1) at (-2,0) {m};
        \node[vertex] (2) at (2,0) {m'};
        \draw (1)--(2) node[near start, above]{$\sigma^{r+1}h$} node[near end, above]{$\iota\sigma^{r+1}h$};
        \draw (1)--(-2.5,-1.75) node[below]{$\sigma^{-1}h$};
        \draw[Orange] (2)--(1.5,-1.75) node[below left]{$\sigma^{r}h$};
        \draw[Orange] (2)--(1,-1.25) node[below, left]{$\sigma^{j}h$};
        \draw[Orange] (2)--(0.5,-0.75) node[below, left]{$h$};
        \draw (2)--(2.5,-1.75) node[below]{$\sigma\iota\sigma^{r+1}h$};
        \draw (0,-3.5) node{$\mu^{+}_{(h,r)}(\Gamma)=(H,\iota,\sigma_{(h,r)},m_{(h,r)})$};
        \end{scope}
        \end{tikzpicture}
    \caption{Generalized Kauer move of a sector (h,r)}
    \label{Generalized Kauer move of a sector}
\end{figure}
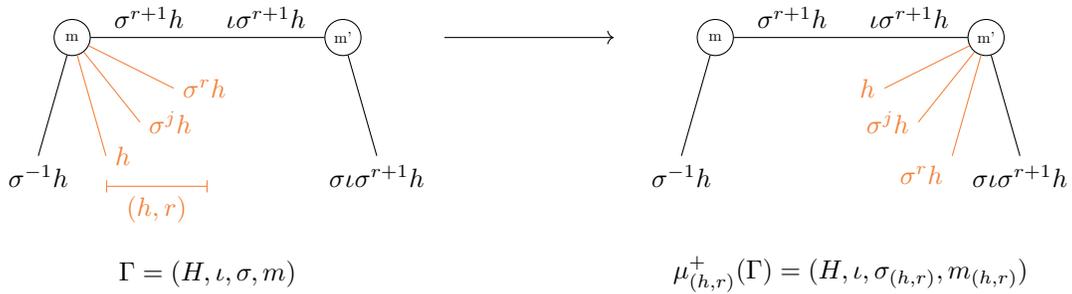
\end{defn}

\medskip

\begin{rem} \label{rem:special case of generalized Kauer move of a sector}
    In the special case where $\iota\sigma^{r+1}h=\sigma^{-1}h$, then the $\Z/\overline{m}\,\Z$-graded generalized Kauer move of the sector $(h,r)$ does not change the underlying Brauer graph which is given as follows

    \smallskip

    
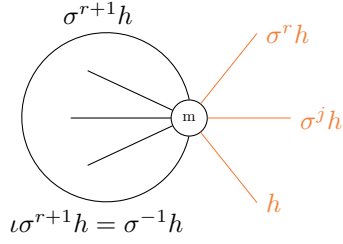
\begin{center}
    \begin{tikzpicture}[scale=0.9]
        \tikzstyle{vertex}=[draw, circle, scale=0.6, minimum size=0.8cm]

        \node[vertex] (1) at (-2,0) {m};
        \draw (1.90) arc(12:347:1.25) node[near start, above]{$\sigma^{r+1}h$} node[near end, below]{$\iota\sigma^{r+1}h=\sigma^{-1}h$};
        \draw[Orange] (1)--(-1,-1.25) node[below, right]{$h$};
        \draw[Orange] (1)--(-0.5,0) node[below, right]{$\sigma^{j}h$};
        \draw[Orange] (1)--(-1,1.25) node[below, right]{$\sigma^{r}h$};
        \draw (-2,-2.75) node{$\Gamma=\mu^{+}_{(h,r)}(\Gamma)$};
        \draw (1)--(-3.5,0.7);
        \draw (1)--(-3.75,0);
        \draw (1)--(-3.5,-0.7);
        \end{tikzpicture}
        \end{center}

        \smallskip

        \noindent However, the degree of the $\sigma^{i}h$ will a priori change in the process. In particular, the Galois covering of $\Gamma$ and $\mu^{+}_{(h,r)}(\Gamma)$ could have different orientation, even in this case.
\end{rem}

By construction of the generalized Kauer move described in the previous picture, it is clear that $m_{(h,r)}$ is constant on the $\sigma_{(h,r)}$-orbits. Moreover, $d_{(h,r)}:H\rightarrow\Z/\overline{m}\, \Z$ is admissible. Indeed, one can construct a bijection $\phi:H/\sigma\rightarrow H/\sigma_{(h,r)}$ satisfying

\begin{equation*}
    \begin{aligned}
        &s_{(h,r)}(\sigma^{i}h)=\phi(s(\iota\sigma^{r+1}h)) &&\mbox{for $i=0,\ldots, r$} \\
        &s_{(h,r)}(h')=\phi(s(h')) &&\mbox{for $h'\neq \sigma^{i}h, i=0\ldots,r$}
    \end{aligned}
\end{equation*}

\noindent where $s:H\rightarrow H/\sigma$ and $s_{(h,r)}:H\rightarrow H/\sigma_{(h,r)}$ are the source map of $\Gamma$ and $\mu^{+}_{(h,r)}(\Gamma)$ respectively. Then, one can easily check that, for all vertex $v\in H/\sigma_{(h,r)}$ in $\mu^{+}_{(h,r)}(\Gamma)$, we have

\[\sum_{h'\in H, s_{(h,r)}(h')=v}d_{(h,r)}(h')=\sum_{h'\in H, s(h')=\phi^{-1}(v)}d(h')=\frac{\overline{m}}{\widetilde{m}(\phi^{-1}(v))}\]

\noindent where $\widetilde{m}:H/\sigma\rightarrow \Z_{>0}$ is induced by $m$ (cf after Definition \ref{def:Brauer graph H}). Note that $\overline{m_{(h,r)}}=\overline{m}$. Moreover, one can easily check that $\widetilde{m}(\phi^{-1}(v))=\widetilde{m}_{(h,r)}(v)$ for all $v\in H/\sigma_{(h,r)}$ where $\widetilde{m}_{(h,r)}:H/\sigma_{(h,r)}\rightarrow \Z_{>0}$ is induced by $m_{(h,r)}$. Hence, $d_{(h,r)}:H\rightarrow \Z/\overline{m}\, \Z$ is indeed admissible.

\medskip

\begin{rem}
    By definition, if $\Gamma$ is a Brauer graph of multiplicity identically one then $\mu^{+}_{(h,r)}(\Gamma)$ is also a Brauer graph of multiplicity identically one. Hence, this definition generalizes Definition 2.2 in \cite{Soto}.
\end{rem}

\medskip

Thanks to Lemma 2.4 of \cite{Soto}, we can consider successive $\Z/\overline{m}\, \Z$-graded generalized Kauer moves. We denote by

\[\mu^{+}_{(h_{1},r_{1})(h_{2},r_{2})}(\Gamma,d)=\left(H,\iota,\sigma_{(h_{1},r_{1})(h_{2},r_{2})},m_{(h_{1},r_{1})(h_{2},r_{2})}, d_{(h_{1},r_{1})(h_{2},r_{2})}\right)\]

\noindent the $\Z/\overline{m}\, \Z$-graded Brauer graph defined by $\mu^{+}_{(h_{1},r_{1})}(\mu^{+}_{(h_{2},r_{2})}(\Gamma,d))$.

\medskip

\begin{prop} \label{prop:commutativity of generalized Kauer moves of sectors}
    Let $\Gamma=(H,\iota,\sigma,m)$ be a Brauer graph and $H'$ be a subset of $H$ stable under $\iota$. Let $(h_{1},r_{1})$,$(h_{2},r_{2})$ be two distinct maximal sectors in $\Gamma$ of elements in $H'$. For any admissible $\Z/\overline{m}\, \Z$-grading $d:H\rightarrow\Z/\overline{m}\, \Z$ of $\Gamma$, we have

    \[\mu^{+}_{(h_{1},r_{1})(h_{2},r_{2})}(\Gamma,d)=\mu^{+}_{(h_{2},r_{2})(h_{1},r_{1})}(\Gamma,d)\]

\end{prop}

\medskip

\begin{proof}
    Thanks to Proposition 2.5 in \cite{Soto}, it only remains to prove that $m_{(h_{1},r_{1})(h_{2},r_{2})}=m_{(h_{2},r_{2})(h_{1},r_{1})}$. It is clear that this equality holds for $h'$ distinct from the $\sigma^{i}h_{1}$ for $i=0,\ldots,r_{1}$ and from the $\sigma^{i}h_{2}$ for $i=0,\ldots,r_{2}$. Since $h_{1}$ and $h_{2}$ have a symmetric role, it suffices to check the equality for the $\sigma^{i}h_{1}$, $i=0,\ldots,r_{1}$. On the one hand, we have

    \[m_{(h_{1},r_{1})(h_{2},r_{2})}(\sigma^{i}h_{1}) =m_{(h_{1},r_{1})(h_{2},r_{2})}(\sigma_{(h_{2},r_{2})}^{i}h_{1})=m_{(h_{2},r_{2})}(\iota\sigma_{(h_{2},r_{2})}^{r_{1}+1}h_{1})=m(\iota\sigma^{r_{1}+1}h_{1})\]

    \noindent On the other hand, since $\sigma^{i}h_{1}$ is distinct from $\sigma^{j}_{(h_{1},r_{1})}h_{2}=\sigma^{j}h_{2}$ for $j=0,\ldots,r_{2}$, we have

    \[m_{(h_{2},r_{2})(h_{1},r_{1})}(\sigma^{i}h_{1})=m_{(h_{1},r_{1})}(\sigma^{i}h_{1})=m(\iota\sigma^{r_{1}+1}h_{1})\] 
\end{proof}

\medskip

\begin{defn} \label{def:graded generalized Kauer move with multiplicity of H'}
    Let $(\Gamma,d)=(H,\iota,\sigma,m,d)$ be a $\Z/\overline{m}\,\Z$-Brauer graph and $H'$ be a subset of $H$ stable under $\iota$. The \textit{$\Z/\overline{m}\,\Z$-graded generalized Kauer move} of $H'$ is the succession of the $\Z/\overline{m}\,\Z$-graded generalized Kauer move of all the maximal sectors $(h,r)\in\mathrm{Sect}(H',\sigma)$ in $\Gamma$. This gives rise to a $\Z/\overline{m}\,\Z$-graded Brauer graph that will be denoted by $\mu^{+}_{H'}(\Gamma,d)=(H,\iota,\sigma_{H'},m_{H'},d_{H'})$.
\end{defn}

\medskip

\begin{ex} \label{ex:generalized Kauer move with multiplicity}
    Let us consider $\Gamma=(H,\iota,\sigma,m,d)$ the $\Z/2\Z$-graded Brauer graph defined in Example \ref{ex:Galois covering of Brauer graph with multiplicity}. Let $H'=\{1^{+}, 1^{-},2^{+}, 2^{-}\}$ be a subset of $H$ stable under $\iota$. Then, the $\Z/2\Z$-graded Brauer graph $\mu^{+}_{H'}(\Gamma,d)=(H,\iota,\sigma_{H'},m_{H'},d_{H'})$ obtained from the $\Z/2\Z$-graded generalized Kauer move of $H'$ is given by

    \smallskip

    
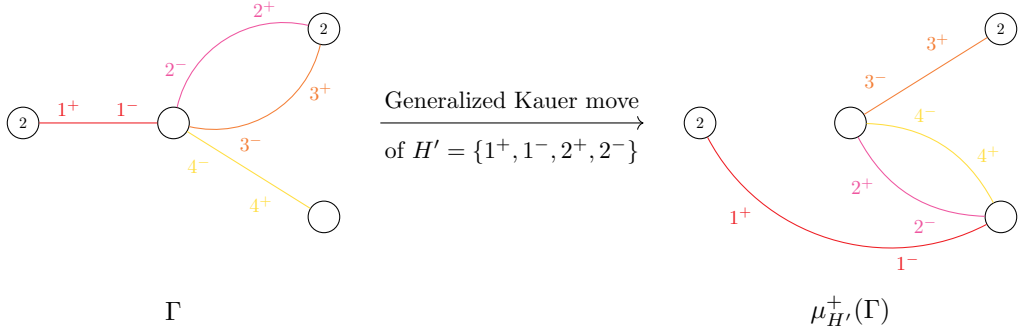
\begin{figure}[H]
        \centering
        \begin{tikzpicture}
        \tikzstyle{vertex}=[circle,draw, scale=0.6, minimum size=0.7cm]
        \begin{scope}[xshift=-4cm]
            \node[vertex] (a) at (-1,0) {2};
        \node[vertex] (b) at (1,0) {};
        \node[vertex] (c) at (3,1.25) {2};
        \node[vertex] (d) at (3,-1.25) {};
        \draw[Red] (a)--(b) node[near end, above, scale=0.75]{$1^{-}$} node[near start, above, scale=0.75]{$1^{+}$};
        \draw[VioletRed] (b) to[bend left=45] node[near end, above, scale=0.75]{$2^{+}$} node[near start,left,scale=0.75]{$2^{-}$} (c) ; 
        \draw[Orange] (b) to[bend right=45] node[near start, right, yshift=-2mm,scale=0.75]{$3^{-}$} node[near end, right, yshift=-1mm, scale=0.75]{$3^{+}$} (c) ;
        \draw[Goldenrod] (b)--(d) node[near start, below left, scale=0.75]{$4^{-}$} node[near end, below left, scale=0.75]{$4^{+}$};
        \draw (1,-2.5) node{$\Gamma$};
        \end{scope}
        \draw[->] (-0.25,0)--(3.25,0) node[midway,above, scale=0.9, yshift=1mm]{Generalized Kauer move} node[midway,below, yshift=-1mm,scale=0.9]{of $H'=\{1^{+},1^{-},2^{+},2^{-}\}$};
        \begin{scope}[xshift=5cm]
        \node[vertex] (a) at (-1,0) {2};
        \node[vertex] (b) at (1,0) {};
        \node[vertex] (c) at (3,1.25) {2};
        \node[vertex] (d) at (3,-1.25) {};
        \draw[Red] (a) to[bend right=45] node[near end, below, scale=0.75]{$1^{-}$} node[near start, below left, scale=0.75]{$1^{+}$} (d) ;
        \draw[VioletRed] (b) to[bend right=30] node[near end, below left, scale=0.75]{$2^{-}$} node[near start, below left,scale=0.75]{$2^{+}$} (d) ; 
        \draw[Orange] (b) -- node[near start, above left,scale=0.75]{$3^{-}$} node[near end, above left, scale=0.75]{$3^{+}$} (c) ;
        \draw[Goldenrod] (b) to[bend left=30] node[near start, above right, scale=0.75]{$4^{-}$} node[near end, above right, scale=0.75]{$4^{+}$} (d) ;
        \draw (1,-2.5) node{$\mu^{+}_{H'}(\Gamma)$};
        \end{scope}
        \end{tikzpicture}
    \end{figure}

    \smallskip

    \noindent where the orientation $\sigma_{H'}$ is $(1^{-}\ 4^{+} \ 2^{-})(2^{+}\ 4^{-} \ 3^{-})=(2^{-} \ 4^{-})(2^{+} \ 3^{+})\sigma(1^{-} \ 4^{+})(2^{+} \ 3^{-})$, the multiplicity $m_{H'}:H\rightarrow\Z_{>0}$ is given by $m_{H'}(1^{+})=m_{H'}(3^{+})=2$ and $m_{H'}(1^{-})=m_{H'}(2^{-})=m_{H'}(4^{+})=m_{H'}(2^{+})=m_{H'}(4^{-})=m_{H'}(3^{-})=1$ and the admissible $\Z/2\Z$-grading $d_{H'}:H\rightarrow\Z/2\Z$ is given by $d_{H'}(1^{+})=d_{H'}(3^{+})=1$ and $d_{H'}(1^{-})=d_{H'}(2^{-})=d_{H'}(4^{+})=d_{H'}(2^{+})=d_{H'}(4^{-})=d_{H'}(3^{-})=0$.
\end{ex}

\medskip

Considering a $\Z/\overline{m}\,\Z$-graded generalized Kauer move of a $\Z/\overline{m}\,\Z$-graded Brauer graph $(\Gamma,d)$, one can construct a Galois covering with group $\Z/\overline{m}\,\Z$ for $\mu_{H'}^{+}(\Gamma,d)$ as defined before Proposition \ref{prop:Asashiba}. The following proposition shows a commutativity between constructing this Galois covering and applying a generalized Kauer move.

\medskip

\begin{prop} \label{prop:commutativity Galois covering and generalized Kauer move with multiplicity}
    Let $(\Gamma,d)=(H,\iota,\sigma,m,d)$ be a $\Z/\overline{m}\,\Z$-graded Brauer graph and $H'$ be a subset of $H$ stable under $\iota$. We denote by $\Gamma_{d}=(H_{d},\iota_{d},\sigma_{d})$ the Galois covering of $\Gamma$ constructed before Proposition \ref{prop:Asashiba} and $H'_{d}=H'\times\Z/\overline{m}\,\Z\subset H_{d}$. Then, the Galois covering of $\mu^{+}_{H'}(\Gamma,d)$ is the Brauer graph $\mu^{+}_{H'_{d}}(\Gamma_{d})$.

\end{prop}

\medskip

The previous proposition can be summarized in the following commutative diagram 

\smallskip

\begin{center}
    \begin{tikzcd}[column sep=5cm, row sep=2cm]
        \Gamma_{d} \arrow{r}[above]{\mbox{generalized Kauer}}[below]{\mbox{ move of $H'_{d}$}} \arrow{d}[left]{p_{\Gamma}} &\mu^{+}_{H'_{d}}(\Gamma_{d}) \arrow{d}[right]{p_{\mu^{+}_{H'}(\Gamma)}} \\
        (\Gamma,d) \arrow{r}[above]{\mbox{$\Z/\overline{m}\,\Z$-graded generalized}}[below]{\mbox{Kauer move of $H'$}}&\mu^{+}_{H'}(\Gamma,d)
    \end{tikzcd}
\end{center}

\smallskip

\noindent where $p_{\Gamma}:\Gamma_{d}\rightarrow\Gamma$ and $p_{\mu^{+}_{H'}(\Gamma)}:\mu^{+}_{H'_{d}}(\Gamma_{d})\rightarrow\mu^{+}_{H'}(\Gamma)$ are the morphisms of Brauer graphs induced by the natural projection and $\mu^{+}_{H'}(\Gamma)$ denotes the underlying Brauer graph of $\mu^{+}_{H'}(\Gamma,d)$.

\medskip

\begin{proof}
    It is easy to see that $H'_{d}$ is a subset of $H_{d}$ stable under $\iota_{d}$ and that $(h,r)\in\mathrm{Sect}(H',\sigma)$ is a maximal sector in $\Gamma$ if and only if $(h_{i},r)\in\mathrm{Sect}(H'_{d},\sigma_{d})$ is a maximal sector in $\Gamma_{d}$ for all $i\in\Z/\overline{m}\,\Z$. Thanks to Proposition \ref{prop:commutativity of generalized Kauer moves of sectors}, it suffices to prove that the Galois covering of $\mu^{+}_{(h,r)}(\Gamma,d)$ is given by 
    
    \[\mu^{+}_{p_{\Gamma}^{-1}(h,r)}(\Gamma_{d})=(H_{d},\iota_{d},(\sigma_{d})_{p_{\Gamma}^{-1}(h,r)})\] 
    
    \noindent for any maximal sector $(h,r)\in\mathrm{Sect}(H',\sigma)$, where $p_{\Gamma}^{-1}(h,r)$ is the product of the maximal sectors $(h_{i},r)$ for $i\in\Z/\overline{m}\,\Z$. We denote by
    
     \[\mu^{+}_{(h,r)}(\Gamma)_{d_{(h,r)}}=(H_{d_{(h,r)}},\iota_{d_{(h,r)}},(\sigma_{(h,r)})_{d_{(h,r)}})\]
    
    \noindent the Galois covering of $\mu^{+}_{(h,r)}(\Gamma,d)$. Since $\overline{m_{(h,r)}}=\overline{m}$,  the Brauer graphs $\mu^{+}_{p_{\Gamma}^{-1}(h,r)}(\Gamma_{d})$ and $\mu^{+}_{(h,r)}(\Gamma)_{d_{(h,r)}}$ have the same set of half-edges and the same pairings. It remains to prove the equality of the orientations. By definition, we have

    \smallskip

    \begin{equation*}
        \begin{gathered}
            (\sigma_{d})_{p_{\Gamma}^{-1}(h,r)}=\left(\prod_{i\in\Z/\overline{m}\,\Z} (h_{i} \quad \sigma_{d}^{r+1}h_{i}) \right) \ \sigma_{d} \ \left(\prod_{i\in\Z/\overline{m}\,\Z} (\sigma_{d}^{r}h_{i} \quad \iota_{d}\sigma_{d}^{r+1}h_{i}) \right) \\[0.5cm]
            (\sigma_{(h,r)})_{d_{(h,r)}}=\left[(h \ \ \sigma^{r+1}h)\ \sigma\ (\sigma^{r}h \ \ \iota\sigma^{r+1}h)\right]_{d_{(h,r)}}
        \end{gathered}
    \end{equation*}

    \smallskip

    \noindent It is clear that for any half-edge distinct from the $\sigma_{d}^{-1}h_{i}$, $\iota_{d}\sigma_{d}^{r+1}h_{i}$ and $\sigma_{d}^{r}h_{i}$, $i\in \Z/\overline{m}\, \Z$, the equality holds. We only detail the computations for $\sigma_{d}^{r}h_{i}$, the other cases being similar. Then,

    \smallskip

    \begin{equation*}
        \begin{aligned}
            (\sigma_{d})_{p_{\Gamma}^{-1}(h,r)}(\sigma_{d}^{r}h_{i})
            &=\left\{\begin{aligned}
                &\sigma_{d}^{r+1}h_{j} &&\mbox{if $\sigma_{d}\iota_{d}\sigma_{d}^{r+1}h_{i}=h_{j}$ for some $j\in\Z/\overline{m}\,\Z$} \\
                &\sigma_{d}\iota_{d}\sigma_{d}^{r+1}h_{i} &&\mbox{else}
            \end{aligned} \right. \\
            &=\left\{\begin{aligned}
                &(\sigma^{r+1}h)_{i+2\sum_{k=0}^{r}d(\sigma^{k}h)+d(\iota\sigma^{r+1}h)} &&\mbox{if $\sigma\iota\sigma^{r+1}h=h$} \\
                &(\sigma\iota\sigma^{r+1}h)_{i+\sum_{k=0}^{r}d(\sigma^{k}h)+d(\iota\sigma^{r+1}h)} &&\mbox{else}
            \end{aligned} \right.
        \end{aligned}
    \end{equation*}

    \smallskip

    \noindent since the first condition is equivalent to $\sigma\iota\sigma^{r+1}h=h$ and in this case $j=i+\sum_{k=0}^{r}d(\sigma^{k}h)+d(\iota\sigma^{r+1}h)$. On the other hand, we have

    \smallskip

    \begin{equation*}
        \begin{aligned}
            (\sigma_{(h,r)})_{d_{(h,r)}}(\sigma_{d}^{r}h_{i})
            &=(\sigma_{(h,r)})_{d_{(h,r)}}((\sigma^{r}h)_{i+\sum_{k=0}^{r-1}d(\sigma^{k}h)}) \\
            &=(\sigma_{(h,r)}(\sigma^{r}h))_{i+\sum_{k=0}^{r-1}d(\sigma^{k}h)+d_{(h,r)}(\sigma^{r}h)} \\
            &=\left\{\begin{aligned}
                &(\sigma^{r+1}h)_{i+\sum_{k=0}^{r-1}d(\sigma^{k}h)+\sum_{i=-1}^{r}d(\sigma^{k}h)+d(\sigma^{r}h)} &&\mbox{if $\sigma\iota\sigma^{r+1}h=h$} \\
                &(\sigma\iota\sigma^{r+1}h)_{i+\sum_{k=0}^{r-1}d(\sigma^{k}h)+d(\iota\sigma^{r+1}h)+d(\sigma^{r}h)} &&\mbox{else}
            \end{aligned} \right. \\
            &=(\sigma_{d})_{p_{\Gamma}^{-1}(h,r)}(\sigma_{d}^{r}h_{i})
        \end{aligned}
    \end{equation*}
\end{proof}

\medskip

\subsection{Compatibility with silting mutations}

\medskip

The goal of this part is to prove the following theorem which is an analogous version of Theorem 3.10 in \cite{Soto} for the case of Brauer graphs with arbitrary multiplicity. 

\medskip

\begin{thm} \label{thm:generalized Kauer moves with multiplicity}
    Let $\Gamma=(H,\iota,\sigma,m)$ be a Brauer graph and $H'$ be a subset of $H$ stable under $\iota$. We assume that $\overline{m}$ is invertible in $k$. Denoting by $B$ and $B'$ the Brauer graph algebras associated respectively to $\Gamma$ and $\mu^{+}_{H'}(\Gamma)$, there is an equivalence of triangulated categories

    \smallskip

    \begin{center}
        \begin{tikzcd}[column sep=3cm]
            \per{B'} \arrow{r}[above]{-\lotimes{B'}{\ \mu^{+}(B\, ;\, e_{H''}B)}} &\per{B}
        \end{tikzcd}
    \end{center}

    \smallskip

    \noindent where $e_{H''}B$ is the projective $B$-module corresponding to the edges in $(H\backslash H')/\iota$.
\end{thm}

\medskip

To prove this theorem, we will need the following result. 

\medskip

\begin{prop} \label{prop:generalized Kauer move and skew group algebra}
    Let us consider the setting of Proposition \ref{prop:commutativity Galois covering and generalized Kauer move with multiplicity}. We assume that $\overline{m}$ is invertible in $k$. We denote by $B_{d}$ and $B_{d}'$ the Brauer graph algebras associated to $\Gamma_{d}$ and $\mu^{+}_{H'_{d}}(\Gamma_{d})$ respectively. Moreover, let $G$ be the cyclic group $(\mathbb{C}_{\overline{m}},.)$ of order $\overline{m}$. Then, there is an equivalence of triangulated categories

    \smallskip

    \begin{center}
        \begin{tikzcd}[column sep=4.5cm]
            \per{B_{d}'\, G} \arrow{r}[above]{-\lotimes{B_{d}'\, G}{\ \mu^{+}(B_{d}\, G \, ;\, e_{H''_{d}}B_{d}\, G)}} &\per{B_{d}\, G}
        \end{tikzcd}
    \end{center}

    \smallskip

    \noindent where $e_{H''_{d}}B_{d}$ denotes the projective $B_{d}$-module corresponding to the edges in $(H_{d}\backslash H'_{d})/\iota_{d}$.
\end{prop}

\medskip

\begin{proof}
    By Theorem 3.10 in \cite{Soto}, we know that there is a triangle equivalence

    \smallskip

    \begin{center}
        \begin{tikzcd}[column sep=2cm]
        \per{B'_{d}} \arrow[cramped]{r}[above]{-\lotimes{B_{d}'}{\, T_{d}}} &\per{B_{d}}
        \end{tikzcd}
    \end{center}

    \smallskip

    \noindent where $T_{d}:=\mu^{+}(B_{d}\,;\,e_{H_{d}''}B_{d})$. Let us prove that $T_{d}$ is $G$-invariant. In this case, we obtain the following equivalence of triangulated categories thanks to Theorem 2.10 in \cite{AB} 

    \smallskip

    \begin{center}
        \begin{tikzcd}[column sep=3.5cm]
            \per{B'_{d}\, G} \arrow[cramped]{r}[above]{-\lotimes{B_{d}'\, G}{\ (T_{d}\, \lotimes{B_{d}}{\, B_{d}\, G}})} &\per{B_{d}\, G}
        \end{tikzcd}
    \end{center}

    \smallskip

    \noindent For all $h_{i}\in H'_{d}$, we denote by
    
    \[\alpha(h_{i}, H'_{d}):e_{[h_{i}]}B_{d}\longrightarrow e_{[\sigma_{d}^{r(h_{i})+1}h_{i}]}B_{d}\]
    
    \noindent the morphism induced by the path in $B_{d}$ from $h_{i}$ to $\sigma_{d}^{r(h_{i})+1}h_{i}$ where $r(h_{i})+1=\min\{r\ge 0\, |\, \sigma_{d}^{r}h_{i}\notin H'_{d}\}$. By an abuse of notation, we set $e_{[\sigma_{d}^{r(h_{i})+1}h_{i}]}B_{d}=0$ if the $\sigma_{d}$-orbit of $h_{i}$ is contained in $H_{d}'$. Then, the left minimal $\mathrm{add}(e_{H_{d}''}B_{d})$-approximation of $e_{[h_{i}]}B_{d}$ is given by

    \smallskip

    \begin{center}
       $\alpha_{[h_{i}]}$ : \begin{tikzcd}[column sep=3cm, ampersand replacement=\&] e_{[h_{i}]}B_{d} \arrow[cramped]{r}[above]{\begin{pmatrix} \alpha(h_{i},H'_{d}) \\ \alpha(\iota_{d}h_{i},H'_{d})           
       \end{pmatrix}} \&e_{[\sigma_{d}^{r(h_{i})+1}h_{i}]}B_{d} \,\oplus\, e_{[\sigma_{d}^{r(\iota_{d}h_{i})+1}\iota_{d}h_{i}]}B_{d} 
        \end{tikzcd}
    \end{center}
    
    \smallskip

    \noindent The two paths in $B_{d}$ inducing this approximation can be represented in $\Gamma_{d}$ as follows

    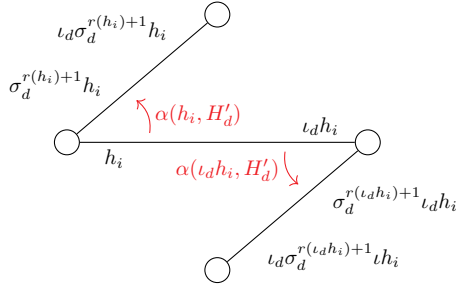
\begin{figure}[H]
        \centering
         \begin{tikzpicture}
        \tikzstyle{vertex}=[draw, circle, minimum size=0.2cm]
        \node[vertex] (1) at (-2,0) {};
        \node[vertex] (2) at (2,0) {};
        \node[vertex] (3) at (0,1.7) {};
        \node[vertex] (4) at (-0,-1.7) {};
        \draw (1)--(2) node[very near start, below, scale=0.8]{$h_{i}$} node[near start, name=A]{} node[very near end, above, scale=0.8]{$\iota_{d} h_{i}$} node[near end, name=C]{};
        \draw (1)--(3) node[near start, above left, scale=0.8]{$\sigma_{d}^{r(h_{i})+1}h_{i}$} node[midway, xshift=-2mm, yshift=-2mm, name=B]{} node[near end, above left, scale=0.8]{$\iota_{d}\sigma_{d}^{r(h_{i})+1}h_{i}$};
        \draw[Red, ->] (A) to[bend right]  node[midway, right, scale=0.8]{$\alpha(h_{i},H'_{d})$} (B);
        \draw (2)--(4) node[near start, below right, scale=0.8]{$\sigma_{d}^{r(\iota_{d} h_{i})+1}\iota_{d} h_{i}$} node[midway, xshift=2mm, yshift=2mm, name=D]{} node[near end, below right, scale=0.8]{$\iota_{d}\sigma_{d}^{r(\iota_{d} h_{i})+1}\iota h_{i}$};
        \draw[Red, ->] (C) to[bend right] node[midway, left, scale=0.8]{$\alpha(\iota_{d}h_{i},H'_{d})$} (D);
        \end{tikzpicture}
        \label{fig:left minimal approximation}
        \caption{Left minimal $\mathrm{add}(e_{H_{d}''}\, B_{d})$-approximation of $e_{[h_{i}]}B_{d}$}
    \end{figure}
    
    \noindent By definition of the left mutation, we have 
    
    \[\mu^{+}(B_{d}\,;\, e_{H_{d}''}B_{d})=e_{H_{d}''}B_{d}\, \oplus \, \bigoplus_{[h_{i}]\in H'_{d}/\iota_{d}}\mathrm{Cone}(\alpha_{[h_{i}]})=e_{H_{d}''}B_{d}\, \oplus \, \bigoplus_{[h]\in H'/\iota}\mathrm{Cone}(\oplus_{i\in \Z/\overline{m}\,\Z}\ \alpha_{[h_{i}]})\]
    
    \noindent Since $G$ acts bijectively on the set of edges $(H_{d}\backslash H'_{d})/\iota_{d}$,  the object $e_{H_{d}''}B_{d}$ is $G$-invariant by Example \ref{ex:G-invariant object} \ref{item:G-invariant object 1} . Moreover, for all $[h]\in H'/\iota$,

    \[\bigoplus_{i\in \Z/\overline{m}\, \Z}\, \alpha_{[h_{i}]}:X_{[h]}:=\bigoplus_{i\in \Z/\overline{m}\,\Z}\, e_{[h_{i}]}B_{d}\longrightarrow Y_{[h]}:=\bigoplus_{i\in \Z/\overline{m}\,\Z}\, (e_{[\sigma_{d}^{r(h_{i})+1}h_{i}]}B_{d} \,\oplus\, e_{[\sigma_{d}^{r(\iota_{d}h_{i})+1}\iota_{d}h_{i}]}B_{d})\]

    \noindent is a morphism in $\Db{B_{d}}$ between two $G$-invariant objects. Indeed, this follows again from Example \ref{ex:G-invariant object} \ref{item:G-invariant object 1} and the isomorphisms $\iota_{g}^{X_{[h]}}$ and $\iota_{g}^{Y_{[h]}}$ (cf Definition \ref{def:G-invariant object}) are respectively given by the action of $g$ on $X_{[h]}$ and $Y_{[h]}$. Moreover, using the definition of $\alpha_{[h_{i}]}$, it is not hard to check that the following diagram commutes for all $g\in G$

    \smallskip

    \begin{center}
        \begin{tikzcd}[row sep=2cm, column sep=2.5cm]
            X_{[h]}^{g^{-1}} \arrow{d}[left]{\iota_{g}^{X_{[h]}}} \arrow{r}[above]{\big(\oplus_{i\in \Z/\overline{m}\,\Z}\ \alpha_{[h_{i}]}\big)^{g^{-1}}} &Y_{[h]}^{g^{-1}} \arrow{d}[right]{\iota_{g}^{Y_{[h]}}} \\
            X_{[h]} \arrow{r}[below]{\oplus_{i\in \Z/\overline{m}\,\Z}\ \alpha_{[h_{i}]}} &Y_{[h]}
        \end{tikzcd}
    \end{center}

    \smallskip

    \noindent By Example \ref{ex:G-invariant object} \ref{item:G-invariant object 2}, we deduce that $\mathrm{Cone}(\oplus_{i\in \Z/\overline{m}\,\Z}\ \alpha_{[h_{i}]})$ is $G$-invariant for all $[h]\in H'/\iota$. Hence, we conclude that $\mu^{+}(B_{d}\,;\, e_{H_{d}''}B_{d})$ is indeed $G$-invariant. 

    \medskip
    
    To complete the proof, it remains to show that 
    
    \[T_{d}G:=\mu^{+}(B_{d}\, G\,;\, e_{H_{d}''}B_{d}\, G)\]
    
    \noindent is isomorphic to $T_{d}\, \lotimes{B_{d}}{B_{d}\, G}$ in $\per{B_{d}\, G}$. For this, it suffices to check that the assumptions of Proposition \ref{prop:silting mutation in skew group algebras} hold. We have seen that $e_{H_{d}''}B_{d}$ is $G$-invariant. One can check with a similar argument that $e_{H'_{d}}B_{d}$, the projective $B_{d}$-module corresponding to the edges in $H'_{d}/\iota_{d}$, is also $G$-invariant. Moreover, the left minimal $\mathrm{add}(e_{H_{d}''}B_{d})$-approximation of $e_{H'_{d}}B_{d}$ is of the form

    \[\bigoplus_{[h]\in H'/\iota}\bigoplus_{i\in \Z/\overline{m}\,\Z} \, \alpha_{[h_{i}]}:e_{H'_{d}}B_{d}=\bigoplus_{[h]\in H'/\iota} \, X_{[h]}\longrightarrow \bigoplus_{[h]\in H'/\iota} \, Y_{[h]}\]

    \noindent We have seen previously that $Y_{[h]}$ is $G$-invariant for all $[h]\in H'/\iota$. Hence $\oplus_{[h]\in H'/\iota} \, Y_{[h]}$ is also $G$-invariant and this concludes the proof. 
\end{proof}

\medskip

\begin{proof}[Proof of Theorem \ref{thm:generalized Kauer moves with multiplicity}]
Let us define an admissible $\Z/\overline{m}\,\Z$-grading $d:H\rightarrow \Z/\overline{m}\, \Z$ on $\Gamma$ as follows : for every vertex $v\in H/\sigma$,

\smallskip

\begin{enumerate}[label=\textbullet, font=\tiny]
\item If there exist half-edges in $H\backslash H'$ and in $H'$ that are incident with $v$ , then we choose any maximal sector $(h,r)\in\mathrm{Sect}(H',\sigma)$ around $v$ and we set $d(\sigma^{-1}h)=\overline{m}/m(\sigma^{-1}h)$ and $d(h')=0$ for all $h'\neq \sigma^{-1}h$ incident with $v$.

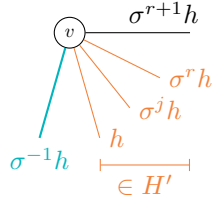
\begin{figure}[H]
\centering
        \begin{tikzpicture}[scale=0.8]
        \tikzstyle{vertex}=[draw, circle, scale=0.8, minimum size=0.2cm]
        \node[vertex, scale=0.9] (1) at (-2,0) {$v$};
        \draw (1)--(0,0) node[near end, above]{$\sigma^{r+1}h$};
        \draw[Aquamarine, thick] (1)--(-2.5,-1.75) node[below]{$\sigma^{-1}h$};
        \draw[Orange] (1)--(-1.5,-1.75) node[below, right]{$h$};
        \draw[Orange] (1)--(-1,-1.25) node[below, right]{$\sigma^{j}h$};
        \draw[Orange] (1)--(-0.5,-0.75) node[below, right]{$\sigma^{r}h$};
        \draw[|-|, Orange] (-1.5,-2.2)--(0,-2.2) node[below, midway]{$\in H'$};
        \end{tikzpicture} 
\end{figure}

\noindent The blue half-edge represents the only half-edge in the cyclic ordering around $v$ that has a non trivial degree for $d$.

\item Else, we choose any $h\in H$ incident with $v$ and we set $d(h)=\overline{m}/m(h)$ and $d(h')=0$ for all $h'\neq h$ incident with $v$.

\end{enumerate}

\smallskip

\noindent Let $\Gamma_{d}=(H_{d},\iota_{d},\sigma_{d})$ be the Galois covering of $\Gamma$ constructed before Proposition \ref{prop:Asashiba} and $H'_{d}=H'\times \Z/\overline{m}\,\Z\subset H_{d}$. We denote by $B_{d}$ and $B_{d}'$ the Brauer graph algebras associated to $\Gamma_{d}$ and $\mu^{+}_{H'_{d}}(\Gamma_{d})$ respectively. Moreover, $G$ denotes the cyclic group $(\mathbb{C}_{\overline{m}},.)$ of order $\overline{m}$. We have the following commutative diagram

\smallskip

\begin{center}
    \begin{tikzcd}[column sep=3cm, row sep=2.5cm]
        \per{B_{d}'\, G} \arrow[cramped]{r}[above]{-\lotimes{B_{d}'\, G}{\ T_{d}G}} \arrow[cramped]{d}[left, xshift=-1mm]{-\lotimes{B'_{d}\, G}{\ B'_{d}\, G f}} &\per{B_{d}\, G} \arrow[cramped]{d}[right]{-\lotimes{B_{d}\, G}{\ B_{d}\, G f}} \\
        \per{B'} \arrow[cramped, dashed]{r}[below]{-\lotimes{B'}{\ T}} &\per{B}
    \end{tikzcd}
\end{center}

\smallskip

\noindent where $f$ is the idempotent given by the sum of the $e_{[h_{0}]}\otimes 1_{G}$ for all $[h]\in H/\iota$ and $T_{d}G=\mu^{+}(B_{d}\, G\,;\, e_{H_{d}''}\, B_{d}\, G)$. The vertical arrows arise from Proposition \ref{prop:Galois covering and skew group algebra} since $\mu^{+}_{H'_{d}}(\Gamma_{d})$ is the Galois covering of $\mu^{+}_{H'}(\Gamma,d)$ by Proposition \ref{prop:commutativity Galois covering and generalized Kauer move with multiplicity}. The top arrow is given by Proposition \ref{prop:generalized Kauer move and skew group algebra}. This leads to a derived equivalence between $B'$ and $B$ and the associated tilting object $T$ is given by

\[T=(fB_{d}'\, G\lotimes{B_{d}'\, G}{T_{d}\, G)\lotimes{B_{d}\, G}{B_{d}\, Gf}}\]

\noindent It remains to prove that $T\simeq\mu^{+}(B; e_{H''}B)$. We saw in the proof of Proposition \ref{prop:Galois covering and skew group algebra} that there is an isomorphism in $\per{B{d}\, G}$

\[T_{d}\, G\simeq T_{d}\lotimes{B_{d}}{B_{d}\, G}\]

\noindent where $T_{d}=\mu^{+}(B_{d}\,;\, e_{H_{d}''}B_{d})$. Hence, denoting by $f_{H_{d}''}$ the idempotent given by the sum of the $e_{[h_{0}]}\otimes1_{G}$ for all $[h]\in (H\backslash H')/\iota$, we have

\[fB_{d}'\, G\lotimes{B_{d}'\, G}{T_{d}\, G}\simeq f_{H_{d}''}B_{d}\, G\, \oplus\, \bigoplus_{[h_{0}]\in H'_{d}/\iota}\, \mathrm{Cone}(\alpha_{[h_{0}]}\lotimes{B_{d}}{B_{d}\, G})\] 

\noindent where $\alpha_{[h_{0}]}$ is the left minimal $\mathrm{add}(e_{H_{d}''}B_{d})$-approximation of $e_{[h_{0}]}B_{d}$. By construction of $d$, $\alpha_{[h_{0}]}\lotimes{B_{d}}{B_{d}\, G}$ is of the form

\[\alpha_{[h_{0}]}\lotimes{B_{d}}{B_{d}\, G}:e_{[h_{0}]}B_{d}\, G\longrightarrow e_{[(\sigma^{r(h_{0})+1}h)_{0}]}B_{d}\, G\,\oplus\, e_{[(\sigma^{r(\iota_{d} h_{0})+1}\iota h)_{0}]}B_{d}\, G\]

\noindent Thus, $\alpha_{[h_{0}]}\lotimes{B_{d}}{B_{d}\, G}$ is the left minimal $\mathrm{add}(f_{H_{d}''}B_{d}\, G)$-approximation of $e_{[h_{0}]}B_{d}\, G$. Hence, 

\[fB_{d}'\, G\lotimes{B_{d}'\, G}{T_{d}\, G}\simeq \mu^{+}(fB_{d}\, G\, ;\, f_{H_{d}''}B_{d}\, G)\]

\noindent Moreover, since $-\lotimes{B_{d}\, G}{B_{d}\, Gf}:\per{B_{d}\, G}\rightarrow\per{B_{\Gamma}}$ is an equivalence of triangulated categories, we conclude that 

\[T\simeq \mu^{+}(fB_{d}\, G\, ;\, f_{H_{d}''}B_{d}\, G)\lotimes{B_{d}\, G}{B_{d}\, Gf}\simeq \mu^{+}(fB_{d}\, Gf\,;\,f_{H_{d}''}B_{d}\, G\,f)\]

\noindent which is isomorphic to $\mu^{+}(B\,;\, e_{H''}B)$ by Proposition \ref{prop:Galois covering and skew group algebra}. 
\end{proof}

\medskip

\begin{rem}
    Using the classification of Brauer graph algebras up to derived equivalence given in \cite{OZ}, we can directly prove that the algebras $B$ and $B'$ are derived equivalent. Indeed, we have to check the following conditions

    \smallskip

    \begin{enumerate}[label=(\roman*)]
        \item $\Gamma$ and $\mu^{+}_{H'}(\Gamma)$ have the same number of edges and vertices;
        \item $\Gamma$ and $\mu^{+}_{H'}(\Gamma)$ have the same number of faces and their multi-sets of the perimeters of the faces coincide;
        \item The multi-sets of the multiplicities of the vertices of $\Gamma$ and $\mu^{+}_{H'}(\Gamma)$ coincide;
        \item Either both or none of $\Gamma$ and $\mu^{+}_{H'}(\Gamma)$ are bipartite.
    \end{enumerate}

    \smallskip

    \noindent Thanks to Proposition 1.17 in \cite{Soto}, it only remains to prove the third point. It is clear by the definition of $m_{(h,r)}$ that the multi-sets of the multiplicities of the vertices of $\Gamma$ and $\mu^{+}_{(h,r)}(\Gamma)$ coincide for all $(h,r)$ maximal sector in $\Gamma$ of elements in $H'$. Since $\mu^{+}_{H'}(\Gamma)$ is obtained from the succession of generalized Kauer moves of all maximal sectors of elements in $H'$, we deduce that the third point holds. Hence, the algebras $B$ and $B'$ are derived equivalent.
\end{rem}

\medskip

\section{Generalized Kauer moves for skew Brauer graph algebras}

\medskip

In this section, we define the notion of skew Brauer graph algebras which is a generalization of Brauer graph algebras. Our aim is to define a notion of generalized Kauer moves for skew Brauer graph algebras so that it corresponds to a silting mutation. This will again be a generalization of Theorem 3.10 in \cite{Soto}. Throughout this section, the characteristic of the field $k$ is supposed to be different than 2.

\medskip

\subsection{Skew Brauer graph algebras}

\medskip

In this part, we define the notion of skew Brauer graph which is a Brauer graph where the pairing may have some fixed points. From this combinatorial data, we construct a new algebra called a skew Brauer graph algebra. This construction is analogous to the definition of a Brauer graph algebra.

\medskip

\begin{defn}\label{def:skew Brauer graph}
A \textit{skew Brauer graph} is the data $\Gamma=(H,\iota,\sigma,m)$ where 

\smallskip

\begin{itemize}[label=\textbullet, font=\tiny]
\item $H$ is the set of half-edges;
\item $\iota$ is a permutation of $H$ with possible fixed points satisfying $\iota^{2}=\id{H}$ ;
\item $\sigma$ is a permutation of $H$ called the \textit{orientation};
\item $m:H\rightarrow\Z_{>0}$ is a map that is constant on a $\sigma$-orbit : it is called the \textit{multiplicity}.
\end{itemize}    
\end{defn}

\medskip

To the data $\Gamma=(H,\iota,\sigma,m)$, one can naturally define a graph. For this, let us denote $H_{\circ}$ the set of half-edges that are not fixed by $\iota$ and $H_{\times}$ the set of half-edges that are fixed by $\iota$. Hence, the vertex set of this graph is given by $\Gamma_{0}\simeq H/\sigma\cup H_{\times}$ and the edge set is given by $\Gamma_{1}=H/\iota$. Moreover, each cycle in the decomposition of $\sigma$ gives rise to a cyclic ordering of the half-edges around a vertex in $H/\sigma$. Since $m:H\rightarrow \Z_{>0}$ is constant on a $\sigma$-orbit, it induces a map $\widetilde{m}:H/\sigma\rightarrow \Z_{>0}$. We denote by $\circ$ the vertices in $H/\sigma$ and by $\times$ the vertices in $H_{\times}$. By construction of the graph, each $\times$-vertex has a unique edge incident to it. One can understand an edge incident to a $\times$-vertex as a ``degenerate" edge. Hence, one can see a skew Brauer graph as a Brauer graph with possible ``degenerate" edges. From now on, we identify $\Gamma=(H,\iota,\sigma,m)$ with the graph that we have constructed above. By convention, the orientation of a skew Brauer graph corresponds to the local embedding of each $\circ$-vertex into the counterclockwise oriented plane.

\medskip

\begin{rem}
    For a similar reason than in Remark \ref{rem:excluded Brauer graphs}, we do not take into consideration the two following skew Brauer graph algebras 

    \smallskip

    \begin{center}
        \begin{tikzpicture}
            \tikzstyle{vertex}=[circle,draw,minimum size=0.2cm]
            \begin{scope}[xshift=-3cm]
                \node[vertex] (A) at (-1.5,0) {};
                \node[vertex] (B) at (1.5,0) {};
                \draw (A)--(B) node[near start, above, scale=0.9]{$1^{+}$} node[near end, above, scale=0.9]{$1^{-}$};
             \end{scope}
             \begin{scope}[xshift=3cm]
                \node[vertex] (A) at (-1.5,0) {};
                \draw (A)--(1.5,0) node[scale=1.3]{\textbf{\texttimes}} node[midway, above]{2};
             \end{scope}
        \end{tikzpicture}
    \end{center}

    \smallskip

    \noindent where the first one is given in Remark \ref{rem:excluded Brauer graphs} and the other is given by $H=\{2\}$, $\iota=\id{H}=\sigma$ and $m$ being identically 1. More precisely, we assume that such graphs do not appear in any connected component of a given skew Brauer graph. In particular, since we have excluded such graphs, there is no half-edge that is fixed simultaneously by $\iota$ and $\sigma$.
\end{rem}

\medskip

\begin{ex}\label{ex:skew Brauer graph}
    Let $\Gamma=(H,\iota,\sigma,m)$ be a skew Brauer graph where the set of half-edges $H$ is $\{1^{+},1^{-},2,3, 4^{+},4^{-},5^{+},5^{-}\}$, the involution $\iota$ is $(1^{+}\ 1^{-})(4^{+}\ 4^{-})(5^{+}\ 5^{-})$, the permutation $\sigma$ is $(1^{-}\ 3\ 2)(1^{+}\ 4^{+}\ 5^{+})$ and the multiplicity $m:H\rightarrow\Z_{>0}$ is given by $m(4^{-})=3$, $m(3)=m(2)=m(1^{-})=2$ and $m(4^{+})=m(1^{+})=m(5^{+})=m(5^{-})=1$. With our previous convention, this skew Brauer graph can be represented as follows

    \smallskip

    
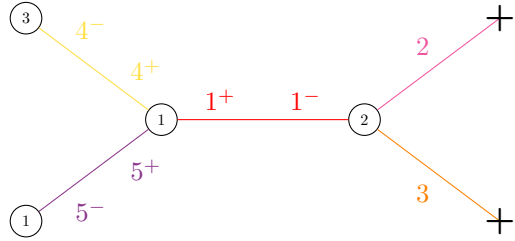
\begin{center}
        \begin{tikzpicture}[scale=0.9]
            \tikzstyle{vertex}=[circle,draw,scale=0.6,minimum size=0.7cm]
            \node[vertex] (A) at (0,0) {2};
            \node[vertex] (B) at (-3,0) {1};
            \node[vertex] (C) at (-5,1.5) {3};
            \node[vertex] (D) at (-5,-1.5) {1};
            \draw[red] (A)--(B) node[near start, above]{$1^{-}$} node[near end, above]{$1^{+}$};
            \draw[VioletRed] (A)--(2,1.5) node[rotate=45, scale=1.3, black]{\textbf{\texttimes}} node[midway, above left]{$2$};
            \draw[orange] (A)--(2,-1.5) node[rotate=45, scale=1.3, black]{\textbf{\texttimes}} node[midway, below left]{$3$};
            \draw[Goldenrod] (B)--(C) node[near start, above right]{$4^{+}$} node[near end, above right]{$4^{-}$} ;
            \draw[Fuchsia] (B)--(D) node[near start, below right]{$5^{+}$} node[near end, below right]{$5^{-}$};
        \end{tikzpicture}
    \end{center}

    \smallskip

    \noindent where the multiplicity of a half-edge is given by the value in the $\circ$-vertex it is incident to. From now on, we will only write on the graph the multiplicities that are strictly greater than one.
\end{ex}

\medskip

Given a skew Brauer graph $\Gamma=(H,\iota,\sigma,m)$, one can construct a quiver $Q_{\Gamma}=(Q_{0},Q_{1})$ as follows

\smallskip

\begin{itemize}[label=\textbullet,font=\tiny]
 \item The vertex set $Q_{0}$ of $Q_{\Gamma}$ is given by $H_{\circ}/\iota\cup(H_{\times}/\iota\times\Z/2\Z)$. We denote by $[h]_{i}$ a vertex of $Q_{\Gamma}$ where $i=\varnothing$ if $[h]\in H_{\circ}/\iota$ and $i=0,1$ if $[h]\in H_{\times}/\iota$.
 \item The arrow set $Q_{1}$ of $Q_{\Gamma}$ is induced by the orientation $\sigma$ and the multiplicity $m$. More precisely, for any half-edge $h\in H$ that is not fixed by $\sigma$ there are arrows $^{j}\alpha_{h}\phantom{}^{i}$ from $[h]_{i}$ to $[\sigma h]_{j}$ for all $i,j=\varnothing,0,1$. Moreover, for any half-edge $h$ that is fixed by $\sigma$ there is an arrow $\alpha_{h}$ from $[h]$ to itself if and only if $m(h)>1$.
\end{itemize}

\smallskip

\noindent Note that any half-edge $h\in H$ inducing an arrow in $Q_{\Gamma}$ gives rise to oriented cycles in $Q_{\Gamma}$ that begins and ends at $[h]_{i}$ for $i=\varnothing,0,1$. We call such an oriented cycle a \textit{special $h_{i}$-cycle}. By definition, there are $2^{|H_{\times}\cap(\mathrm{Orb}_{\sigma}(h)\backslash\{h\})|}$ such special cycles where $\mathrm{Orb}_{\sigma}(h)$ denotes the $\sigma$-orbit of $h$. We denote by $\mathscr{C}_{h_{i}}$ the set of all special $h_{i}$-cycle. Moreover, an element of $\mathscr{C}_{h_{i}}$ is of the form $\phantom{}^{i}\alpha_{\sigma^{n-1}h}\phantom{}^{i_{n-1}}\ldots \phantom{}^{i_{1}}\alpha_{h}\phantom{}^{i}$ for some $i_{1}\ldots,i_{n-1}=\varnothing,0,1$ where $n$ is the size of the $\sigma$-orbit of $h$.

\medskip

\begin{ex}\label{ex:quiver of a skew Brauer graph}
    Let $\Gamma=(H,\iota,\sigma,m)$ be the skew Brauer graph defined in Example \ref{ex:skew Brauer graph}. Its corresponding quiver is given by

    \smallskip

    
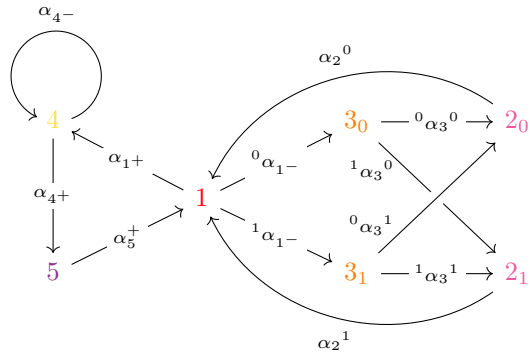
\begin{center}
        \begin{tikzcd}[row sep=0.5cm, column sep=1.5cm]
            \textcolor{Goldenrod}{4} \arrow[cramped, "\alpha_{4^{-}}", to path={(\tikztostart.10) arc(-70:245:0.6) [midway, yshift=4mm]\tikztonodes}]\arrow[cramped]{dd}[description]{\alpha_{4^{+}}}&& \textcolor{orange}{3_{0}} \arrow[cramped]{r}[description]{^{0}\alpha_{3}\phantom{}^{0}} \arrow[cramped]{rdd}[near start, left, xshift=-1mm]{^{1}\alpha_{3}\phantom{}^{0}}& \textcolor{VioletRed}{2_{0}} \arrow[cramped, bend right=50]{lld}[midway, above, yshift=1mm]{\alpha_{2}\phantom{}^{0}}\\
            &\textcolor{red}{1} \arrow[cramped]{ul}[description]{\alpha_{1^{+}}} \arrow[cramped]{ru}[description]{^{0}\alpha_{1^{-}}} \arrow[cramped]{rd}[description]{^{1}\alpha_{1^{-}}} && \\
            \textcolor{Fuchsia}{5} \arrow[cramped]{ur}[description]{\alpha_{5}^{+}}&& \textcolor{orange}{3_{1}} \arrow[cramped]{r}[description]{^{1}\alpha_{3}\phantom{}^{1}} \arrow[cramped, crossing over]{uur}[near start, left, xshift=-1mm]{^{0}\alpha_{3}\phantom{}^{1}} & \textcolor{VioletRed}{2_{1}} \arrow[cramped, bend left=50]{llu}[below, yshift=-1mm]{\alpha_{2}\phantom{}^{1}}
        \end{tikzcd}
    \end{center}

    \smallskip

    \noindent In this case, the sets of special $h_{i}$-cycles are given by $\mathscr{C}_{1^{-}}=\{\alpha_{2}{}^{i_{2}}\ ^{i_{2}}\alpha_{3}{}^{i_{1}} \ ^{i_{1}}\alpha_{1^{-}}\,|\, i_{1},i_{2}=0,1\}$,  $\mathscr{C}_{2_{i}}=\{\phantom{}^{i}\alpha_{3}\phantom{}^{i_{1}}\ ^{i_{1}}\alpha_{1^{-}}\ \alpha_{2}{}^{i}\,|\,i_{1}=0,1\}$ for $i=0,1$, $\mathscr{C}_{3_{i}}=\{{}^{i}\alpha_{1^{-}}{} \ \alpha_{2}{}^{i_{1}}\ ^{i_{1}}\alpha_{3}\phantom{}^{i}\,|\,i_{1}=0,1\}$ for $i=0,1$, $\mathscr{C}_{1^{+}}=\{\alpha_{5^{+}}\alpha_{4^{+}}\alpha_{1^{+}}\}$, $\mathscr{C}_{4^{+}}=\{\alpha_{1^{+}}\alpha_{5^{+}}\alpha_{4^{+}}\}$, $\mathscr{C}_{4^{-}}=\{\alpha_{4^{-}}\}$ and $\mathscr{C}_{5^{+}}=\{\alpha_{4^{+}}\alpha_{1^{+}}\alpha_{5^{+}}\}$. Note that there is no special $5^{-}$-cycle since $5^{-}$ is fixed by $\sigma$ and is of multiplicity 1 so it does not induce an arrow in $Q_{\Gamma}$. 

\end{ex}

\medskip

\begin{defn}\label{def:skew Brauer graph algebra}
Let $\Gamma=(H,\iota,\sigma,m)$ be a skew Brauer graph and $Q_{\Gamma}$ be its associated quiver. The \textit{skew Brauer graph algebra} $B_{\Gamma}$ of $\Gamma$ is the path algebra $kQ_{\Gamma}/I_{\Gamma}$ where the set of relations $I_{\Gamma}$ is generated by

\begin{enumerate}[label=(\Roman*)]
\item \label{item:skew Brauer graph algebra 1} \[(2^{n_{\times}(h)}C_{h})^{m(h)}=(2^{n_{\times}(\iota h)}C_{\iota h})^{m(\iota h)}\]

\noindent for all special cycles $C_{h}\in\mathscr{C}_{h}$ and $C_{\iota h}\in\mathscr{C}_{\iota h}$ where $n_{\times}(h)$ denotes the number of half-edges in the $\sigma$-orbit of $h$ that are in $H_{\times}$.

\item \label{item:skew Brauer graph algebra 2} \[^{i_{1}}\alpha_{h}\phantom{}^{i}\ (C_{h_{i}})^{m(h)}\]

\noindent for all special cycles $C_{h_{i}}=\phantom{}^{i}\alpha_{\sigma^{n-1}h}\phantom{}^{i_{n-1}}\ldots \phantom{}^{i_{1}}\alpha_{h}\phantom{}^{i}\in\mathscr{C}_{h_{i}}$.

\item \label{item:skew Brauer graph algebra 3}\[^{i'}\alpha_{\iota\sigma h}\ \alpha_{h}\phantom{}^{i} \ \]

\noindent for all $h\in H$ such that $h$ and $\iota\sigma h$ both induce an arrow with $\sigma h\in H_{\circ}$ and $i,i'=\varnothing,0,1$.

\item \label{item:skew Brauer graph algebra 4} \[\phantom{}^{i+1}\alpha_{\sigma^{n-1}h}\phantom{}^{i_{n-1}}\ldots \phantom{}^{i_{1}}\alpha_{h}\phantom{}^{i}(C_{h_{i}})^{m(h)-1}\]

\noindent for all special cycles $C_{h_{i}}=\phantom{}^{i}\alpha_{\sigma^{n-1}h}\phantom{}^{i_{n-1}}\ldots \phantom{}^{i_{1}}\alpha_{h}\phantom{}^{i}\in\mathscr{C}_{h_{i}}$ with $h\in H_{\times}$.

\item \label{item:skew Brauer graph algebra 5} \[^{i'}\alpha_{\sigma h}\phantom{}^{0}\ ^{0}\alpha_{h}\phantom{}^{i}-\phantom{}^{i'}\alpha_{\sigma h}\phantom{}^{1}\ ^{1}\alpha_{h}\phantom{}^{i}\]

\noindent for all $h\in H$ not fixed by $\sigma$ such that $\sigma h\in H_{\times}$ and $i,i'=\varnothing,0,1$.

\end{enumerate}
    
\end{defn}

\medskip

In particular, for $H_{\times}=\varnothing$, the definitions of skew Brauer graph and skew Brauer graph algebra coincide with the definitions of Brauer graph and  Brauer graph algebra. Thanks to the relation \ref{item:skew Brauer graph algebra 5}, note that all the special $h_{i}$-cycles are equal in $B_{\Gamma}$ . Moreover, these relations are not minimal as we can see in the following example.

\medskip

\begin{ex} \label{ex:skew Brauer graph algebra}
    Let $\Gamma=(H,\iota,\sigma,m)$ the skew Brauer graph defined in Example \ref{ex:skew Brauer graph}. The skew Brauer graph algebra associated to $\Gamma$ is defined as follows : its quiver is described in Example \ref{ex:quiver of a skew Brauer graph} and its ideal of relations is generated by

    \smallskip

    \begin{enumerate}[label=(\Roman*)]
    \item \label{item1:skew Brauer graph algebra}$\alpha_{1^{+}}\alpha_{5^{+}}\alpha_{4^{+}}-\alpha_{4^{-}}^{3}$ and $16 \ (\alpha_{2}\phantom{}^{i_{2}}\ ^{i_{2}}\alpha_{3}\phantom{}^{i_{1}} \ ^{i_{1}}\alpha_{1^{-}})^{2}-\alpha_{5^{+}}\alpha_{4^{+}}\alpha_{1^{+}}$ for all $i_{1},i_{2}=0,1$ ;
    \item \label{item2:skew Brauer graph algebra} $\alpha_{1^{+}}\alpha_{5^{+}}\alpha_{4^{+}}\alpha_{1^{+}}, \alpha_{4^{+}}\alpha_{1^{+}}\alpha_{5^{+}}\alpha_{4^{+}}$, $\alpha_{5^{+}}\alpha_{4^{+}}\alpha_{1^{+}}\alpha_{5^{+}}$, $\alpha_{4^{-}}^{4}$ and $^{i_{1}}\alpha_{1^{-}} (\alpha_{2}\phantom{}^{i_{2}}\ ^{i_{2}}\alpha_{3}\phantom{}^{i_{1}} \ ^{i_{1}}\alpha_{1^{-}})^{2}$, \\$\phantom{}^{i_{1}}\alpha_{3}\phantom{}^{i} (\phantom{}^{i}\alpha_{1^{-}}\ \alpha_{2}{}^{i_{1}}\ ^{i_{1}}\alpha_{3}\phantom{}^{i})^{2}$, $\alpha_{2}\phantom{}^{i} (\phantom{}^{i}\alpha_{3}\phantom{}^{i_{1}}\ ^{i_{1}}\alpha_{1^{-}}\ \alpha_{2}{}^{i})^{2}$ for all $i,i_{1},i_{2}=0,1$;
    \item \label{item3:skew Brauer graph algebra}  $\alpha_{4^{-}}\alpha_{1^{+}}$, $\alpha_{4^{+}}\alpha_{4^{-}}$ and $\alpha_{1^{+}}\ \alpha_{2}\phantom{}^{i}$, $^{i}\alpha_{1^{-}}\ \alpha_{5^{+}}$ for all $i=0,1$;
    \item \label{item4:skew Brauer graph algebra}$\phantom{}^{i+1}\alpha_{1^{-}}\ \alpha_{2}^{i_{1}}\ ^{i_{1}}\alpha_{3}\phantom{}^{i}\ \phantom{}^{i}\alpha_{1^{-}}\ \alpha_{2}^{i_{1}}\ ^{i_{1}}\alpha_{3}\phantom{}^{i}$, $\phantom{}^{i+1}\alpha_{3}\phantom{}^{i_{1}}\ ^{i_{1}}\alpha_{1^{-}}\ \alpha_{2}{}^{i} \ \phantom{}^{i}\alpha_{3}\phantom{}^{i_{1}}\ ^{i_{1}}\alpha_{1^{-}}\ \alpha_{2}{}^{i}$ for all $i_{1}=0,1$;
    \item \label{item5:skew Brauer graph algebra}$^{i}\alpha_{3}\phantom{}^{0}\ ^{0}\alpha_{1^{-}}-{}^{i}\alpha_{3}\phantom{}^{1}\ ^{1}\alpha_{1^{-}}$, $\alpha_{2}\phantom{}^{0}\ ^{0}\alpha_{3}{}^{i}-\alpha_{2}\phantom{}^{1}\ ^{1}\alpha_{3}{}^{i}$ for all $i=0,1$.

    \end{enumerate}

    \smallskip

    \noindent Note that some of the relations in \ref{item2:skew Brauer graph algebra} follows from the relations in \ref{item1:skew Brauer graph algebra} and \ref{item3:skew Brauer graph algebra}.

\end{ex}

\medskip

\subsection{Covering of skew Brauer graphs}

\medskip

In this section, we explain how skew Brauer graph can be covered by a Brauer graph. This construction is analogous to the Galois covering of a Brauer graph with multiplicity (cf after Definition \ref{def:graded Brauer graphs}). In particular, we need to introduce a particular $\Z/2\Z$-grading on the skew Brauer graph.

\medskip

\begin{defn} \label{def:graded skew Brauer graph}
Let $\Gamma=(H,\iota,\sigma,m)$ be a skew Brauer graph.

\smallskip

\begin{itemize}[label=\textbullet, font=\tiny]
\item We say that a $\Z/2\Z$-grading $d:H\rightarrow\Z/2\Z$ is \textit{0-homogeneous} if 

\[\sum_{h\in H,s(h)=v}d(h)=0\]

\noindent for all $v\in H/\sigma$, where $s:H\rightarrow H/\sigma$ is the source map for $\circ$-vertex in $\Gamma$.

\item We say that $(\Gamma,d)$ is a \textit{$\Z/2\Z$-graded skew Brauer graph} if $d:H\rightarrow\Z/2\Z$ is $0$-homogeneous.

\end{itemize}
\end{defn}

\medskip

Let $(\Gamma,d)=(H,\iota,\sigma,m,d)$ be a $\Z/2\Z$-graded skew Brauer graph. One can construct a Brauer graph $\Gamma_{d}=(H_{d},\iota_{d},\sigma_{d},m_{d})$ as follows

\smallskip

\begin{itemize}[label=\textbullet, font=\tiny]
\item $H_{d}=H\times (\Z/2\Z)$ : an element of $H_{d}$ will be denoted $h_{i}$ for $h\in H$ and $i\in\Z/2\Z$;
\item For all $h_{i}\in H_{d}$, $\iota_{d}(h_{i})=\left\{\begin{aligned}
    &h_{i+1} &&\mbox{if $h\in H_{\times}$} \\
    &(\iota h)_{i} &&\mbox{if $h\in H_{\circ}$}
\end{aligned}\right.$;
\item For all $h_{i}\in H_{d}$, $\sigma_{d}(h_{i})=(\sigma h)_{i+d(h)}$;
\item For all $h_{i}\in H_{d}$, $m_{d}(h_{i})=m(h)$.

\end{itemize}

\smallskip

\noindent One can easily check that $\iota_{d}$ has no fixed points and that $m_{d}$ is constant on $\sigma_{d}$-orbits. Hence, $\Gamma_{d}$ is indeed a Brauer graph. Moreover, it is clear that the projection $p_{\Gamma}:H_{d}\rightarrow H$ defined a morphism of skew Brauer graphs from $\Gamma_{d}$ to $\Gamma$ i.e. it commutes with the involutions and the orientations of the skew Brauer graphs (cf Definition \ref{def:morphism of Brauer graph}). The Brauer graph $\Gamma_{d}$ will be called the \textit{covering} of $\Gamma$.

\medskip

\begin{ex} \label{ex:covering of a skew Brauer graph algebra}
    Let us consider the skew Brauer graph $\Gamma=(H,\iota,\sigma,m)$ defined in Example \ref{ex:skew Brauer graph}. We equip $\Gamma$ with a 0-homogeneous $\Z/2\Z$-grading $d:H\rightarrow\Z/2\Z$ defined by $d(h)=0$ for all $h\in H$ which can be represented on the graph as follows

    \smallskip

    
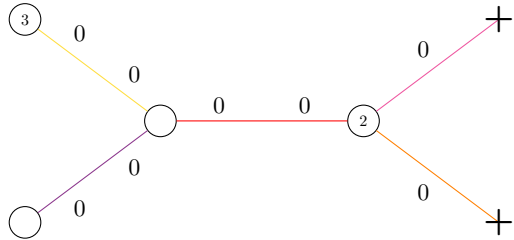
\begin{center}
        \begin{tikzpicture}[scale=0.9]
            \tikzstyle{vertex}=[circle,draw,scale=0.6,minimum size=0.7cm]
            \node[vertex] (A) at (0,0) {2};
            \node[vertex] (B) at (-3,0) {};
            \node[vertex] (C) at (-5,1.5) {3};
            \node[vertex] (D) at (-5,-1.5) {};
            \draw[red] (A)--(B) node[near start, above, scale=0.9,black]{0} node[near end, above, scale=0.9,black]{0};
            \draw[VioletRed] (A)--(2,1.5) node[rotate=45, scale=1.3, black]{\textbf{\texttimes}} node[midway, above left, scale=0.9,black]{0};
            \draw[orange] (A)--(2,-1.5) node[rotate=45, scale=1.3, black]{\textbf{\texttimes}} node[midway, below left, scale=0.9,black]{0};
            \draw[Goldenrod] (B)--(C) node[near start, above right, scale=0.9,black]{0} node[near end, above right, scale=0.9,black]{0} ;
            \draw[Fuchsia] (B)--(D) node[near start, below right, scale=0.9,black]{0} node[near end, below right, scale=0.9,black]{0};
        \end{tikzpicture}
    \end{center}

    \smallskip

    \noindent Then, the covering $\Gamma_{d}=(H_{d},\iota_{d},\sigma_{d},m_{d})$ of the $\Z/2\Z$-graded skew Brauer graph ($\Gamma,d)$ is given by the following Brauer graph

    \smallskip

    
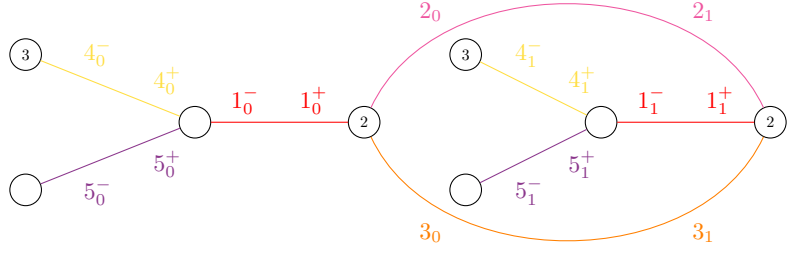
\begin{center}
        \begin{tikzpicture}[scale=0.9]
            \tikzstyle{vertex}=[circle,draw,scale=0.6,minimum size=0.7cm]
            \node[vertex] (A) at (0,0) {2};
            \node[vertex] (B) at (-2.5,0) {};
            \node[vertex] (C) at (-5,1) {3};
            \node[vertex] (D) at (-5,-1) {};
            \node[vertex] (E) at (6,0) {2};
            \node[vertex] (F) at (3.5,0) {};
            \node[vertex] (G) at (1.5,1) {3};
            \node[vertex] (H) at (1.5,-1) {};
            \draw[red] (A)--(B) node[near start, above, scale=0.9]{$1^{+}_{0}$} node[near end, above, scale=0.9]{$1^{-}_{0}$};
            \draw[VioletRed] (A) to[bend left=65] node[near start, above left, scale=0.9]{$2_{0}$} node[near end, above right, scale=0.9]{$2_{1}$}(E);
            \draw[orange] (A) to[bend right=65] node[near start, below left, scale=0.9]{$3_{0}$} node[near end, below right, scale=0.9]{$3_{1}$} (E);
            \draw[Goldenrod] (B)--(C) node[near start, above right, scale=0.9]{$4^{+}_{0}$} node[near end, above right, scale=0.9]{$4^{-}_{0}$} ;
            \draw[Fuchsia] (B)--(D) node[near start, below right, scale=0.9]{$5^{+}_{0}$} node[near end, below right, scale=0.9]{$5^{-}_{0}$};
            \draw[red] (E)--(F) node[near start, above, scale=0.9]{$1^{+}_{1}$} node[near end, above, scale=0.9]{$1^{-}_{1}$};
            \draw[Goldenrod] (F)--(G) node[near start, above right, scale=0.9]{$4^{+}_{1}$} node[near end, above right, scale=0.9]{$4^{-}_{1}$} ;
            \draw[Fuchsia] (F)--(H) node[near start, below right, scale=0.9]{$5^{+}_{1}$} node[near end, below right, scale=0.9]{$5^{-}_{1}$};
        \end{tikzpicture}
    \end{center}

\end{ex}

\medskip

In what follows, we will prove that the skew Brauer graph algebra $B$ associated to $\Gamma$ can be recovered from the Brauer graph algebra $B_{d}$ associated to the covering $\Gamma_{d}$ thanks to the construction of the quiver and relations of a skew group algebra for a cyclic group \cite[Section 2.3]{RR}. We first need to define an action of the cyclic group $G:=(\mathbb{C}_{2},.)$ of order 2 on $B_{d}$ whose generator will be denoted by $g$. Denoting by $e_{[h_{i}]}$ the idempotent in $B_{d}$ corresponding to the edge $[h_{i}]\in H_{d}/\iota_{d}$, there is a natural action of $G$ on $B_{d}=kQ_{d}/I_{d}$ given by 

\smallskip

\begin{itemize}[label=\textbullet, font=\tiny]
    \item $g\ .\ e_{[h_{i}]}=e_{[h_{i+1}]}$ for all $[h_{i}]\in H_{d}/\iota_{d}$;
    \item $g\ .\ \alpha_{h_{i}}=\alpha_{h_{i+1}}$ where $\alpha_{h_{i}}$ is the arrow in $Q_{d}$ induced by the half-edge $h_{i}\in H_{d}$. 
\end{itemize}

\smallskip

\noindent Similarly, the dual $\widehat{G}$ of $G$ is also a cyclic group and we denote by $\chi$ its generator. Denoting by $e_{[h]_{i}}$ for $i=\varnothing,0,1$ the idempotent(s) in $B$ arising from the edge $[h]\in H/\iota$, there is a natural action of $\widehat{G}$ on $B=kQ/I$ given by

\smallskip

\begin{itemize}[label=\textbullet, font=\tiny]
    \item $\chi\ .\ e_{[h]}=e_{[h]}$ for all $[h]\in H_{\circ}/\iota$ and $\chi\ .\ e_{[h]_{i}}=e_{[h]_{i+1}}$ for all $[h]\in H_{\times}/\iota$;
    \item $\chi\ .\ {}^{j}\alpha_{h}{}^{i}=(-1)^{-d(h)} \ {}^{j+1}\alpha_{h}{}^{i+1}$ where $^{j}\alpha_{h}^{i}$ is an arrow in $Q$ induced by the half-edge $h\in H$. By an abuse of notation, if $i=\varnothing$ then $i+1=\varnothing$.
\end{itemize}

\smallskip

Note that the $G$-action on $B_{d}$ arises from a $G$-action on $Q_{d}$ which preserves the ideal of relations $I_{d}$ whereas the $\widehat{G}$-action on $B$ does not come from an action on its quiver.

\medskip

\begin{prop} \label{prop:covering and skew group algebra}
    Let $f$ be the idempotent in $B_{d} \, G$ given by the sum of the $e_{[h_{0}]}\otimes 1_{G}$ for all $[h]\in H/\iota$. Then, the algebra $fB_{d}\, Gf$ has a natural $\widehat{G}$-action and the map 

    \smallskip

    \begin{equation*}
        \begin{aligned}
            \phi: &&B &&&\longrightarrow &&fB_{d}\, Gf \\
            &&e_{[h]} &&&\longmapsto &&f_{[h]}:=e_{[h_{0}]}\otimes 1_{G} \\
            &&e_{[h]_{i}} &&&\longmapsto &&f_{[h]_{i}}:=e_{[h_{0}]}\otimes \frac{1_{G}+(-1)^{i}g}{2} &&\mbox{for $i=0,1$ } \\
            &&^{j}\alpha_{h}{}^{i} &&&\longmapsto &&^{j}\beta_{h}{}^{i} &&\mbox{for $i,j=\varnothing,0,1$}
        \end{aligned}
    \end{equation*}

    \smallskip

    \noindent is an isomorphism of algebras which commutes with the $\widehat{G}$-actions, where $^{j}\beta_{h}^{i}:f_{[h]_{i}}\rightarrow f_{[h]_{j}}$ are the arrows in $fB_{d}\, Gf$.

\end{prop}

\medskip

\begin{proof}
Since $\chi\ . \ f=f$, the natural $\widehat{G}$-action on $B_{d}\, G$ extends onto a $\widehat{G}$-action on $fB_{d}\, Gf$. As for the proof of Proposition \ref{prop:Galois covering and skew group algebra}, the isomorphism follows again from the construction of the quiver and relations of $fB_{d}\, Gf$ given in \cite{RR}.

\begin{itemize}[label=\textbullet,font=\tiny]
\item The vertices of the quiver of $fB_{d}\, Gf$ are in bijection with the idempotents $f_{[h]_{i}}$ defined in the proposition for $i=\varnothing,0,1$. By definition, its arrows are of the form $^{j}\beta_{h}{}^{i}:f_{[h]_{i}}\rightarrow f_{[h]_{j}}$ for $i,j=\varnothing,0,1$.
\item The relations of $fB_{d}\, Gf$ are determined thanks to a representative in the $G$-orbits of the relations of $B_{d}$. Hence, its ideal of relation is generated by

\smallskip

\begin{enumerate}[label=(\Roman*')]
    \item \label{item:covering and skew group algebra 1} \[\sum_{\substack{i_{k}^{l}=\varnothing,0,1 \\ (k,l)\in \{1,\ldots,n\}\times\{1,\ldots,m(h)\}\backslash\{(n,m(h))\}}}\hspace{-1cm}{}^{i+1}\beta_{\sigma^{n-1}h}{}^{i_{n-1}^{m(h)}}\ldots{}^{i_{1}^{m(h)}}\beta_{h}{}^{i_{n}^{m(h)-1}}\ldots{}^{i_{n}^{1}}\beta_{\sigma^{n-1}h}{}^{i_{n-1}^{1}}\ldots{}^{i_{1}^{1}}\beta_{h}{}^{i}\]

    \noindent for all $i=0,1$ and $h\in H_{\times}$ inducing an arrow in $Q$, where $n$ is the size of the $\sigma$-orbit of $h$.

    \item \label{item:covering and skew group algebra 2}\[\left(\sum_{i_{1},\ldots,i_{n-1}=\varnothing,0,1}\hspace{-0.5cm}\beta_{\sigma^{n-1}h}{}^{i_{n-1}}\ldots{}^{i_{1}}\beta_{h}\right)^{m(h)}-\left(\sum_{i_{1},\ldots,i_{n'-1}-\varnothing,0,1}\hspace{-0.5cm}\beta_{\sigma^{n'-1}\iota h}{}^{i_{n-1}}\ldots{}^{i_{1}}\beta_{\iota h}\right)^{m(\iota h)}\]

    \noindent for all $h\in H_{\circ}$ such that $h$ and $\iota h$ both induce an arrow in $Q$, where $n$ and $n'$ are the size of the $\sigma$-orbit of $h$ and $\iota h$ respectively.

    \item \label{item:covering and skew group algebra 3} \[\sum_{\substack{i_{k}^{l}=\varnothing,0,1 \\ k=1,\ldots,n \\
    l=1,\ldots,m(h)}}{}^{j}\beta_{h}{}^{i_{n}^{m(h)}}\ {}^{i_{n}^{m(h)}}\beta_{\sigma^{n-1}h}{}^{i_{n-1}^{m(h)}}\ldots{}^{i_{1}^{m(h)}}\beta_{h}{}^{i_{n}^{m(h)-1}}\ldots{}^{i_{n}^{1}}\beta_{\sigma^{n-1}h}{}^{i_{n-1}^{1}}\ldots{}^{i_{1}^{1}}\beta_{h}{}^{i}\]

    \noindent for all $i,j=\varnothing,0,1$ and $h\in H$ inducing an arrow in $Q$, where $n$ is the size of the $\sigma$-orbit of $h$.

    \item \label{item:covering and skew group algebra 4}\[^{j}\beta_{\sigma h}{}^{0}\ {}^{0}\beta_{h}{}^{i}-{}^{j}\beta_{\sigma h}{}^{1}\ {}^{1}\beta_{h}{}^{i}\]

    \noindent for all $i,j=\varnothing,0,1$ and $h\in H$ inducing an arrow in $Q$ such that $\sigma h\in H_{\times}$.

    \item \[^{j}\beta_{\iota\sigma h}\ \beta_{h}{}^{i}\]

    \noindent \label{item:covering and skew group algebra 5}for all $i,j=\varnothing,0,1$ and $h\in H$ such that $h$ and $\iota \sigma h$ induce an arrow in $Q$ and $\sigma h\in H_{\circ}$.
\end{enumerate}
\end{itemize}

\smallskip

\noindent The relations are obtained the same way as in the proof of Proposition \ref{prop:Galois covering and skew group algebra}. In this case, there are sums appearing in the relations since we have

\[\alpha_{h_{-d(h)}}\otimes g^{-d(h)}=\sum_{i,j=\varnothing,0,1}{}^{i}\beta_{h}{}^{j}\]

\noindent Note that the relations \ref{item:covering and skew group algebra 1} and \ref{item:covering and skew group algebra 2} arise from the type \ref{item:Brauer graph algebra with multiplicity 1} relation in $B_{d}$, the relation \ref{item:covering and skew group algebra 3} arise from the type \ref{item:Brauer graph algebra with multiplicity 2} relation in $B_{d}$ and the relations \ref{item:covering and skew group algebra 4} and \ref{item:covering and skew group algebra 5} arise from the type \ref{item:Brauer graph algebra with multiplicity 3} relation in $B_{d}$. In particular, the relation \ref{item:covering and skew group algebra 4} means that all special $h_{i}$-cycle are equal in $fB_{d}\, Gf$. Hence, these five relations may be reformulated to obtain a correspondence with the relations of $B$ under the morphism $\phi$. 
\end{proof}

\medskip

Thanks to this proposition, one can understand skew Brauer graph algebras of multiplicity identically one as the trivial extension of a skew gentle algebra. In particular, this generalizes Theorem 3.13 in \cite{Schroll} which interprets Brauer graph algebras of multiplicity identically one as the trivial extension of gentle algebras. We recall that a skew gentle algebra is a basic algebra that is Morita equivalent to the skew group algebra of a gentle algebra equipped with a $\Z/2\Z$-action. (cf for instance Section 3.2 in \cite{AB}).

\medskip

\begin{defn} \label{def:admissible cut for skew Brauer graph}
Let $B=kQ/I$ be a skew Brauer graph algebra associated to $\Gamma=(H,\iota,\sigma,m)$ and $\Delta$ be a subset of $H$ consisting of representatives of each $\sigma$-orbits . An \textit{admissible cut} of $B$ is the set of all arrows in $Q$ induced by the half-edges in $\Delta$. By an abuse of notation, we also denote by $\Delta$ the admissible cut.
\end{defn}

\medskip

\begin{thm} \label{thm:skew Brauer graph algebras of multiplicity as trivial extension}
Let $B=kQ/I$ be a skew Brauer graph algebra of multiplicity identically one associated to $\Gamma=(H,\iota,\sigma)$ and equipped with an admissible cut $\Delta$. Then, the following holds.

\smallskip

\begin{enumerate}[label=(\arabic*)]
\item The algebra $B_{\Delta}:=kQ/\langle I\cup \Delta\rangle$ is a skew gentle algebra, where $\langle I\cup\Delta\rangle$ denotes the ideal of $kQ$ generated by $I\cup\Delta$.
\item The trivial extension of $B_{\Delta}$ is isomorphic to $B$.
\end{enumerate}
    
\end{thm}

\medskip

\begin{proof}
    Let us consider $d:H\rightarrow\Z/2\Z$ a $0$-homogeneous $\Z/2\Z$-grading  on $\Gamma$. Hence, one can construct $\Gamma_{d}=(H_{d},\iota_{d},\sigma_{d})$ the covering of $\Gamma$ which is a Brauer graph of multiplicity identically one by construction. We denote by $B_{d}$ the Brauer graph algebra associated to $\Gamma_{d}$ and by $p_{\Gamma}:\Gamma_{d}\rightarrow \Gamma$ the morphism of skew Brauer graphs given by the projection (cf before Example \ref{ex:covering of a skew Brauer graph algebra}). One can easily check that $p_{\Gamma}^{-1}(\Delta)$ is an admissible cut for $\Gamma_{d}$ that is stable under the $G$-action on $\Gamma_{d}$ defined before Proposition \ref{prop:covering and skew group algebra}, where $G=(\mathbb{C}_{2},.)$ denotes the cyclic group of order 2. By Theorem 3.13 in \cite{Schroll}, $(B_{d})_{p_{\Gamma}^{-1}(\Delta)}$ is a gentle algebra whose trivial extension is isomorphic to $B_{d}$. Moreover, $(B_{d})_{p_{\Gamma}^{-1}(\Delta)}$ has a natural $G$-action which can be extended as a $G$-action onto its trivial extension as explained before Proposition \ref{prop:trivial extension of a skew group algebra}. One can check that the isomorphism of algebras $\mathrm{Triv}((B_{d})_{p_{\Gamma}^{-1}(\Delta)})\simeq B_{d}$  commute with the $G$-actions. Thanks to Proposition \ref{prop:trivial extension of a skew group algebra}, we deduce the following isomorphisms of algebras

    \[B_{d}\, G\simeq \mathrm{Triv}((B_{d})_{p_{\Gamma}^{-1}(\Delta)})\, G\simeq \mathrm{Triv}((B_{d})_{p_{\Gamma}^{-1}(\Delta)}\, G)\]

    \noindent By Proposition \ref{prop:covering and skew group algebra}, we have the following isomorphisms of algebras

    \[B\simeq fB_{d}\, Gf\simeq f\, \mathrm{Triv}((B_{d})_{p_{\Gamma}^{-1}(\Delta)}\, G)f\]

    \noindent Moreover, denoting by $A$ the algebra $(B_{d})_{p_{\Gamma}^{-1}(\Delta)}\, G$, one can check that the following morphism also gives an isomorphism of algebras

    \smallskip

    \begin{equation*}
        \begin{aligned}
            &&f\,\mathrm{Triv}(A)f &&\overset{\sim}{\longrightarrow} &&&\mathrm{Triv}(fAf) \\
            &&f.(a,\phi).f &&\longmapsto &&& (faf,(f\phi f)_{|fAf}) \\
            &&f.(b,\psi').f &&\longmapsfrom &&&(fbf,\psi)
        \end{aligned}
    \end{equation*}

    \smallskip

    \noindent where $\psi':A\rightarrow k$ is defined by $\psi'(c)=\psi(fcf)$ for all $c\in A$. Furthermore, one can easily check that the isomorphism of algebras given in Proposition \ref{prop:covering and skew group algebra} restricts to an isomorphism of algebras between $B_{\Delta}=B/\langle\Delta\rangle$ and $fAf=f(B_{d}/\langle p_{\Gamma}^{-1}(\Delta)\rangle )\, Gf$. This concludes the proof.
\end{proof}

\medskip

\subsection{Generalized Kauer moves}

\medskip

Our goal is to define generalized Kauer moves for skew Brauer graph algebras so that they can be interpreted in terms of silting mutations. The idea of the proof is analogous to the one used in Section 2. We will use the covering defined in the previous subsection which can be constructed whenever the skew Brauer graph is equipped with a 0-homogeneous $\Z/2\Z$-grading. Hence, we need to define a graded version of these generalized Kauer moves. We begin by defining them for sectors (cf Definition \ref{def:sector}).

\medskip

\begin{defn}\label{def:graded generalized Kauer moves of a sector for skew Brauer graph algebras}
    Let $(\Gamma,d)=(H,\iota,\sigma,m,d)$ be a $\Z/2\Z$-graded skew Brauer graph and $H'\subset H$ be stable under $\iota$. The \textit{$\Z/2\Z$-graded generalized Kauer move} of a sector $(h,r)\in\mathrm{sect}(H',\sigma)$ in $\Gamma$ gives rise to a $\Z/2\Z$-graded skew Brauer graph $\mu^{+}_{(h,r)}(\Gamma,d)=(H,\iota,\sigma_{(h,r)},m_{(h,r)},d_{(h,r)})$ where

\smallskip

\begin{equation*}
    \begin{aligned}
        \sigma_{(h,r)}=(h \ \sigma^{r+1}h)\sigma(\sigma^{r}h \ \iota\sigma^{r+1}h) \qquad \mbox{and} \qquad 
        m_{(h,r)}:  &&H &&&\rightarrow &&\Z_{>0} \\
        &&\sigma^{i}h &&&\mapsto &&m(\iota\sigma^{r+1}h) \quad \mbox{for $i=0,\ldots,r$} \\
        &&h' &&&\mapsto &&m(h') \quad \mbox{for $h'\neq\sigma^{i}h$}
    \end{aligned}
\end{equation*}

\smallskip

\noindent and where the $\Z/2\Z$-grading $d_{(h,r)}:H\rightarrow \Z/2\Z$ is defined by

\smallskip

\begin{align*}
        d_{(h,r)}: \ &\iota\sigma^{r+1}h &&\mapsto &&\left\{\begin{aligned}&-\sum_{i=0}^{r}d(\sigma^{i}h) &&\mbox{if $\sigma^{r+1}h\in H_{\circ}$} \\
        &-\sum_{i=0}^{r}d(\sigma^{i}h)-1 &&\mbox{if $\sigma^{r+1}h\in H_{\times}$}\end{aligned}\right. \\
        &\sigma^{r}h &&\mapsto &&\left\{\begin{aligned}
            &d(\iota\sigma^{r+1}h)+d(\sigma^{r}h) &&\mbox{if $\iota\sigma^{r+1}h\neq\sigma^{-1}h$ and $\sigma^{r+1}h\in H_{\circ}$} \\
            &d(\iota\sigma^{r+1}h)+d(\sigma^{r}h)+1 &&\mbox{if $\iota\sigma^{r+1}h\neq\sigma^{-1}h$ and $\sigma^{r+1}h\in H_{\times}$} \\
            &\sum_{i=-1}^{r}d(\sigma^{i}h)+d(\sigma^{r}h) &&\mbox{if $\iota\sigma^{r+1}h=\sigma^{-1}h$ and $\sigma^{r+1}h\in H_{\circ}$} \\
            &\sum_{i=-1}^{r}d(\sigma^{i}h)+d(\sigma^{r}h)+1 &&\mbox{if $\iota\sigma^{r+1}h=\sigma^{-1}h$ and $\sigma^{r+1}h\in H_{\times}$}
        \end{aligned}\right. \\ 
        &\sigma^{-1}h &&\mapsto &&\left\{\begin{aligned}
            &\sum_{i=-1}^{r}d(\sigma^{i}h) &&\mbox{if $\iota\sigma^{r+1}h\neq\sigma^{-1}h$} \\
            &d_{(h,r)}(\iota\sigma^{r+1}h) &&\mbox{if $\iota\sigma^{r+1}h=\sigma^{-1}h$}   
        \end{aligned}\right. \\
        &h' &&\mapsto &&d(h') \qquad \mbox{for $h'\neq\iota\sigma^{r+1}h,\sigma^{r}h,\sigma^{-1}h$}
\end{align*}

\smallskip

\noindent The underlying Brauer graph $\mu^{+}_{(h,r)}(\Gamma)=(H,\iota,\sigma_{(h,r)},m_{(h,r)})$ can be obtained from $\Gamma$ as in Figure \ref{Generalized Kauer move of a sector} and in Remark \ref{rem:special case of generalized Kauer move of a sector} when $\sigma^{r+1}h\in H_{\circ}$ and as follows when $\sigma^{r+1}h\in H_{\times}$

\smallskip

\begin{figure}[H]
    \centering
              
    \begin{tikzpicture}[scale=0.9]
        \tikzstyle{vertex}=[draw, circle, scale=0.6, minimum size=0.8cm]
        \begin{scope}[xshift=-3cm]
        \node[vertex] (1) at (-2,0) {m};
        \draw (1)--(0,0) node[midway, above]{$\sigma^{r+1}h$}  node[scale=1.3]{\textbf{\texttimes}};
        \draw (1)--(-2.5,-1.75) node[below]{$\sigma^{-1}h$};
        \draw[Orange] (1)--(-1.5,-1.75) node[below, right]{$h$};
        \draw[Orange] (1)--(-1,-1.25) node[below, right]{$\sigma^{j}h$};
        \draw[Orange] (1)--(-0.5,-0.75) node[below, right]{$\sigma^{r}h$};
        \draw (1)--(-2.5,1.75) node[above]{$\sigma^{r+2}h$};
        \draw[|-|, Orange] (-1.5,-2.2)--(0,-2.2) node[below, midway]{$(h,r)$};
        \draw (-1.25,-3.5) node{$\Gamma=(H,\iota,\sigma,m)$};
        \end{scope} 
        
        \draw[->] (-1,0)--(2.5,0);

        \begin{scope}[xshift=6.5cm]
        \node[vertex] (1) at (-2,0) {m};
        \draw (1)--(0,0) node[midway, below,yshift=-1mm]{$\sigma^{r+1}h$}  node[scale=1.3]{\textbf{\texttimes}};
        \draw (1)--(-2.5,-1.75) node[below]{$\sigma^{-1}h$};
        \draw[Orange] (1)--(-1.5,1.75) node[below, right]{$\sigma^{r}h$};
        \draw[Orange] (1)--(-1,1.25) node[below, right]{$\sigma^{j}h$};
        \draw[Orange] (1)--(-0.5,0.75) node[below, right]{$h$};
        \draw (1)--(-2.5,1.75) node[above]{$\sigma^{r+2}h$};
        \draw (-1.25,-3.5) node{$\mu^{+}_{(h,r)}(\Gamma)=(H,\iota,\sigma_{(h,r)},m_{(h,r)})$};
        \end{scope}
        \end{tikzpicture}
    \caption{Generalized Kauer move of a sector $(h,r)$ when $\sigma^{r+1}h$ is in $H_{\times}$}
    \label{Generalized Kauer move of a sector in skew Brauer graph algebra}
\end{figure}
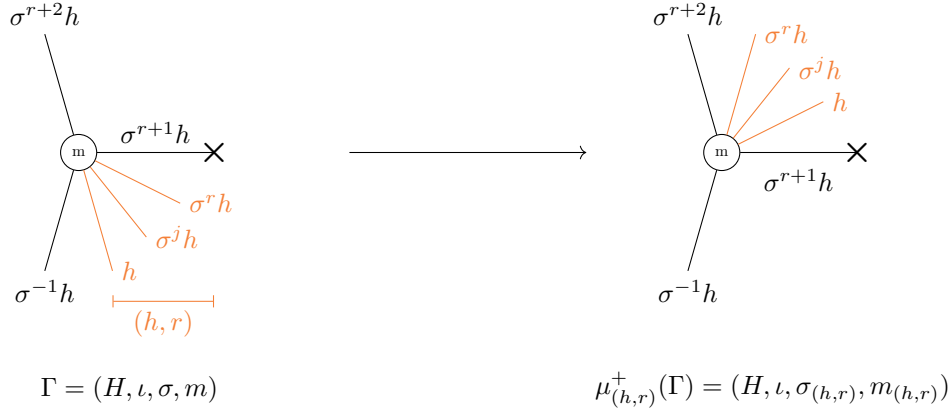
\end{defn}

\medskip

\begin{rem} \label{rem:special case of generalized Kauer move of a sector for skew Brauer graph}
    In the special case where $\iota\sigma^{r+1}h=\sigma^{-1}h$, then the $\Z/2\Z$-graded generalized Kauer move of the sector $(h,r)$ does not change the underlying Brauer graph which is given in Remark \ref{rem:special case of generalized Kauer move of a sector} when $\sigma^{r+1}h\in H_{\circ}$ and as follows when $\sigma^{r+1}h\in H_{\times}$

    
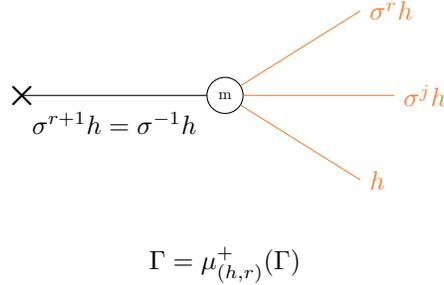
\smallskip

    \begin{center}
    \begin{tikzpicture}[scale=0.9]
        \tikzstyle{vertex}=[draw, circle, scale=0.6, minimum size=0.8cm]

        \node[vertex] (1) at (-2,0) {m};
        \draw (1)--(-5,0) node[midway, below,yshift=-1mm]{$\sigma^{r+1}h=\sigma^{-1}h$}  node[scale=1.3]{\textbf{\texttimes}};
        \draw[Orange] (1)--(0,-1.25) node[below, right]{$h$};
        \draw[Orange] (1)--(0.5,0) node[below, right]{$\sigma^{j}h$};
        \draw[Orange] (1)--(0,1.25) node[below, right]{$\sigma^{r}h$};
        \draw (-2,-2.5) node{$\Gamma=\mu^{+}_{(h,r)}(\Gamma)$};
        \end{tikzpicture}
        \end{center}

        \smallskip

        \noindent Similarly to Remark \ref{rem:special case of generalized Kauer move of a sector}, the degree of the  $\sigma^{i}h$ will a priori change in the process. In particular, the covering of $\Gamma$ and $\mu^{+}_{(h,r)}(\Gamma)$ could be different even in this case.
\end{rem}

\medskip

By construction of the generalized Kauer move described in Figure \ref{Generalized Kauer move of a sector} and in the previous figure, it is clear that $m_{(h,r)}$ is constant on the $\sigma_{(h,r)}$-orbits. Moreover, $d_{(h,r)}:H\rightarrow\Z/2\Z$ is 0-homogeneous. Indeed, one can construct a bijection $\phi:H/\sigma\rightarrow H/\sigma_{(h,r)}$ satisfying

\begin{equation*}
    \begin{aligned}
        &s_{(h,r)}(\sigma^{i}h)=\phi(s(\iota\sigma^{r+1}h)) &&\mbox{for $i=0,\ldots, r$} \\
        &s_{(h,r)}(h')=\phi(s(h')) &&\mbox{for $h'\neq \sigma^{i}h, i=0\ldots,r$}
    \end{aligned}
\end{equation*}

\noindent where $s:H\rightarrow H/\sigma$ and $s_{(h,r)}:H\rightarrow H/\sigma_{(h,r)}$ are the source map for $\circ$-vertices of $\Gamma$ and $\mu^{+}_{(h,r)}(\Gamma)$ respectively. Then, one can easily check that for all vertex $v\in H/\sigma_{(h,r)}$ in $\mu^{+}_{(h,r)}(\Gamma)$

\[\sum_{h'\in H, s_{(h,r)}(h')=v}d_{(h,r)}(h')=\sum_{h'\in H, s(h')=\phi^{-1}(v)}d(h')\]

\medskip

\begin{rem}
    By definition, if $H_{\times}=\varnothing$ then $\Gamma$ is a Brauer graph and by construction $\mu^{+}_{(h,r)}(\Gamma)$ is also a Brauer graph. Hence, this definition generalizes Definition \ref{def:graded generalized Kauer moves with multiplicity of a sector}.
\end{rem}

\medskip

Thanks to Lemma 2.4 in \cite{Soto}, we can consider successive $\Z/2\Z$-graded generalized Kauer moves. In the following, we denote by

\[\mu^{+}_{(h_{1},r_{1})(h_{2},r_{2})}(\Gamma,d)=(H,\iota,\sigma_{(h_{1},r_{1})(h_{2},r_{2})},m_{(h_{1},r_{1})(h_{2},r_{2})},d_{(h_{1},r_{1})(h_{2},r_{2})})\]

\noindent the $\Z/2\Z$-graded skew Brauer graph defined by $\mu^{+}_{(h_{1},r_{1})}(\mu^{+}_{(h_{2},r_{2})}(\Gamma,d))$. Using similar computations as in the proof of Proposition 2.5 in \cite{Soto}, we obtain the following result.

\medskip

\begin{prop} \label{prop:successive graded generalized Kauer moves in skew Brauer graph algebras}
    Let $\Gamma=(H,\iota,\sigma,m)$ be a skew Brauer graph and $H'$ be a subset of $H$ stable under $\iota$. Let $(h_{1},r_{1}),(h_{2},r_{2})$ be two distinct maximal sectors in $\Gamma$ of elements in $H'$. For any $0$-homogeneous $\Z/2\Z$-grading $d:H\rightarrow\Z/2\Z$ of $\Gamma$, we have

    \[\mu^{+}_{(h_{1},r_{1})(h_{2},r_{2})}(\Gamma,d)=\mu^{+}_{(h_{2},r_{2})(h_{1},r_{1})}(\Gamma,d)\]
\end{prop}

\medskip

\begin{defn}\label{def:graded generalized Kauer moves in skew Brauer graph algebra}
Let $(\Gamma,d)=(H,\iota,\sigma,m,d)$ be a $\Z/2\Z$-graded skew Brauer graph and $H'$ be a subset of $H$ stable under $\iota$. The \textit{$\Z/2\Z$-graded generalized Kauer move} of $H'$ is the succession of the $\Z/2\Z$-graded generalized Kauer move of all the maximal sectors $(h,r)\in\mathrm{Sect}(H',\sigma)$ in $\Gamma$. This gives rise to a $\Z/2\Z$-graded skew Brauer graph that will be denoted by $\mu^{+}_{H'}(\Gamma,d)=(H,\iota,\sigma_{H'},m_{H'},d_{H'})$.
    
\end{defn}

\medskip

\begin{ex} \label{ex:generalized Kauer move in skew Brauer graph}
Let us consider $\Gamma=(H,\iota,\sigma,m,d)$ the $\Z/2\Z$-graded skew Brauer graph defined in Example \ref{ex:covering of a skew Brauer graph algebra}. Let $H'=\{1^{+},1^{-},4^{+},4^{-}\}$ be a subset of H stable under $\iota$. Then, the $\Z/2\Z$-graded skew Brauer graph $\mu^{+}_{H'}(\Gamma,d)=(H,\iota,\sigma_{H'},m_{H'},d_{H'})$ obtained from a $\Z/2\Z$-graded generalized Kauer move is given by

\smallskip

 
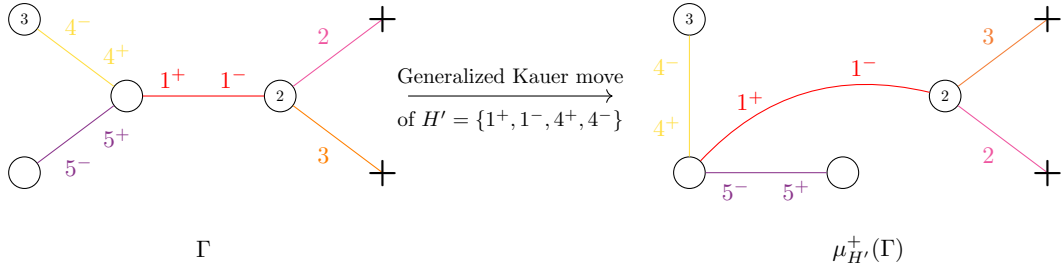
\begin{center}
        \begin{tikzpicture}[scale=0.68]
            \tikzstyle{vertex}=[circle,draw,scale=0.6,minimum size=0.7cm]
            \begin{scope}[xshift=-6.5cm]
            \node[vertex] (A) at (0,0) {2};
            \node[vertex] (B) at (-3,0) {};
            \node[vertex] (C) at (-5,1.5) {3};
            \node[vertex] (D) at (-5,-1.5) {};
            \draw[red] (A)--(B) node[near start, above,scale=0.9]{$1^{-}$} node[near end, above,scale=0.9]{$1^{+}$};
            \draw[VioletRed] (A)--(2,1.5) node[rotate=45, scale=1.3, black]{\textbf{\texttimes}} node[midway, above left,scale=0.9]{$2$};
            \draw[orange] (A)--(2,-1.5) node[rotate=45, scale=1.3, black]{\textbf{\texttimes}} node[midway, below left,scale=0.9]{$3$};
            \draw[Goldenrod] (B)--(C) node[near start, above right,scale=0.9]{$4^{+}$} node[near end, above right,scale=0.9]{$4^{-}$} ;
            \draw[Fuchsia] (B)--(D) node[near start, below right,scale=0.9]{$5^{+}$} node[near end, below right,scale=0.9]{$5^{-}$};
            \draw (-1.5,-3) node[scale=0.9]{$\Gamma$};
            \end{scope}
            \draw[->] (-4,0)--(0,0) node[scale=0.8,above, midway, yshift=1mm]{Generalized Kauer move} node[scale=0.8, below, yshift=-1mm, midway]{of $H'=\{1^{+},1^{-},4^{+},4^{-}\}$};
            \begin{scope}[xshift=6.5cm]
            \node[vertex] (A) at (0,0) {2};
            \node[vertex] (B) at (-2,-1.5) {};
            \node[vertex] (C) at (-5,1.5) {3};
            \node[vertex] (D) at (-5,-1.5) {};
            \draw[red] (A) to[bend right=30] node[near start, above,scale=0.9]{$1^{-}$} node[near end, above,scale=0.9]{$1^{+}$} (D) ;
            \draw[Orange] (A)--(2,1.5) node[rotate=45, scale=1.3, black]{\textbf{\texttimes}} node[midway, above left,scale=0.9]{$3$};
            \draw[VioletRed] (A)--(2,-1.5) node[rotate=45, scale=1.3, black]{\textbf{\texttimes}} node[midway, below left,scale=0.9]{$2$};
            \draw[Goldenrod] (D)--(C) node[near start, left,scale=0.9]{$4^{+}$} node[near end, left,scale=0.9]{$4^{-}$} ;
            \draw[Fuchsia] (B)--(D) node[near start, below,scale=0.9]{$5^{+}$} node[near end, below,scale=0.9]{$5^{-}$};
            \draw (-1.5,-3) node[scale=0.9]{$\mu^{+}_{H'}(\Gamma)$};
            \end{scope}
        \end{tikzpicture}
    \end{center}

\smallskip

\noindent where the orientation $\sigma_{H'}$ is  $(5^{-}\ 1^{+} \ 4^{+})(1^{-} \ 2 \ 3)=(1^{-}\ 3)(1^{+}\ 5^{+})\sigma (4^{+} \ 5^{-})(1^{-} \ 3)$, the multiplicity $m_{H'}:H\rightarrow\Z_{>0}$ is given by $m_{H'}(4^{-})=3$, $m_{H'}(1^{-})=m_{H'}(2)=m_{H'}(3)=2$ and $m_{H'}(5^{-})=m_{H'}(1^{+})=m_{H'}(4^{+})=m_{H'}(5^{+})=1$ and the $\Z/2\Z$-grading $d_{H'}:H\rightarrow\Z/2\Z$ is given by $d_{H'}(3)=-1$, $d_{H'}(1^{-})=1$ and $d_{H'}(4^{-})=d_{H'}(4^{+})=d_{H'}(1^{+})=d_{H'}(5^{-})=d_{H'}(5^{+})=d_{H'}(2)=0$.
\end{ex}

\medskip

Considering a $\Z/2\Z$-graded generalized Kauer move of a $\Z/2\Z$-graded skew Brauer graph $(\Gamma,d)$, one can construct the covering of $\mu^{+}_{H'}(\Gamma,d)$ as defined before Example \ref{ex:covering of a skew Brauer graph algebra}. The following proposition shows a commutativity between constructing this covering and applying a generalized Kauer move.

\medskip

\begin{prop}\label{prop:commutativity covering and generalized Kauer move in skew Brauer graphs}
    Let $(\Gamma,d)=(H,\iota,\sigma,d)$ be a $\Z/2\Z$-graded skew Brauer graph and $H'$ be a subset of $H$ stable under $\iota$. We denote by $\Gamma_{d}=(H_{d},\iota_{d},\sigma_{d},m_{d})$ the covering of $\Gamma$ constructed before Example \ref{ex:covering of a skew Brauer graph algebra} and $H'_{d}=H'\times\Z/2\Z\subset H_{d}$. Then, the covering of $\mu^{+}_{H'}(\Gamma,d)$ is the Brauer graph $\mu^{+}_{H'_{d}}(\Gamma_{d})$.
\end{prop}

\medskip

The previous proposition can be summarized in the following commutative diagram 

\smallskip

\begin{center}
    \begin{tikzcd}[column sep=5cm, row sep=2cm]
        \Gamma_{d} \arrow{r}[above]{\mbox{generalized Kauer}}[below]{\mbox{ move of $H'_{d}$}} \arrow{d}[left]{p_{\Gamma}} &\mu^{+}_{H'_{d}}(\Gamma_{d}) \arrow{d}[right]{p_{\mu^{+}(\Gamma)}} \\
        (\Gamma,d) \arrow{r}[above]{\mbox{$\Z/2\Z$-graded generalized}}[below]{\mbox{Kauer move of $H'$}}&\mu^{+}_{H'}(\Gamma,d)
    \end{tikzcd}
\end{center}

\smallskip

\noindent where $p_{\Gamma}:\Gamma_{d}\rightarrow\Gamma$ and $p_{\mu^{+}_{H'}}(\Gamma):\mu^{+}_{H'_{d}}(\Gamma_{d})\rightarrow\mu^{+}_{H'}(\Gamma)$ are the morphisms of skew Brauer graphs induced by the natural projection and $\mu^{+}_{H'}(\Gamma)$ denotes the underlying skew Brauer graph of $\mu^{+}_{H'}(\Gamma,d)$.

\medskip

\begin{proof}
    One can easily check that $H'_{d}$ is a subset of $H_{d}$ stable under $\iota_{d}$ and that $(h,r)\in\mathrm{Sect}(H',\sigma)$ is a maximal sector in $\Gamma$ if and only if $(h_{i},r)\in\mathrm{Sect}(H'_{d},\sigma_{d})$ is a maximal sector in $\Gamma_{d}$ for all $i\in\Z/2\Z$. By Proposition \ref{prop:successive graded generalized Kauer moves in skew Brauer graph algebras}, it suffices to prove that the covering of $\mu^{+}_{(h,r)}(\Gamma,d)$ is given by

    \[\mu^{+}_{p_{\Gamma}^{-1}(h,r)}(\Gamma_{d})=(H_{d},\iota_{d},(\sigma_{d})_{p_{\Gamma}^{-1}(h,r)},(m_{d})_{p_{\Gamma}^{-1}(h,r)})\]

    \noindent for any maximal sector $(h,r)\in\mathrm{Sect}(H',\sigma)$, where $p_{\Gamma}^{-1}(h,r)$ is the product of the maximal sectors $(h_{i},r)$ for $i\in\Z/2\Z$. We denote by 

    \[\mu^{+}_{(h,r)}(\Gamma)_{d_{(h,r)}}=(H_{d_{(h,r)}},\iota_{d_{(h,r)}},(\sigma_{(h,r)})_{d_{(h,r)}}, (m_{(h,r)})_{d_{(h,r)}})\]

    \noindent the covering of $\mu^{+}_{(h,r)}(\Gamma,d)$. By definition, these two Brauer graphs have the same set of half-edges. By construction of the covering, it is clear that they have the same pairing since it only depends on $H_{\circ}$ and $H_{\times}$. Moreover, using similar computations as in the proof of Proposition \ref{prop:commutativity Galois covering and generalized Kauer move with multiplicity}, one can check that their orientations are equal. It remains to prove the equality of the multiplicities. It is clear that the equality holds for any half-edge different from the $\sigma_{d}^{i}h_{j}$. Moreover, we have

    \begin{equation*}
        \begin{aligned}
            &(m_{d})_{p^{-1}(h,r)}(\sigma_{d}^{i}h_{j})=m_{d}(\iota_{d}\sigma_{d}^{r+1}h_{j})=m(\iota\sigma^{r+1}h) \\
            &(m_{(h,r)})_{d(h,r)}(\sigma_{d}^{i}h_{j})=m_{(h,r)}(\sigma^{i}h)=m(\iota\sigma^{r+1}h)
        \end{aligned}
    \end{equation*}

    \noindent This concludes the proof of the proposition. 
\end{proof}

\medskip

\subsection{Compatibility with silting mutations}

\medskip

The goal of this part is to prove the following theorem which is an analogous version of Theorem 3.10 in \cite{Soto} for skew Brauer graph algebras.

\medskip

\begin{thm} \label{thm:generalized Kauer moves for skew Brauer graphs}
    Let $\Gamma=(H,\iota,\sigma,m)$ be a skew Brauer graph and $H'$ be a subset of $H$ stable under $\iota$. We assume that 2 and $\overline{m}$, the least common multiple of the $m(h)$ for $h\in H$, are invertible the field $k$. Denoting by $B$ and $B'$ the skew Brauer graph algebras associated respectively to $\Gamma$ and $\mu^{+}_{H'}(\Gamma)$, there is an equivalence of triangulated categories

    \smallskip

    \begin{center}
        \begin{tikzcd}[column sep=3.5cm]
            \per{B'} \arrow[cramped]{r}[above]{-\lotimes{B'}{\ \mu^{+}(B\, ;\, e_{H''}B)}} &\per{B}
        \end{tikzcd}
    \end{center}

    \smallskip

    \noindent where $e_{H''}B$ is the projective $B$-module corresponding to the edges in $(H\backslash H')/\iota$.
\end{thm}

\medskip

The idea of the proof is similar to the proof of Theorem \ref{thm:generalized Kauer moves with multiplicity}. The following result is an analogous version of Proposition \ref{prop:generalized Kauer move and skew group algebra} for skew Brauer graph algebras.

\medskip

\begin{prop} \label{prop:generalized Kauer move and skew group algebra for skew Brauer graph}
    Let us consider the setting of Proposition \ref{prop:commutativity covering and generalized Kauer move in skew Brauer graphs}. We assume that 2 and $\overline{m}$, the least common multiple of the $m(h)$ for $h\in H$, are invertible in the field $k$. We denote by $B_{d}$ and $B_{d}'$ the Brauer graph algebras associated to $\Gamma_{d}$ and $\mu^{+}_{H'_{d}}(\Gamma_{d})$ respectively. Let $G$ be the cyclic group $(\mathbb{C}_{2},.)$ of order 2. Then, there is an equivalence of triangulated categories

    \smallskip

    \begin{center}
        \begin{tikzcd}[column sep=4.5cm]
            \per{B_{d}'\, G} \arrow[cramped]{r}[above]{-\lotimes{B_{d}'\, G}{\ \mu^{+}(B_{d}\, G\, ;\, e_{H''_{d}}B_{d}\, G)}} &\per{B_{d}\, G}
        \end{tikzcd}
    \end{center}

    \smallskip

    \noindent where $e_{H''_{d}}B_{d}$ denotes the projective $B_{d}$-module corresponding to the edges in $(H_{d}\backslash H'_{d})/\iota_{d}$.
    
\end{prop}

\medskip

\begin{proof}
    By Theorem \ref{thm:generalized Kauer moves with multiplicity}, there is an equivalence of triangulated categories

    \smallskip

    \begin{center}
        \begin{tikzcd}[column sep=2cm]
            \per{B'_{d}} \arrow[cramped]{r}[above]{-\lotimes{B'_{d}}{\ T_{d}}} &\per{B_{d}}
        \end{tikzcd}
    \end{center}

    \smallskip

    \noindent where $T_{d}:=\mu^{+}(B_{d}\, ;\, e_{H''_{d}}B_{d})$. Similarly to the proof of Proposition \ref{prop:generalized Kauer move and skew group algebra}, it suffices to prove that $T_{d}$ is $G$-invariant and that the assumptions of Proposition \ref{prop:silting mutation in skew group algebras} hold. Let us first check that $T_{d}$ is indeed $G$-invariant. Considering the same notations as in the proof of Proposition \ref{prop:generalized Kauer move and skew group algebra}, the left minimal $\mathrm{add}(e_{H''_{d}}B_{d})$-approximation of $e_{H'_{d}}B_{d}$ is given by (cf Figure \ref{fig:left minimal approximation})

    \smallskip

    \begin{center}
       $\alpha_{[h_{i}]}$ : \begin{tikzcd}[column sep=3cm, ampersand replacement=\&] e_{[h_{i}]}B_{d} \arrow[cramped]{r}[above]{\begin{pmatrix} \alpha(h_{i},H'_{d}) \\ \alpha(\iota_{d}h_{i},H'_{d})           
       \end{pmatrix}} \&e_{[\sigma_{d}^{r(h_{i})+1}h_{i}]}B_{d} \,\oplus\, e_{[\sigma_{d}^{r(\iota_{d}h_{i})+1}\iota_{d}h_{i}]}B_{d} 
        \end{tikzcd}
    \end{center}
    
    \smallskip

    \noindent Moreover, by definition of the left mutation, we have

    \smallskip

    \begin{equation*}
            \mu^{+}(B_{d}\, ;\, e_{H''_{d}}B_{d})
            =e_{H''_{d}}B_{d} \oplus \bigoplus_{[h_{i}]\in H'_{d}/\iota_{d}} \mathrm{Cone}(\alpha_{[h_{i}]}) 
            =e_{H''_{d}}B_{d} \oplus \bigoplus_{[h]\in H'/\iota} \mathrm{Cone}(\beta_{[h]})
    \end{equation*}

    \smallskip

    \noindent where $\beta_{[h]}:X_{[h]}\rightarrow Y_{[h]}$ is defined as follows
    
    \smallskip

    \begin{equation*}
        \begin{gathered}
            \beta_{[h]}=\left\{\begin{aligned}
            &\bigoplus_{i\in\Z/2\Z} \alpha_{[h_{i}]} &&\mbox{if $[h]\in H'_{\circ}/\iota$} \\
            &\alpha_{[h_{0}]} &&\mbox{if $[h]\in H'_{\times}/\iota$}
            \end{aligned}\right. \qquad \mbox{with} \\[0.5cm]
            X_{[h]}=\left\{\begin{aligned}& \bigoplus_{i\in\Z/2\Z} e_{[h_{i}]}B_{d} &&\mbox{if $[h]\in H'_{\circ}/\iota$} \\
            &e_{[h_{0}]}B_{d} &&\mbox{if $[h]\in H'_{\times}/\iota$}
            \end{aligned}\right. \qquad \mbox{and} \\[0.5cm]
            Y_{[h]}=\left\{\begin{aligned}&\bigoplus_{i\in\Z/2\Z} (e_{[\sigma_{d}^{r(h_{i})+1}h_{i}]}B_{d} \oplus e_{[\sigma_{d}^{r(\iota_{d}h_{i})+1}\iota_{d}h_{i}]}B_{d}) &&\mbox{if $[h]\in H'_{\circ}/\iota$} \\
            &e_{[\sigma_{d}^{r(h_{0})+1}h_{0}]}B_{d} \oplus e_{[\sigma_{d}^{r(h_{1})+1}h_{1}]}B_{d} &&\mbox{if $[h]\in H'_{\times}/\iota$}
            \end{aligned}\right.
            \end{gathered}
    \end{equation*}

    \smallskip

    \noindent where $H'_{\circ}=H'\cap H_{\circ}$ and $H'_{\times}=H'\cap H_{\times}$. Since $G$ acts bijectively on the set of edges $(H_{d}\backslash H'_{d})/\iota_{d}$, the object $e_{H''_{d}}B_{d}$ is $G$-invariant by Example \ref{ex:G-invariant object} \ref{item:G-invariant object 1}. Moreover, one can easily check that $X_{[h]}$ and $Y_{[h]}$ are $G$-invariant objects for all $[h]\in H'/\iota$ thanks to Example \ref{ex:G-invariant object} \ref{item:G-invariant object 1}. In this case, the isomorphisms $\iota_{g}^{X_{[h]}}$ and $\iota_{g}^{Y_{[h]}}$ (cf Definition \ref{def:G-invariant object}) are respectively given by the action of $g$ on $X_{[h]}$ and $Y_{[h]}$. Furthermore, using the definition of the $\alpha_{[h_{i}]}$, it is not hard to check that the following diagram commutes for all $g\in G$.

    \smallskip

    \begin{center}
        \begin{tikzcd}[column sep=2cm, row sep=2cm]
            X_{[h]}^{g^{-1}} \arrow[cramped]{d}[left]{\iota_{g}^{X_{[h]}}} \arrow[cramped]{r}[above]{\beta_{[h]}^{g^{-1}}} &Y_{[h]}^{g^{-1}} \arrow[cramped]{d}[right]{\iota_{g}^{Y_{[h]}}} \\
            X_{[h]} \arrow[cramped]{r}[below]{\beta_{[h]}} &Y_{[h]}
        \end{tikzcd}
    \end{center}

    \smallskip

    \noindent By Example \ref{ex:G-invariant object} \ref{item:G-invariant object 2}, we deduce that $\mathrm{Cone}(\beta_{[h]})$ is $G$-invariant for all $[h]\in H'/\iota$. Hence, we conclude that $\mu^{+}(B_{d}\, ;\, e_{H''_{d}}B_{d})$ is indeed $G$-invariant. With similar arguments than in the proof of Proposition \ref{prop:generalized Kauer move and skew group algebra}, one can check that the assumptions of Proposition \ref{prop:silting mutation in skew group algebras} holds. This concludes the proof. 
    \end{proof}

    \medskip

    \begin{proof}[Proof of Theorem \ref{thm:generalized Kauer moves for skew Brauer graphs}]
        Let us consider the 0-homogeneous $\Z/2\Z$-grading $d:H\rightarrow\Z/2\Z$ defined by $d(h)=0$ for all $h\in H$. Let $\Gamma_{d}=(H_{d},\iota_{d},\sigma_{d},m_{d})$ be the covering of $\Gamma$ constructed after Definition \ref{def:graded skew Brauer graph} and $H'_{d}=H'\times\Z/2\Z\subset H_{d}$. We denote by $B_{d}$ and $B_{d}'$ the Brauer graph algebras associated to $\Gamma_{d}$ and $\mu^{+}_{H'_{d}}(\Gamma_{d})$ respectively. Moreover, $G$ denotes the cyclic group $(\mathbb{C}_{2},.)$ of order 2. We have the following commutative diagram

        \smallskip

        \begin{center}
            \begin{tikzcd}[column sep=3cm, row sep=2cm]
                \per{B_{d}'\, G} \arrow[cramped]{d}[left]{-\lotimes{B'_{d}\, G}{\ B'_{d}\, Gf}}\arrow[cramped]{r}[above]{-\lotimes{B_{d}'\, G}{\ T_{d}\, G}} &\per{B_{d}\, G} \arrow[cramped]{d}[right]{-\lotimes{B_{d}\, G}{\ B_{d}\, Gf}} \\
                \per{B'} \arrow[cramped, dashed]{r}[below]{-\lotimes{B'}{\ T}} &\per{B}
            \end{tikzcd}
        \end{center}

        \smallskip

        \noindent where $f$ is the idempotent given by the sum of the $e_{[h_{0}]}\otimes 1_{G}$ for all $[h]\in H/\iota$ and $T_{d}\, G=\mu^{+}(B_{d}\, G\,;\, e_{H''_{d}}B_{d}\, G)$. The vertical arrows arise from Proposition \ref{prop:covering and skew group algebra} since $\mu^{+}_{H'_{d}}(\Gamma_{d})$ is the covering of $\mu^{+}_{H'}(\Gamma,d)$ by Proposition \ref{prop:commutativity covering and generalized Kauer move in skew Brauer graphs}. The top arrow is given by Proposition \ref{prop:generalized Kauer move and skew group algebra for skew Brauer graph}. This leads to a derived equivalence between $B'$ and $B$ and the associated tilting object $T$ is given by

        \[T=(fB_{d}'\, G\lotimes{B'_{d}\, G}{\ T_{d}\, G})\lotimes{B_{d}\, G}{\ B_{d}\, G f}\]

        \noindent It remains to prove that $T\simeq \mu^{+}(B\,;\,e_{H''}B)$. Similarly to the proof of Theorem \ref{thm:generalized Kauer moves with multiplicity}, it suffices to show that $\alpha_{[h_{0}]}\lotimes{B_{d}}{B_{d}\, G}$ is isomorphic to the left minimal $\mathrm{add}(f_{H''_{d}}B_{d}\, G)$-approximation of $e_{[h_{0}]}B_{d}\, G$ where $f_{H''_{d}}$ denotes the idempotent in $B_{d}\, G$ given by the sum of the $e_{[h_{0}]}\otimes 1_{G}$ for all $[h]\in (H\backslash H')/\iota$. Thanks to Proposition \ref{prop:covering and skew group algebra}, we know that 
        
        \[\alpha_{[h_{0}]}\lotimes{B_{d}}{B_{d}\, G}:e_{[h_{0}]}B_{d}\,G\longrightarrow e_{[\sigma_{d}^{r(h_{0})+1}h_{0}]}B_{d}\, G \oplus e_{[\sigma_{d}^{r(\iota_{d}h_{0})+1}\iota_{d}h_{0}]}B_{d}\, G\]
        
        \noindent is the left minimal $\mathrm{add}(e_{H''_{d}}B_{d}\, G)$-approximation of $e_{[h_{0}]}B_{d}\, G$. By definition of $d$, we have  

        \begin{equation*}
            \sigma_{d}^{r(\iota_{d}h_{0})+1}\iota_{d}h_{0}=\left\{\begin{aligned}&(\sigma^{r(\iota_{d}h_{0})+1}\iota h)_{0} &&\mbox{if $h\in H_{\circ}$} \\
            &(\sigma^{r(\iota_{d}h_{0})+1} h)_{1} &&\mbox{if $h\in H_{\times}$}
            \end{aligned}\right.
        \end{equation*}

        \noindent Moreover, we have the following isomorphism in $\per{B_{d}\, G}$

        \[e_{[(\sigma^{r(\iota_{d}h_{0})+1}h)_{1}]}B_{d}\, G\overset{\sim}{\longrightarrow} e_{[(\sigma^{r(\iota_{d}h_{0})+1}h)_{0}]}B_{d}\, G\]

        \noindent Indeed, if $\sigma^{r(\iota_{d}h_{0})+1}h\in H_{\times}$, then $[(\sigma^{r(\iota_{d}h_{0})+1}h)_{0}]=[(\sigma^{r(\iota_{d}h_{0})+1}h)_{1}]$ and the isomorphism is just the identity. On the other hand, if $\sigma^{r(\iota_{d}h_{0})+1}h\in H_{\circ}$, the isomorphism is given by the multiplication on the left by $(1\otimes g)$ where $g$ denotes the generator of $G$. Hence, in each case $\alpha_{[h_{0}]}\lotimes{B_{d}}{B_{d}\, G}$ can be written as

        \[\alpha_{[h_{0}]}\lotimes{B_{d}}{B_{d}\, G}:e_{[h_{0}]}B_{d}\,G\longrightarrow e_{[(\sigma^{r(h_{0})+1}h)_{0}]}B_{d}\, G \oplus e_{[(\sigma^{r(\iota_{d}h_{0})+1}\iota h)_{0}]}B_{d}\, G\]

        \noindent Thus, $\alpha_{[h_{0}]}\lotimes{B_{d}}{B_{d}\, G}$ is isomorphic to the left minimal $\mathrm{add}(f_{H''_{d}}B_{d}\, G)$-approximation of $e_{[h_{0}]}B_{d}\, G$. This concludes the proof. 
    \end{proof}

\phantomsection

\setlength{\bibitemsep}{0pt}
\setlength{\biblabelsep}{7pt}
\defbibnote{nom}{\\ \small{\scshape{Université Grenoble Alpes, CNRS, Institut Fourier, 38610 Gières}} \\ \textit{E-mail address :} \href{mailto:valentine.soto@univ-grenoble-alpes.fr}{valentine.soto@univ-grenoble-alpes.fr}}
\printbibliography[postnote=nom, heading=bibintoc,title={References}]

\end{document}